\documentclass[11pt]{amsart}
\usepackage{epsfig,amsmath,amssymb, xspace, graphicx, url, amscd, euscript, mathrsfs,color}
\input xy
\xyoption{all}
\usepackage{tikz}
\usetikzlibrary{intersections}
\usetikzlibrary{arrows}
\usetikzlibrary{decorations.pathmorphing}

\newtheorem{theorem}[subsection]{Theorem}
\newtheorem{proposition}[subsection]{Proposition}
\newtheorem{proposition-definition}[subsection]
{Proposition-Definition}
\newtheorem{corollary}[subsection]{Corollary}
\newtheorem{lemma}[subsection]{Lemma}

\theoremstyle{definition}
\newtheorem{definition}[subsection]{Definition}
\newtheorem{notation}[subsection]{Notation}
\newtheorem{example}[subsection]{Example}

\newtheorem{remark}[subsection]{Remark}
\newtheorem{terminology}[subsection]{Terminology}
\newtheorem{question}[subsection]{Question}
\newtheorem*{ack}{Acknowledgements}

\newcommand{\scr}[1]{\mathbf{\EuScript{#1}}}

\newcommand\cB{\mathcal{B}}
\newcommand\cC{\mathcal{C}}
\newcommand\cD{\mathcal{D}}
\newcommand\cE{\mathcal{E}}
\newcommand\cF{\mathcal{F}}
\newcommand\cG{\mathcal{G}}

\newcommand\cK{\mathcal{K}}
\newcommand\cL{\mathcal{L}}
\newcommand\cM{\mathcal{M}}

\newcommand\cO{\mathcal{O}}
\newcommand\cP{\mathcal{P}}
\newcommand\cQ{\mathcal{Q}}
\newcommand\cR{\mathcal{R}}

\newcommand\cT{\mathcal{T}}

\newcommand\cX{\mathcal{X}}
\newcommand\cY{\mathcal{Y}}
\newcommand\cZ{\mathcal{Z}}

\renewcommand\frm{\mathfrak{m}}

\newcommand{\mc}{\mathcal}

\title{Moduli of elliptic curves via twisted stable maps}
\author[A. Niles]{Andrew Niles}
\address{Department of Mathematics\\
University of California\\
Berkeley, CA 94720\\
USA}
\email{andrew@math.berkeley.edu}
\date{\today}

\begin{document}

\begin{abstract}
Abramovich, Corti and Vistoli have studied modular compactifications of stacks of curves equipped with abelian level structures arising as substacks of the stack of twisted stable maps into the classifying stack of a finite group, provided the order of the group is invertible on the base scheme. Recently Abramovich, Olsson and Vistoli extended the notion of twisted stable maps to allow arbitrary base schemes, where the target is a tame stack, not necessarily Deligne-Mumford. We use this to extend the results of Abramovich, Corti and Vistoli to the case of elliptic curves with level structures over arbitrary base schemes; we prove that we recover the compactified Katz-Mazur regular models, with a natural moduli interpretation in terms of level structures on Picard schemes of twisted curves. Additionally, we study the interactions of the different such moduli stacks contained in a stack of twisted stable maps in characteristics dividing the level.
\end{abstract}

\maketitle
\tableofcontents

\section{Introduction}
Abramovich and Vistoli introduced in \cite{AV} the stack $\cK_{g,n}(\cX)$ of \textit{$n$-pointed genus $g$ twisted stable maps} into $\cX$, where $\cX$ is a proper tame Deligne-Mumford stack over a base scheme $S$ with a projective coarse moduli space $X/S$. They prove that $\cK_{g,n}(\cX)$ is an algebraic stack, proper over the Kontsevich stack $\overline{\cM}_{g,n}(X)$ of stable maps into $X$, giving an appropriate analogue of the usual Kontsevich stack of stable maps when the target is allowed to be a tame stack instead of merely a scheme or algebraic space. In particular, taking $\cX = \cB G$ for a finite group $G$ naturally yields a modular compactification of the stack of $n$-pointed genus $g$ curves with certain level structure, where the level structure on a curve corresponds to giving a $G$-torsor over the curve; this is studied extensively in \cite{ACV}. This of course differs from the approach in \cite{AR}, where the coverings of curves are assumed themselves to be stable curves; here the coverings are generally not stable curves. In \cite{P} it is shown that in good characteristics we recover the usual compactification of $\scr Y(N)$ in $\cK_{1,1}(\cB(\mathbb{Z}/(N))^2)$.

However, the algebraic stack $\cB G$ is not tame in characteristics dividing $|G|$, so a priori the results of \cite{ACV} only hold over $\mathbb{Z}[1/|G|]$. The notion of a tame stack is generalized in \cite{AOV1} to algebraic stacks that are not necessarily Deligne-Mumford, and in the following paper \cite{AOV2} the stack $\cK_{g,n}(\cX)$ of twisted stable maps into certain proper tame algebraic stacks is introduced and shown to be a proper algebraic stack. This naturally leads us to attempt replacing the finite group $G$ with a group scheme $\cG$, agreeing with $G$ over $\mathbb{Z}[1/|G|]$ but such that $\cB\cG$ is a tame algebraic stack over $\mathrm{Spec}(\mathbb{Z})$; see for example \cite{Abr}. In particular, the group scheme $\mu_N$ is (noncanonically) isomorphic to $\mathbb{Z}/(N)$ over $\mathbb{Z}[1/N,\zeta_N]$, and the classifying stack $\cB\mu_N$ (unlike $\cB(\mathbb{Z}/(N))$) is a tame stack over $\mathrm{Spec}(\mathbb{Z})$. 

The purpose of this paper is to record how the results of \cite{ACV} extend to moduli of elliptic curves with level structure over arbitrary base schemes, using the group scheme $\mu_N$ in place of $\mathbb{Z}/(N)$. For example, the moduli stack $\scr Y_1(N)$ arises as an open substack in the rigidification $\overline{\cK}_{1,1}(\cB\mu_N)$ (defined below) of $\cK_{1,1}(\cB\mu_N)$. Explicitly, given an elliptic curve $E/S$ and a $[\Gamma_1(N)]$-structure $P$ on $E$, via the canonical group scheme isomorphism $E \cong \mathrm{Pic}^0_{E/S}$ over $S$, we may view $P$ as a ``point of exact order $N$'' on $\mathrm{Pic}^0_{E/S}$, determining a group scheme homomorphism $\phi: \mathbb{Z}/(N) \rightarrow \mathrm{Pic}^0_{E/S}$. The stack $\overline{\cK}_{1,1}(\cB\mu_N)$ will be seen to classify certain pairs $(\cC,\phi)$ where $\cC/S$ is a $1$-marked genus $1$ twisted stable curve and $\phi:\mathbb{Z}/(N) \rightarrow \mathrm{Pic}^0_{\cC/S}$ is a group scheme homomorphism, so this construction defines the immersion $\scr Y_1(N) \rightarrow \overline{\cK}_{1,1}(\cB\mu_N)$. We will define the notion of a $[\Gamma_1(N)]$-structure on a $1$-marked genus $1$ twisted stable curve, and write $\scr X^{\mathrm{tw}}_1(N)$ for the stack classifying such structures. Our first main result is:
\begin{theorem}\label{firstgamma1}
Let $S$ be a scheme and $\scr X^{\mathrm{tw}}_1(N)$ the stack over $S$ classifying $[\Gamma_1(N)]$-structures on $1$-marked genus $1$ twisted stable curves with non-stacky marking. Then $\scr X^{\mathrm{tw}}_1(N)$ is a closed substack of $\overline{\cK}_{1,1}(\cB \mu_N)$, which contains $\scr Y_1(N)$ as an open dense substack.
\end{theorem}
This will be proved as Theorem \ref{gamma1} below. The main point is to verify the valuative criterion for properness. We accomplish this by using properness of $\scr X_1(N)$ and the $e_N$-pairing on generalized elliptic curves. To complete a family of twisted curves with level structure when the generic fiber is smooth, we first complete it to a family of generalized elliptic curves with level structure. Then we use a quotient construction involving the $e_N$-pairing to produce a completed family of twisted curves with level structure, modifying our strategy in characteristics dividing $N$ by exploiting the relationship between cyclotomic torsors and line bundles (Lemma \ref{bigpicard}). The same techniques immediately give a corresponding result for $\scr Y(N)$, which is an open substack of $\overline{\cK}_{1,1}(\cB\mu_N^2)$; its closure $\scr X^{\mathrm{tw}}(N)$ classifies $[\Gamma(N)]$-structures on $1$-marked genus $1$ twisted stable curves, as we will show in Theorem \ref{gamma} below.

Unsurprisingly, it turns out that these closures are isomorphic as algebraic stacks to the stacks $\scr X_1(N)$ and $\scr X(N)$ classifying $[\Gamma_1(N)]$-structures and $[\Gamma(N)]$-structures on generalized elliptic curves, as studied in \cite{C}:
\begin{theorem}\label{firstclassical}
Over the base $S = \mathrm{Spec}(\mathbb{Z})$, there is a canonical isomorphism of algebraic stacks $\scr X_1^{\mathrm{tw}}(N) \cong \scr X_1(N)$ extending the identity map on $\scr Y_1(N)$, and a canonical isomorphism of algebraic stacks $\scr X^{\mathrm{tw}}(N) \cong \scr X(N)$ extending the identity map on $\scr Y(N)$.
\end{theorem}
We will prove this as Theorem \ref{classical} below. We first verify that $\scr X_1^{\mathrm{tw}}(N)$ and $\scr X^{\mathrm{tw}}(N)$ are finite over $\overline{\cM}_{1,1}$. Some commutative algebra then tells us they are Cohen-Macaulay over $\mathrm{Spec}(\mathbb{Z})$, at which point we may proceed as in \cite[\S4.1]{C} to identify these stacks with the normalizations of $\overline{\cM}_{1,1}$ in $\scr X_1(N)|_{\mathbb{Z}[1/N]}$ and $\scr X(N)|_{\mathbb{Z}[1/N]}$. These isomorphisms have a natural moduli interpretation, as discussed in Corollaries \ref{modular1} and \ref{modular2}.

The techniques of this paper also yield new moduli interpretations of various moduli stacks of elliptic curves that are not (apparently) contained in a stack of twisted stable maps. These results may be known to experts, but are not recorded in the literature; for completeness we include a careful proof of the modular interpretation of the closure of $\scr Y_1(N)$ in $\overline{\cK}_{1,1}(\cB \mu_N)$. We also study how the different moduli stacks of elliptic curves contained in $\overline{\cK}_{1,1}(\cB\mu_N)$ and $\overline{\cK}_{1,1}(\cB\mu_N^2)$ interact in characteristics dividing $N$; this easily generalizes an example in \cite{AOV2} but otherwise does not appear in the literature. 

In the appendix we recall an example of \cite{CN} which implies that the techniques of this paper do not generalize nicely to curves of higher genus: using the Katz-Mazur notion of a ``full set of sections'' we can define a stack over $\mathrm{Spec}(\mathbb{Z})$ classifying genus $g$ curves with full level $N$ structures on their Jacobians, but this stack is not even flat over $\mathrm{Spec}( \mathbb{Z})$. The corresponding stack over $\mathbb{Z}[1/N]$ is well-behaved, and arises as an open substack of $\overline{\cK}_{g,0}(\cB\mu_N^{2g})$, but no moduli interpretation for its closure in characteristics dividing $N$ appears to be known.

\begin{ack}
The author is grateful to his advisor, Martin Olsson, for suggesting this research topic and for many helpful conversations, and also to the reviewer for several helpful comments and corrections. This research was supported by an NSF Graduate Research Fellowship.
\end{ack}

\section{Review of moduli of generalized elliptic curves}\label{section2}

\subsection*{Drinfeld level structures on generalized elliptic curves}
For the convenience of the reader, we recall the definitions and results from the theory of generalized elliptic curves that we require in this paper.

\begin{definition}
A \textit{Deligne-Rapoport (DR) semistable curve of genus $1$} over a scheme $S$ is a proper, flat, separated, finitely presented morphism of schemes $f:C \rightarrow S$, all of whose geometric fibers are non-empty, connected semistable curves with trivial dualizing sheaves.
\end{definition}

By \cite[\S$\textrm{II.1}$]{DR}, an equivalent definition for $f:C \rightarrow S$ to be a DR semistable curve of genus $1$ is that $f$ is a proper flat morphism of finite presentation and relative dimension $1$, such that every geometric fiber is either a smooth connected genus $1$ curve or a N\'eron polygon. Recall (loc. cit.) that over a base scheme $S$, the \textit{standard N\'eron $N$-gon} $C_N/S$ (for any $N\geq 1$) is obtained from $\widetilde{C}_N := \mathbb{P}^1_S \times \mathbb{Z}/(N)$ by ``gluing'' the section $0$ in the $i^{\mathrm{th}}$ copy of $\mathbb{P}^1_S$ to the section $\infty$ in the $(i+1)^{\mathrm{th}}$ copy of $\mathbb{P}^1_S$:

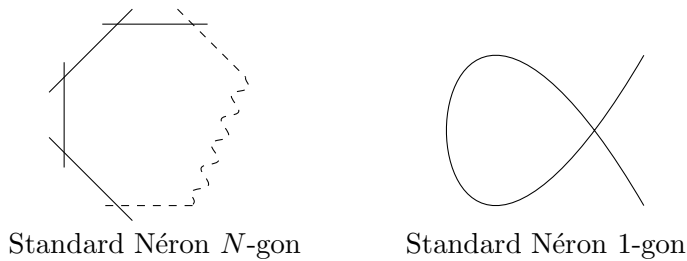
\begin{center}
\begin{tikzpicture}
\clip (-1.5,-1) rectangle (7.6,3);
\draw (0.2,-0.2) -- (-0.907,0.907);
\draw (-0.707,0.507) -- (-0.707,1.907);
\draw (-0.907,1.507) -- (0.2,2.614);
\draw (-0.2,2.414) -- (1.2,2.414);
\begin{scope}[dashed]
\draw (0.8,2.614) -- (1.707,1.707);
\draw (1,0) -- (-0.2,0);
\draw [decorate,decoration=snake] (1.707,1.707) -- (1,0);
\end{scope}
\node at (0.5,-0.2) [inner sep=0pt,label=270:Standard N\'eron $N$-gon] {};
\draw [name path=curve] (7,2) .. controls (3.5,-4) and (3.5,6) .. (7,0);
\node at (5.7,-0.2) [inner sep=0pt,label=270:Standard N\'eron $1$-gon] {};
\end{tikzpicture}
\end{center}

The natural multiplication action of $\mathbb{G}_m$ on $\mathbb{P}^1_S$, together with the action of $\mathbb{Z}/(N)$ on itself via the group law, determines an action of the group scheme $\mathbb{G}_m \times \mathbb{Z}/(N)$ on $\mathbb{P}^1_S \times \mathbb{Z}/(N)$, descending uniquely to an action of $\mathbb{G}_m \times \mathbb{Z}/(N) = C_N^{\mathrm{sm}}$ on $C_N$ (\cite[II.1.9]{DR}).

\begin{definition}
A \textit{generalized elliptic curve} over a scheme $S$ is a DR semistable curve $E/S$ of genus $1$, equipped with a morphism $E^{\mathrm{sm}} \times E \rightarrow E$ and a section $0_E \in E^{\mathrm{sm}}(S)$ such that the restriction $E^{\mathrm{sm}} \times E^{\mathrm{sm}} \rightarrow E^{\mathrm{sm}}$ makes $E^{\mathrm{sm}}$ a commutative group scheme over $S$ with identity $0_E$, and such that on any singular geometric fiber $E_{\overline{s}}$, the translation action by a rational point on $E^{\mathrm{sm}}_{\overline{s}}$ acts by a rotation on the graph $\Gamma(E_{\overline{s}})$ (\cite[I.3.5]{DR}) of the irreducible components of $E_{\overline{s}}$.
\end{definition}

By \cite[II.1.15]{DR}, over an algebraically closed field a generalized elliptic curve is either a smooth elliptic curve or a N\'eron $N$-gon (for some $N\geq 1$) with the action described above. In fact this result says more: for any generalized elliptic curve $E/S$, there is a locally finite family $(S_N)_{N\geq 1}$ of closed subschemes of $S$ such that $\cup S_N$ is the non-smooth locus in $S$ of $E \rightarrow S$, and $E\times_S S_N$ is fppf-locally on $S_N$ isomorphic to the standard N\'eron $N$-gon over $S_N$.

Recall that for an $S$-scheme $X$, a \textit{relative effective Cartier divisor} in $X$ over $S$ is an effective Cartier divisor in $X$ which is flat over $S$.

\begin{definition}\label{gamma1elliptic}
Let $E/S$ be a generalized elliptic curve. A \textit{$[\Gamma_1(N)]$-structure} on $E$ is a section $P \in E^{\mathrm{sm}}(S)$ such that:
\begin{itemize}
  \item $N\cdot P = 0_E$; 
  \item the relative effective Cartier divisor 
  \begin{displaymath}
  D := \sum_{a \in \mathbb{Z}/(N)} [a\cdot P]
  \end{displaymath}
  in $E^{\mathrm{sm}}$ is a subgroup scheme; and 
  \item for every geometric point $\overline{p} \rightarrow S$, $D_{\overline{p}}$ meets every irreducible component of $E_{\overline{p}}$.
\end{itemize}
We write $\scr X_1(N)$ for the stack over $\mathrm{Spec}(\mathbb{Z})$ associating to a scheme $S$ the groupoid of pairs $(E,P)$, where $E/S$ is a generalized elliptic curve and $P$ is a $[\Gamma_1(N)]$-structure on $E$. We write $\scr Y_1(N)$ for the substack classifying such pairs where $E/S$ is a smooth elliptic curve.
\end{definition}

\begin{definition}\label{gammaelliptic}
Let $E/S$ be a generalized elliptic curve. A \textit{$[\Gamma(N)]$-structure} on $E$ is an ordered pair $(P,Q)$ of sections in $E^{\mathrm{sm}}[N](S)$ such that:
\begin{itemize}
  \item the relative effective Cartier divisor 
  \begin{displaymath}
  D := \sum_{a,b \in \mathbb{Z}/(N)} [a\cdot P + b \cdot Q]
  \end{displaymath}
  in $E^{\mathrm{sm}}$ is an $N$-torsion subgroup scheme, hence $D = E^{\mathrm{sm}}[N]$; and 
  \item for every geometric point $\overline{p} \rightarrow S$, $D_{\overline{p}}$ meets every irreducible component of $E_{\overline{p}}$.
\end{itemize}
We write $\scr X(N)$ for the stack over $\mathrm{Spec}(\mathbb{Z})$  associating to a scheme $S$ the groupoid of tuples $(E,(P,Q))$, where $E/S$ is a generalized elliptic curve and $(P,Q)$ is a $[\Gamma(N)]$-structure on $E$. We write $\scr Y(N)$ for the substack classifying such tuples where $E/S$ is a smooth elliptic curve.
\end{definition}

\begin{definition}\label{gstructure}
Let $E/S$ be a generalized elliptic curve, and let $G$ be a $2$-generated finite abelian group, say $G \cong \mathbb{Z}/(n_1) \times \mathbb{Z}/(n_2)$, $n_2|n_1$. A \textit{$G$-structure} on $E$ is a homomorphism $\phi: G \rightarrow E^{\mathrm{sm}}$ of group schemes over $S$ such that:
\begin{itemize}
  \item the relative effective Cartier divisor 
  \begin{displaymath}
  D := \sum_{a \in G} [\phi(a)]
  \end{displaymath}
  in $E^{\mathrm{sm}}$ is an $n_1$-torsion subgroup scheme; and 
  \item for every geometric point $\overline{p} \rightarrow S$, $D_{\overline{p}}$ meets every irreducible component of $E_{\overline{p}}$.
\end{itemize}
\end{definition}

\begin{theorem}[$\textrm{\cite[3.1.7, 3.2.7, 3.3.1, 4.1.1]{C}}$]
$\scr X_1(N)$ and $\scr X(N)$ are regular Deligne-Mumford stacks, proper and flat over $\mathrm{Spec}(\mathbb{Z})$ of pure relative dimension $1$.
\end{theorem}

In particular, it follows (cf. \cite[4.1.5]{C}) that $\scr X_1(N)$ (resp. $\scr X(N)$) is canonically identified with the normalization of $\overline{\cM}_{1,1}$ in the normal Deligne-Mumford stack $\scr X_1(N)|_{\mathbb{Z}[1/N]}$ (resp. $\scr X(N)|_{\mathbb{Z}[1/N]}$), as in \cite[\S8.6]{KM1} and \cite[\S$\textrm{VII.2}$]{DR}.

\subsection*{Reductions mod $p$ of the moduli stacks}

It will be useful for us to have a description of the ``reduction mod $p$'' of the stacks $\scr X_1(N)$ and $\scr X(N)$ for primes $p$ dividing $N$. These reductions are described using the ``crossings theorem'' (\cite[13.1.3]{KM1}) which we now recall.

Work over a fixed field $k$. Let $Y/k$ be a smooth curve, and let $X \rightarrow Y$ be finite and flat. Suppose there is a non-empty finite set of $k$-rational points of $Y$ (which we will call the \textit{supersingular points} of $Y$) such that for each supersingular point $y_0$, $X_{y_0}$ is a single $k$-rational point, and $\widehat{\cO}_{X,x_0} \cong k[\![ x,y]\!]/(f)$ for some $f$ (depending on $y_0$).

Also suppose we have a finite collection of closed immersions $\{Z_i \hookrightarrow X\}_{i = 1}^n$, where each $Z_i$ is finite and flat over $Y$, with $Z_i^{\mathrm{red}}$ a smooth curve over $k$, such that the morphism $\sqcup Z_i \rightarrow X$ is an isomorphism over the non-supersingular locus of $Y$, and such that for each $i$ and each supersingular point $y_0 \in Y$, $(Z_i)_{y_0}$ is a single $k$-rational point.

\begin{theorem}[Crossings Theorem, $\textrm{\cite[13.1.3]{KM1}}$]\label{crossings}
Under the above assumptions, let $y_0 \in Y$ be a supersingular point and $x_0 = X_{y_0}$. Then 
\begin{displaymath}
\widehat{\mathcal{O}}_{X,x_0} \cong k[\![x,y]\!]/\prod_{i = 1}^m f_i^{e_i},
\end{displaymath}
where each $f_i$ is irreducible in $k[\![x,y]\!]$, each $f_i$ and $f_j$ $(i\neq j)$ are distinct in $k[\![x,y]\!]$ modulo multiplication by units, and for $z_{i,0} := (Z_i)_{y_0}$ we have 
\begin{displaymath}
\widehat{\mathcal{O}}_{Z_i,z_{i,0}} \cong k[\![x,y]\!]/(f_i^{e_i}).
\end{displaymath}
If $Y$ is (geometrically) connected then each $Z_i$ is (geometrically) connected, in which case $\{Z_i\}_{i = 1}^n$ is the set of irreducible components of $X$.
\end{theorem}

\begin{definition}
Under the above assumptions, we say that $X$ is the \textit{disjoint union with crossings at the supersingular points} of the closed subschemes $\{Z_i\}_{i = 1}^n$. 

If $\cX \rightarrow \cY$ is a finite (hence representable) flat morphism of algebraic stacks over $k$, with $\cY$ Deligne-Mumford (hence $\cX$ is also Deligne-Mumford), and $\{\cZ_i \hookrightarrow \cX\}_{i = 1}^n$ is a finite collection of closed immersions of algebraic stacks, such that for some \'etale surjection $Y' \rightarrow \cY$ with $Y'$ a scheme, the algebraic spaces $Y'$, $\cX \times_{\cY} Y'$, $\{\cZ_i \times_{\cY} Y'\}$ are schemes satisfying the assumptions of the Crossings Theorem, we say that $\cX$ is the \textit{disjoint union with crossings at the supersingular points} of the closed substacks $\{\cZ_i\}$.
\end{definition}

\begin{remark}
We will solely be applying the above theorem in the case $\cY = \overline{\cM}_{1,1}$, with $\cX$ an algebraic stack known to be quasi-finite and proper over $\overline{\cM}_{1,1}$. $\overline{\cM}_{1,1}$ is a Deligne-Mumford stack with separated diagonal, so once we know that $\cX \rightarrow \overline{\cM}_{1,1}$ is representable, it follows from \cite[II.6.16]{K} that the algebraic spaces $\cX \times_{\overline{\cM}_{1,1}} Y'$ and $\{\cZ_i \times_{\overline{\cM}_{1,1}} Y'\}$ are schemes.
\end{remark}

\begin{definition}\label{abcyclic}
Let $S \in \mathrm{Sch}/\mathbb{F}_p$. Let $E/S$ be an elliptic curve, with relative Frobenius $F: E \rightarrow E^{(p)}$. Let $G \subset E$ be a finite, locally free $S$-subgroup scheme of rank $p^n$, $n\geq 1$. For integers $a,b\geq 0$ with $a+b=n$, we say that $G$ is an \textit{$(a,b)$-subgroup} if $\mathrm{ker}(F^a) \subset G$, and if in the resulting factorization of $E \rightarrow E/G$ as 
\begin{displaymath}
E \stackrel{F^a}{\rightarrow} E^{(p^a)} \stackrel{\pi}{\rightarrow} E/G
\end{displaymath}
we have $\mathrm{ker}(\widehat{\pi}) = \mathrm{ker}(F^b_{E/G})$ (where $\widehat{\pi}$ denotes the dual isogeny and $F_{E/G}: E/G \rightarrow (E/G)^{(p)}$ is the relative Frobenius). In particular note that $E^{(p^a)} \cong (E/G)^{(p^b)}$.

We say that $G$ is an \textit{$(a,b)$-cyclic subgroup} if it is an $(a,b)$-subgroup, and either $a=0$, $b=0$, or there exists a closed subscheme $Z \subset S$ defined by a sheaf of ideals $\mathcal{I} \subset \mathcal{O}_S$ with $\mathcal{I}^{p-1} = 0$, such that the isomorphism $E^{(p^a)}|_Z \cong (E/G)^{(p^b)}|_Z$ is induced by an isomorphism $E^{(p^{a-1})}|_Z \cong (E/G)^{(p^{b-1})}|_Z$.

Finally, a \textit{$[\Gamma_1(p^n)]$-$(a,b)$-cyclic structure} on $E$ is a $[\Gamma_1(p^n)]$-structure $P \in E[p^n](S)$ such that the $S$-subgroup scheme 
\begin{displaymath}
D := \sum_{m \in \mathbb{Z}/(p^n)} [m\cdot P]
\end{displaymath}
in $E$ is an $(a,b)$-cyclic subgroup of $E$. We write $\scr Y_1(p^n)^{(a,b)}_S$ for the substack of $\scr Y_1(p^n)_S$ associating to a scheme $T/S$ the groupoid of pairs $(E,P)$, where $E/T$ is an elliptic curve and $P$ is a $[\Gamma_1(p^n)]$-$(a,b)$-cyclic structure on $E$.
\end{definition}

If $E$ is an ordinary elliptic curve over a scheme $S$ of characteristic $p$, it is considerably easier to describe what is meant by a $[\Gamma_1(p^n)]$-$(a,b)$-cyclic structure on $E$. Namely, it is just a $[\Gamma_1(p^n)]$-structure $P$ on $E$ such that the relative effective Cartier divisor 
\begin{displaymath}
D_b := \sum_{m = 1}^{p^b} [m\cdot P]
\end{displaymath}
is a subgroup scheme of $E$ which is \'etale over $S$. Note that for an arbitrary $[\Gamma_1(p^n)]$-structure $P$ on $E$, $D_b$ will not generally be a subgroup scheme of $E$ if $b<n$, and even if it is a subgroup scheme it might not be \'etale over $S$. Being a $[\Gamma_1(p^n)]$-$(a,b)$-cyclic structure on $E$ (for $E$ ordinary) means that for any geometric fiber $E_{\overline{s}}$ (so $E_{\overline{s}}[p^n] \cong \mu_{p^n} \times \mathbb{Z}/(p^n)$ over $k(\overline{s})$) $P_{\overline{s}}$ has exact order $p^b$ as an element of the group $E_{\overline{s}}[p^n](k(\overline{s})) \cong \mathbb{Z}/(p^n)$.

The following elementary result will be required when we study the interactions in characteristic $p$ of different moduli stacks of elliptic curves contained in a moduli stack of twisted stable maps. We include a proof for lack of a suitable reference.

\begin{lemma}\label{gamma1comps}
Let $E/S/\mathbb{F}_p$ be an elliptic curve, and let $P$ be a $[\Gamma_1(p^n)]$-$(a,b)$-cyclic structure on $E$. Then $P$ is also a $[\Gamma_1(p^{n+1})]$-structure on $E$, and is $[\Gamma_1(p^{n+1})]$-$(a+1,b)$-cyclic.
\end{lemma}
\begin{proof}
Consider the relative effective Cartier divisor 
\begin{displaymath}
G := \sum_{m \in \mathbb{Z}/(p^n)} [m\cdot P]
\end{displaymath}
in $E$. This is an $S$-subgroup scheme containing $\mathrm{ker}(F^a)$, and in the resulting factorization of the quotient map $E \rightarrow E/G$ as 
\begin{displaymath}
E \stackrel{F^a}{\rightarrow} E^{(p^a)} \stackrel{\pi}{\rightarrow} E/G
\end{displaymath}
we have $\mathrm{ker}(\widehat{\pi}) = \mathrm{ker}(F^b_{E/G})$. Consider also the relative effective Cartier divisor 
\begin{displaymath}
G' := \sum_{m \in \mathbb{Z}/(p^{n+1})} [m\cdot P]
\end{displaymath}
in $E$. $G$ is a subgroup of $G'$, and the image of $G'$ in $E/G$ is the relative effective Cartier divisor 
\begin{displaymath}
\sum_{m \in \mathbb{Z}/(p)} [m \cdot 0_{E/G}].
\end{displaymath}
This is simply the kernel of $F_{E/G}: E/G \rightarrow (E/G)^{(p)}$, so we may identify $E/G'$ with $(E/G)^{(p)}$ and the quotient map $E/G \rightarrow E/G'$ with the relative Frobenius $F_{E/G}: E/G \rightarrow (E/G)^{(p)}$. In particular, the quotient map $E \rightarrow E/G'$ is a cyclic $p^{n+1}$-isogeny of elliptic curves with kernel generated by $P$, so we may already conclude that $P$ is a $[\Gamma_1(p^{n+1})]$-structure on $E$.

Now, since $G' = p\cdot G$ as Cartier divisors, if $\mathrm{ker}(F^a) \subset G$ then $\mathrm{ker}(F^{a+1}) \subset G'$. Factor the quotient map $E \rightarrow E/G' = (E/G)^{(p)}$ as 
\begin{displaymath}
E \stackrel{F^{a+1}}{\rightarrow} E^{(p^{a+1})} \stackrel{\pi'}{\rightarrow} (E/G)^{(p)}.
\end{displaymath}
We have a diagram 
\begin{align*}
\xymatrix{ 
  E \ar[r]^{F^a} \ar[dr]_{F^{a+1}} & E^{(p^a)} \ar[r]^\pi \ar[d]^F & E/G \ar[r]^{F_{E/G}} & (E/G)^{(p)} \\
  & E^{(p^{a+1})} \ar[urr]_{\pi'} }
\end{align*}
where the composite of the top arrows $E \rightarrow (E/G)^{(p)}$ is the natural quotient map $E \rightarrow E/G'$. The outer arrows and the left-hand triangle commute, hence the right-hand triangle commutes as well. Now consider the following diagram:
\begin{align*}
\xymatrix{ 
  E^{(p^a)} \ar[d]_F \ar[r]^{\pi} & E/G \ar[d] \ar[r]^{\widehat{\pi}} & E^{(p^a)} \ar[d]^F \\
  E^{(p^{a+1})} \ar[r]_{\pi'} & (E/G)^{(p)} \ar[r]_{\widehat{\pi}'} & E^{(p^{a+1})} }
\end{align*}
The outer rectangle commutes since the horizontal composites are the isogenies $[\mathrm{deg}(\pi)]$ and $[\mathrm{deg}(\pi')]$ (and $\mathrm{deg}(\pi) = \mathrm{deg}(\pi')$), and we have shown that the left-hand square commutes, so the right-hand square commutes as well. Since $\mathrm{ker}(\widehat{\pi}) = \mathrm{ker}(F^b_{E/G})$, we conclude that $\mathrm{ker}(\widehat{\pi}') = \mathrm{ker}(F^b_{(E/G)^{(p)}})$.

Finally, suppose $a,b \geq 1$. Since $E^{(p^a)} \cong (E/G)^{(p^b)}$ we also have 
\begin{displaymath}
E^{(p^{a+1})} \cong (E/G)^{(p^{b+1})} = ((E/G)^{(p)})^{(p^b)}, 
\end{displaymath}
and this is already induced by the isomorphism $E^{(p^a)} \cong (E/G)^{(p^b)} = ((E/G)^{(p)})^{(p^{b-1})}$. Therefore $P$ is a $[\Gamma_1(p^{n+1})]$-$(a+1,b)$-cyclic structure on $E$.
\end{proof}

\begin{remark}
The reason the above argument fails when $p$ is invertible on the base scheme $S$ is that, preserving the notation of the above proof, the image of the relative Cartier divisor $G'$ in $E/G$ is not a subgroup of $E/G$ unless $S$ is an $\mathbb{F}_p$-scheme. Indeed, the image of $G'$ in $E/G$ is the relative effective Cartier divisor $p \cdot [0_{E/G}]$, which is finite flat of rank $p$ over $S$. If $p$ is invertible on $S$, any finite flat commutative group scheme of rank $p$ over $S$ is \'etale (cf. \cite[Cor. 4.3]{Sha}), but $p \cdot [0_{E/G}]$ is obviously not \'etale.
\end{remark}

It is straightforward to extend the definition of a $[\Gamma_1(p^n)]$-$(a,b)$-cyclic structure, and the result of the above lemma, to the case of generalized elliptic curves. Let $E/S/\mathbb{F}_p$ be a generalized elliptic curve; let $S_1 \subset S$ be the open locus where the morphism $E \rightarrow S$ is smooth, and let $S_2 \subset S$ be the complement of the closed locus in $S$ where the geometric fibers are supersingular elliptic curves. We say that a $[\Gamma_1(p^n)]$-structure $P$ on $E/S$ is a \textit{$[\Gamma_1(p^n)]$-$(a,b)$-cyclic structure} if:
\begin{itemize}
  \item $P|_{S_1}$ is a $[\Gamma_1(p^n)]$-$(a,b)$-cyclic structure on the elliptic curve $E_{S_1}/S_1$, and 
  \item the relative effective Cartier divisor 
  \begin{displaymath}
  D_b := \sum_{m = 1}^{p^b} [m\cdot P|_{S_2}]
  \end{displaymath}
  in $E^{\mathrm{sm}}|_{S_2}$ is a subgroup scheme of $E^{\mathrm{sm}}|_{S_2}$ which is \'etale over $S_2$.
\end{itemize}
We write $\scr X_1(p^n)^{(a,b)}_S$ for the substack of $\scr X_1(p^n)_S$ associating to a scheme $T/S$ the groupoid of pairs $(E,P)$, where $E/T$ is a generalized elliptic curve and $P$ is a $[\Gamma_1(p^n)]$-$(a,b)$-cyclic structure on $E$.

\begin{theorem}[$\textrm{\cite[13.5.4]{KM1}}$]\label{y1-char-p}
Let $k$ be a perfect field of characteristic $p$. $\scr Y_1(p^n)_k$ (resp. $\scr X_1(p^n)_k$) is the disjoint union, with crossings at the supersingular points, of the $n+1$ substacks $\scr Y_1(p^n)^{(a,n-a)}_k$ (resp. $\scr X_1(p^n)^{(a,n-a)}_k$) for $0 \leq a \leq n$.
\end{theorem}

It can be helpful to visualize $\scr X_1(p^n)_k$ as follows:

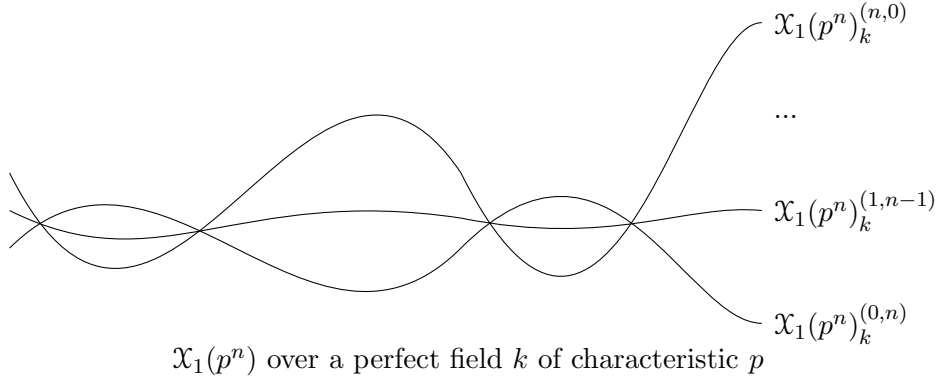
\begin{center}
\begin{tikzpicture}
\clip (0,-2) rectangle (12.5,3.3);
\draw (0,0) .. controls (2,2) and (4,-2) .. (6,0);
\draw (6,0) .. controls (8,2) and (9,-1) .. (10,-1);
\draw (0,1) .. controls (2,-3) and (4,4) .. (6,1);
\draw (6,1) .. controls (8,-3) and (9,3) .. (10,3);
\draw (0,0.5) .. controls (2,-0.5) and (3,0.85) .. (6,0.4);
\draw (6,0.4) .. controls (8,0) and (9,0.6) .. (10,0.5);
\node at (10,-1) [inner sep=0pt,label=0:$\scr X_1(p^n)^{(0,n)}_k$] {};
\node at (10,0.5) [inner sep=0pt,label=0:$\scr X_1(p^n)^{(1,n-1)}_k$] {};
\node at (10,1.75) [inner sep=0pt,label=0:...] {};
\node at (10,3) [inner sep=0pt,label=0:$\scr X_1(p^n)^{(n,0)}_k$] {};
\node at (2,-1.5) [inner sep=0pt,label=0:$\scr X_1(p^n)$ over a perfect field $k$ of characteristic $p$] {};
\end{tikzpicture}
\end{center}

\begin{definition}
Let $S \in \mathrm{Sch}/\mathbb{F}_p$, and let $H \leq (\mathbb{Z}/(p^n))^2$ such that $H \cong \mathbb{Z}/(p^n)$ (hence also $(\mathbb{Z}/(p^n))^2/H \cong \mathbb{Z}/(p^n)$). Let $E/S$ be a generalized elliptic curve. Let $(P,Q)$ be a $[\Gamma(p^n)]$-structure on $E$, corresponding to a group homomorphism $\phi:(\mathbb{Z}/(p^n))^2 \rightarrow E[p^n]$, $(1,0) \mapsto P$, $(0,1) \mapsto Q$. We say that $(P,Q)$ has \textit{component label $H$} if:
\begin{itemize}
  \item $\phi(H) \subseteq \mathrm{ker}(F^n)$, where $F:E \rightarrow E^{(p)}$ is the relative Frobenius, and 
  \item the resulting group scheme homomorphism $\mathbb{Z}/(p^n) \cong (\mathbb{Z}/(p^n))^2/H \rightarrow E^{(p^n)} = E/\mathrm{ker}(F^n)$ is a $[\Gamma_1(p^n)]$-structure on $E^{(p^n)}$. (Note that this is independent of the choice of isomorphism $(\mathbb{Z}/(p^n))^2/H \cong \mathbb{Z}/(p^n)$, although the $[\Gamma_1(p^n)]$-structure obtained depends on this choice.)
\end{itemize}

We define $\scr Y(p^n)^H_S$ (resp. $\scr X(p^n)^H_S$) to be the substack of $\scr Y(p^n)_S$ (resp. $\scr X(p^n)_S$) associating to a scheme $T/S$ the groupoid of tuples $(E,(P,Q))$, where $E/T$ is an elliptic curve (resp. generalized elliptic curve) and $(P,Q)$ is a $[\Gamma(p^n)]$-structure on $E$ of component label $H$.
\end{definition}

\begin{theorem}[$\textrm{\cite[13.7.6]{KM1}}$]
Let $k$ be a perfect field of characteristic $p$. $\scr Y(p^n)_k$ (resp. $\scr X(p^n)_k$) is the disjoint union, with crossings at the supersingular points, of the substacks $\scr Y(p^n)^H_k$ (resp. $\scr X(p^n)^H_k$) for $H \leq (\mathbb{Z}/(p^n))^2$ with $H \cong \mathbb{Z}/(p^n)$.
\end{theorem}

The proof of the above theorem immediately generalizes to a slightly more general setting which we will find useful when studying compactified stacks of $[\Gamma(N)]$-structures below. Let $K \leq (\mathbb{Z}/(p^n))^2$, and write $G_K = (\mathbb{Z}/(p^n))^2/K$. Then there exist integers $m\geq l \geq 0$ with $G_K \cong \mathbb{Z}/(p^m) \times \mathbb{Z}/(p^l)$. Suppose that $l\geq 1$, so $G_K \cong \mathbb{Z}/(p^m) \times \mathbb{Z}/(p^l)$ with $m\geq l \geq 1$. Given a $G_K$-structure $\phi$ on an ordinary elliptic curve $E/T/\mathbb{F}_p$ (in the sense of Definition \ref{gstructure}), \'etale locally on $T$ we can consider the composite 
\begin{displaymath}
G_K = \mathbb{Z}/(p^m) \times \mathbb{Z}/(p^l) \stackrel{\phi}{\rightarrow} E[p^n] \cong \mathbb{Z}/(p^n) \times \mu_{p^n} \stackrel{\pi_1}{\rightarrow} \mathbb{Z}/(p^n).
\end{displaymath}
Since $\phi$ is a $G_K$-structure, the kernel and image of this composite are necessarily cyclic. The same argument used in \cite[13.7.6]{KM1} to prove the above theorem shows that in characteristic $p$, $\scr Y_K$ breaks up into a union of substacks indexed in this way by the possible kernels of group homomorphisms $G_K \rightarrow \mathbb{Z}/(p^n)$, subject to the condition on the image just described. So we can describe $\scr Y_K$ as a union of closed substacks indexed by the set 
\begin{displaymath}
L_K := \{H\leq G_K|\textrm{$H$ and $G/H$ are both cyclic}\}.
\end{displaymath}
Since $G_K \cong \mathbb{Z}/(p^m) \times \mathbb{Z}/(p^l)$ with $m\geq l \geq 1$, we have that $H \in L_K$ if and only if $H \cong \mathbb{Z}/(p^m)$ or $H \cong \mathbb{Z}/(p^l)$. 

The rigorous definition, accounting for elliptic curves that might not be ordinary and for the case where $G_K$ is cyclic, is as follows:

\begin{definition}\label{gkstructures}
Let $K \leq (\mathbb{Z}/(N))^2$ with corresponding quotient $G_K = (\mathbb{Z}/(N))^2/K$. We write $\scr Y_K$ (resp. $\scr X_K$) for the algebraic stack over $\mathrm{Spec}(\mathbb{Z})$ associating to a scheme $S$ the groupoid of pairs $(E,\phi)$, where $E/S$ is an elliptic curve (resp. generalized elliptic curve) and $\phi$ is a $G_K$-structure on $E$. 

If $S$ is an $\mathbb{F}_p$-scheme, $N = p^n$ and $G_K \cong \mathbb{Z}/(p^m) \times \mathbb{Z}/(p^l)$ with $m \geq l \geq 1$, then for any $H \in L_K$, we define $\scr Y^H_{K,S} \subset \scr Y_{K,S}$ (resp. $\scr X^H_{K,S} \subset \scr X_{K,S}$) to be the substack associating to a scheme $T/S$ the groupoid of pairs $(E,\phi)$, where $E/T$ is an elliptic curve (resp. generalized elliptic curve) and $\phi: G_K \rightarrow E[p^m]$ is a $G_K$-structure such that:
\begin{itemize}
  \item $\phi(H) \subset \mathrm{ker}(F^m)$, where $F^m: E \rightarrow E^{(p^m)}$ is the $m$-fold relative Frobenius on $E$, and 
  \item the resulting group scheme homomorphism $G_K/H \rightarrow E/\mathrm{ker}(F^m) \cong E^{(p^m)}$ is a $G_K/H$-structure on $E^{(p^m)}$ in the sense of \cite[\S1.5]{KM1}.
\end{itemize}
In this case we say that the $G_K$-structure $\phi$ has \textit{component label $H$}.

If $G_K \cong \mathbb{Z}/(p^m)$ (i.e. $l = 0$), the stack $\scr Y_K$ (resp. $\scr X_K$) is isomorphic to $\scr Y_1(p^m)$ (resp. $\scr X_1(p^m)$), and for $H \cong \mathbb{Z}/(p^a) \in L_K$ and $S \in \mathrm{Sch}/\mathbb{F}_p$ we define $\scr Y^H_{K,S}$ (resp. $\scr X^H_{K,S}$) to be the substack $\scr Y_1(p^m)^{(a,m-a)}_S \subset \scr Y_1(p^m)_S$ (resp. $\scr X_1(p^m)^{(a,m-a)}_S \subset \scr X_1(p^m)_S$), as in Definition \ref{abcyclic}. We still say that $\scr Y^H_{K,S}$ and $\scr X^H_{K,S}$ classify $G_K$-structures of \textit{component label $H$}.
\end{definition}
The result is:

\begin{theorem}\label{gkstructurescharp}
Let $k$ be a perfect field of characteristic $p$. $\scr Y_{K,k}$ (resp. $\scr X_{K,k}$) is the disjoint union, with crossings at the supersingular points, of the closed substacks $\scr Y_{K,k}^H$ (resp. $\scr X_{K,k}^H$) for $H \in L_K$.
\end{theorem}

We also note that an analogue of Lemma \ref{gamma1comps} holds in this case, which we record for future use. It is clear that if $K' \leq K \leq (\mathbb{Z}/(p^n))^2$, giving a canonical quotient map $\pi: G_{K'} \rightarrow G_K$, and if $\phi: G_K \rightarrow E^{\mathrm{sm}}$ is a $G_K$-structure on a generalized elliptic curve $E/S/\mathbb{F}_p$, then a necessary condition for the composite $\phi \circ \pi: G_{K'} \rightarrow E^{\mathrm{sm}}$ to be a $G_{K'}$-structure is that $\mathrm{ker}(\pi)$ is cyclic. In fact, unwinding the definitions we immediately deduce:

\begin{lemma}\label{gammacomps}
Let $E/S/\mathbb{F}_p$ be a generalized elliptic curve with no supersingular fibers. Let $K' \leq K \leq (\mathbb{Z}/(p^n))^2$ such that the canonical quotient $\pi: G_{K'} \rightarrow G_K$ has cyclic kernel. Let $H \in L_K$ and suppose that $\phi: G_K \rightarrow E$ is a $G_K$-structure with component label $H$. Then $\phi \circ \pi$ is a $G_{K'}$ structure on $E$ if and only if $\pi^{-1}(H) \subseteq G_{K'}$ is cyclic, in which case $\phi \circ \pi$ has component label $\pi^{-1}(H) \in L_{K'}$.
\end{lemma}

\section{Generalities/review of twisted stable maps}

\subsection*{Twisted stable maps}

We will be studying moduli stacks of elliptic curves embedded in moduli stacks of twisted stable maps to tame stacks. We now recall the relevant definitions and results relating to twisted stable maps.

\begin{definition}[$\textrm{\cite[2.2]{AOV1}}$]
Let $G$ be a group scheme over a scheme $S$. Write $\mathrm{QCoh}^G(S)$ for the category of $G$-equivariant quasi-coherent sheaves on $S$; writing $\cB G$ for the classifying stack of $G$ over $S$, this is equivalent to $\mathrm{QCoh}(\cB G)$ (\cite[\S2.1]{AOV1}). We say that $G$ is \textit{linearly reductive} if the functor $\mathrm{QCoh}^G(S) \rightarrow \mathrm{QCoh}(S)$, $\mathcal{F} \mapsto \mathcal{F}^G$ is exact, or equivalently if the pushforward $\mathrm{QCoh}(\cB G) \rightarrow \mathrm{QCoh}(S)$ is exact.
\end{definition}

Linearly reductive group schemes are classified in \cite[\S2.3]{AOV1}. The examples in which we are most interested for this paper are the finite flat commutative linearly reductive group schemes $\mu_N$ and $\mu_N^2$ over $S$.

\begin{definition}[$\textrm{\cite[3.1]{AOV1}}$]
Let $\mathcal{X}$ be a locally finitely presented algebraic stack over a scheme $S$, with finite inertia. By \cite{KM2}, $\mathcal{X}$ has a coarse moduli space $\rho: \mathcal{X} \rightarrow X$, with $\rho$ proper. We say that $\mathcal{X}/S$ is \textit{tame} if $\rho_*: \mathrm{QCoh}(\mathcal{X}) \rightarrow \mathrm{QCoh}(X)$ is exact.
\end{definition}

As observed in \cite[\S3]{AOV1}, for any finite flat group scheme $G$ over a scheme $S$, the classifying stack $\cB G$ over $S$ is tame if and only if $G$ is linearly reductive. So in particular the classifying stacks $\cB\mu_N$ and $\cB\mu_N^2$ are always tame.

\begin{definition}[$\textrm{\cite[\S2]{AOV2}}$]\label{twistedcurvedef}
An \textit{$n$-marked twisted curve} over a scheme $S$ is a proper tame stack $\cC$ over $S$, with connected dimension $1$ geometric fibers, and coarse space $f:C \rightarrow S$ a nodal curve over $S$; together with $n$ closed substacks 
\begin{displaymath}
\{\Sigma_i \subset \cC\}_{i = 1}^n 
\end{displaymath}
which are fppf gerbes over $S$ mapping to $n$ markings $\{p_i \in C^{\mathrm{sm}}(S)\}$, such that:
\begin{itemize}
  \item the preimage in $\cC$ of the complement $C' \subset C$ of the markings and singular locus of $C/S$ maps isomorphically onto $C'$; 
  \item if $\overline{p} \rightarrow C$ is a geometric point mapping to the image in $C$ of a marking $\Sigma_i \subset \cC$, then 
  \begin{displaymath}
  \mathrm{Spec}(\cO_{C,\overline{p}}) \times_C \cC \simeq [D^{\mathrm{sh}}/\mu_r]
  \end{displaymath}
  for some $r\geq1$, where $D^{\mathrm{sh}}$ is the strict Henselization at $(\frm_{S,f(\overline{p})},z)$ of 
  \begin{displaymath}
  D = \mathrm{Spec}(\cO_{S,f(\overline{p})}[z]) 
  \end{displaymath}
  and $\zeta\in \mu_r$ acts by $z\mapsto \zeta\cdot z$; and 
  \item if $\overline{p} \rightarrow C$ is a geometric point mapping to a node of $C$, then
  \begin{displaymath}
  \mathrm{Spec}(\cO_{C,\overline{p}})\times_C \cC \simeq [D^{\mathrm{sh}}/\mu_r]
  \end{displaymath}
  for some $r\geq 1$, where $D^{\mathrm{sh}}$ is the strict Henselization at $(\frm_{S,f(\overline{p})},x,y)$ of 
  \begin{displaymath}
  D = \mathrm{Spec}(\cO_{S,f(\overline{p})}[x,y]/(xy-t))
  \end{displaymath}
  for some $t \in \frm_{S,f(\overline{p})}$, and $\zeta\in \mu_r$ acts by $x \mapsto \zeta\cdot x$ and $y\mapsto \zeta^{-1}\cdot y$.
\end{itemize}
We say a twisted curve $\cC/S$ has \textit{genus $g$} if the geometric fibers of its coarse space $C/S$ have geometric genus $g$, and we say an $n$-marked genus $g$ twisted curve $\cC/S$ is \textit{stable} if the genus $g$ curve $C/S$ with the markings $\{p_i\}$ is an $n$-marked genus $g$ stable curve over $S$.
\end{definition}

\begin{example}\label{stacky1gon}
Over any base scheme $S$, consider a N\'eron $1$-gon $C/S$ as in \S\ref{section2}. We have $C^{\mathrm{sm}} \cong \mathbb{G}_m$, and $C$ admits the structure of a generalized elliptic curve with an action $C^{\mathrm{sm}} \times C \rightarrow C$ extending the group scheme structure of $\mathbb{G}_m$. For any positive integer $N$, the inclusion $\mu_N \subset \mathbb{G}_m$ determines an action of $\mu_N$ on $C$, and the stack quotient $\cC := [C/\mu_N]$ is a twisted stable curve over $S$, with coarse space $f:C'\rightarrow S$ again a N\'eron $1$-gon. If $\overline{p} \rightarrow C'$ is a geometric point mapping to a node of $C'$, then 
\begin{displaymath}
\mathrm{Spec}(\cO_{C',\overline{p}}) \times_{C'} \cC \simeq [D^{\mathrm{sh}}/\mu_N],
\end{displaymath}
where $D^{\mathrm{sh}}$ denotes the strict Henselization of 
\begin{displaymath}
D := \mathrm{Spec}(\cO_{S,f(\overline{p})}[x,y]/(xy))
\end{displaymath}
at the point $(\mathfrak{m}_{S,f(\overline{p})},x,y)$ and $\zeta \in \mu_N$ acts by $x \mapsto \zeta \cdot x$ and $y \mapsto \zeta^{-1} \cdot y$. We will refer to this twisted curve as the \textit{standard $\mu_N$-stacky N\'eron $1$-gon over $S$}.
\end{example}

\begin{center}
\begin{tikzpicture}
\clip (-1,-1.6) rectangle (5,1.2);
\draw [name path=curve] (3,1) .. controls (-0.5,-5) and (-0.5,5) .. (3,-1);
\node at (2.35,0) [circle,draw,fill,inner sep=2pt,label=0:$\mu_{N}$] {};
\node at (1.8,-1) [inner sep=0pt,label=270: Standard $\mu_N$-stacky N\'eron $1$-gon] {};
\end{tikzpicture}
\end{center}

\begin{definition}[$\textrm{\cite[\S4]{AOV2}}$]
Let $\cX$ be a finitely presented algebraic stack, proper over a scheme $S$ and with finite inertia. A \textit{twisted stable map} to $\cX$ from an $n$-marked twisted curve $(\cC/S,\{\Sigma_i\})$ over $S$ with coarse space $(C/S,\{p_i\})$ is a morphism $\cC \rightarrow \cX$ of stacks over $S$ such that:
\begin{itemize}
  \item $\cC \rightarrow \cX$ is a representable morphism, and 
  \item the induced map $C \rightarrow X$ is a stable map from $(C,\{p_i\})$ to $X$.
\end{itemize}
Over any base scheme $S$, we write $\cK_{g,n}(\cX)$ for the stack over $S$ associating to a scheme $T/S$ the groupoid of pairs $(\cC,\alpha)$, where $\cC/T$ is an $n$-marked twisted curve whose coarse space $C/T$ is a genus $g$ nodal curve, and $\alpha: \cC \rightarrow \cX \times_S T$ is a twisted stable map. 
\end{definition}

\begin{proposition}[$\textrm{\cite[4.2]{AOV2}}$]
For $\cX$ as above, $\cK_{g,n}(\cX)$ is a locally finitely presented algebraic stack over $S$. 
\end{proposition}

There is a natural morphism of stacks $\cK_{g,n}(\cX) \rightarrow \overline{\cM}_{g,n}(X)$ (where $\overline{\cM}_{g,n}(X)$ is the usual Kontsevich stack of stable maps into $X$) defined by passing to coarse spaces. In particular, note that since a twisted stable map is required to be respresentable, if $\cX$ is representable then this map is an equality.

\begin{theorem}[$\textrm{\cite[4.3]{AOV2}}$]\label{propqfinite}
Let $\cX$ be as above, and also assume that $\cX$ as tame. Then $\cK_{g,n}(\cX)$ is proper and quasi-finite over $\overline{\cM}_{g,n}(X)$.
\end{theorem}

A quick word of caution: the twisted stable maps of \cite{AOV2} are a generalization of what are referred to as \textit{balanced} twisted stable maps in \cite{AV} and \cite{ACV}; we are only interested in balanced maps for this paper, so hopefully our notation $\cK_{g,n}(\cX)$ will not cause any confusion. 

Let $(\cC/S,\{\Sigma_i\})$ be an $n$-marked twisted curve with coarse space $C/S$. If $\overline{p} \rightarrow C$ is a geometric point mapping to the image in $C$ of the gerbe $\Sigma_i$, such that $\mathrm{Spec}(\cO_{C,\overline{p}}) \times_C \cC \simeq [D^{\mathrm{sh}}/\mu_r]$ as in Definition \ref{twistedcurvedef}, then the integer $r$ is uniquely determined by $\overline{p}$, and we call $r$ the \textit{local index of $\cC$ at $\overline{p}$}. In fact one verifies immediately (cf. \cite[8.5.1]{AV}) that the local index only depends on $i$ and on the image of $\overline{p}$ in $S$, and that the function $\epsilon_i: S \rightarrow \mathbb{Z}_{>0}$ sending $s = f(\overline{p}) \in S$ to the local index $r$ is locally constant.

\begin{notation}\label{globalindex}
Let $\mathbf{e} = (e_1,...,e_n) \in \mathbb{Z}^n_{>0}$. We say an $n$-marked twisted curve $(\cC/S,\{\Sigma_i\})$ has \textit{global index $\mathbf{e}$} if $\epsilon_i$ is constant of value $e_i$ for all $i = 1,...,n$.

If $\cX/S$ is a finitely presented algebraic stack, proper over $S$ and with finite inertia, we write $\cK^{\mathbf{e}}_{g,n}(\cX)$ for the substack of $\cK_{g,n}(\cX)$ associating to a scheme $T/S$ the groupoid of pairs $(\cC,\alpha)$, where $\cC/T$ is an $n$-marked genus $g$ twisted curve of global index $\mathbf{e}$, and $\alpha: \cC \rightarrow \cX \times_S T$ is a twisted stable map. Since the functions $\epsilon_i: S \rightarrow \mathbb{Z}_{>0}$ are locally constant, we see that $\cK^{\mathbf{e}}_{g,n}(\cX)$ is an open and closed substack of $\cK_{g,n}(\cX)$, and 
\begin{displaymath}
\cK_{g,n}(\cX) = \coprod_{\mathbf{e} \in \mathbb{Z}^n_{>0}} \cK^{\mathbf{e}}_{g,n}(\cX).
\end{displaymath}
In this paper we will generally only be interested in the case $\mathbf{e} = (1,...,1)$, in which case $\cK^{\mathbf{e}}_{g,n}(\cX)$ classifies twisted stable maps to $\cX$ from twisted curves whose markings are honest sections. We will write $\cK'_{g,n}(\cX)$ for $\cK^{(1,...,1)}_{g,n}(\cX)$.
\end{notation}

\begin{notation}
We write $\cK_{g,n}^\circ(\cX)$ for the open substack of $\cK'_{g,n}(\cX)$ associating to $T/S$ the groupoid of pairs $(\cC,\alpha)$, where $\cC/T$ is a smooth $n$-marked genus $g$ curve (with no stacky structure) and $\alpha: \cC \rightarrow \cX \times_S T$ is a twisted stable map. 
\end{notation}

\subsection*{Twisted covers and Picard schemes of twisted curves}

If $G/S$ is a linearly reductive finite flat group scheme, then we have already observed that $\cB G$ is tame. So in this case we can consider the algebraic stack $\cK_{g,n}(\cB G)$, which is proper and quasi-finite over the Deligne-Mumford stack $\overline{\cM}_{g,n}$ by \ref{propqfinite}, since the coarse space of $\cB G$ is $S$. 

\begin{theorem}[$\textrm{\cite[5.1]{AOV2}}$]
Let $G/S$ be a linearly reductive finite flat group scheme. Then $\cK_{g,n}(\cB G)$ is flat over $S$, with local complete intersection fibers. 
\end{theorem}

\begin{definition}
Since a map $\cC \rightarrow \cB G$ is equivalent to a $G$-torsor $P \rightarrow \cC$, it is often convenient to view $\cK_{g,n}( \cB G)$ as classifying such torsors. A $G$-torsor $P \rightarrow \cC$ is a \textit{twisted $G$-cover} of $\cC$ if it arises in this way. 
\end{definition}

Note that since the coarse space of $\cB G$ is $S$, the condition that the map $\cC \rightarrow \cB G$ is stable just says that $C/S$ is stable in the usual sense. So for $G$ as above, $\cK_{g,n}(\cB G)$ can and will be viewed as the stack classifying twisted $G$-covers of $n$-marked genus $g$ twisted stable curves.

Still writing $G$ for a finite flat linearly reductive group scheme over $S$, suppose in addition that $G$ is commutative. Then every object of $\cK_{g,n}(\cB G)$ canonically contains $G$ in the center of its automorphism group scheme (determining a unique subgroup stack $H$ in the center of the inertia stack of $\cK_{g,n}(\cB G)$ such that for any object $\xi \in \cK_{g,n}(\cB G)(T)$ the pullback of $H$ to $T$ is $G$). We can therefore apply the rigidification procedure described in \cite[\S5.1]{ACV} and generalized in \cite[Appendix A]{AOV1}, thereby ``removing'' $G$ from the automorphism group of each object:

\begin{definition}
$\overline{\cK}_{g,n}(\cB G)$ is the rigidification of $\cK_{g,n}(\cB G)$ with respect to the copy of $G$ naturally contained in the center of its inertia stack. We refer to the objects of $\overline{\cK}_{g,n}(\cB G)$ as \textit{rigidified twisted stable maps} into $\cB G$, or \textit{rigidified twisted $G$-covers} of $n$-marked genus $g$ twisted stable curves. We write $\overline{\cK}_{g,n}^\circ(\cB G)$ for the open substack corresponding to smooth curves and $\overline{\cK}'_{g,n}(\cB G)$ for the open and closed substack corresponding to twisted stable curves whose markings are representable.
\end{definition}

In \cite[$\textrm{\S}$C.2]{AGV} it is shown that the rigidification of an algebraic stack with respect to a group scheme admits a natural moduli interpretation. In this paper we are only interested in the case where $G$ is diagonalizable, in which case $\overline{\cK}_{g,n}(\cB G)$ can be given a more concrete moduli interpretation which we will now describe.

Given a twisted curve $\cC$ over a scheme $S$, let $\scr Pic_{\cC/S}$ denote the stack associating to $T/S$ the groupoid of line bundles on $\cC\times_S T$.

\begin{proposition}[$\textrm{\cite[2.7]{AOV2}}$]
$\scr Pic_{\cC/S}$ is a smooth locally finitely presented algebraic stack over $S$, and for any $\cL \in \scr Pic_{\cC/S}(T)$ the group scheme $\mathrm{Aut}_T(\cL)$ is canonically isomorphic to $\mathbb{G}_{m,T}$. 
\end{proposition}

Write $\mathrm{Pic}_{\cC/S}$ for the rigidification of this stack with respect to $\mathbb{G}_m$; $\mathrm{Pic}_{\cC/S}$ is none other than the relative Picard functor of $\cC/S$. From the analysis of $\mathrm{Pic}_{\cC/S}$ in \cite[\S2]{AOV2} we have:

\begin{proposition}\label{picardexact}
$\mathrm{Pic}_{\cC/S}$ is a smooth group scheme over $S$, and if $\pi:\cC \rightarrow C$ is the coarse space of $\cC/S$, then there is a short exact sequence of group schemes over $S$
\begin{displaymath}
0 \rightarrow \mathrm{Pic}_{C/S} \stackrel{\pi^*}{\rightarrow} \mathrm{Pic}_{\cC/S} \rightarrow W \rightarrow 0,
\end{displaymath}
with $W$ quasi-finite and \'etale over $S$. 
\end{proposition}

In fact, since $\pi_* \mathbb{G}_m = \mathbb{G}_m$ we can deduce from the exact sequence of low-degree terms of the fppf Leray spectral sequence 
\begin{displaymath}
E_2^{p,q} = H^p(C,\mathbf{R}^q\pi_* \mathbb{G}_m) \Rightarrow H^{p+q}(\cC, \mathbb{G}_m) 
\end{displaymath}
that $W$ is the sheaf associated to the presheaf $T \mapsto H^0(C_T,\mathbf{R}^1\pi_* \mathbb{G}_m)$ (where we still write $\pi: \cC_T \rightarrow C_T$ for the morphism induced by base change from $\pi:\cC \rightarrow C$).

For any integer $N$ annihilating $W$, we get a natural morphism 
\begin{displaymath}
\times N: \mathrm{Pic}_{\cC/S} \rightarrow \mathrm{Pic}_{C/S}.
\end{displaymath}

\begin{definition}[$\textrm{\cite[2.11]{AOV2}}$]
The \textit{generalized Jacobian} of $\cC$ is 
\begin{displaymath}
\mathrm{Pic}^0_{\cC/S} := \mathrm{Pic}_{\cC/S} \times_{ \times N, \mathrm{Pic}_{C/S}} \mathrm{Pic}^0_{C/S}, 
\end{displaymath}
where $\mathrm{Pic}^0_{C/S}$ is the fiberwise connected component of the identity in the group scheme $\mathrm{Pic}_{C/S}$.
\end{definition}

$\mathrm{Pic}^0_{\cC/S}$ is independent of $N$ and is viewed as classifying line bundles of fiberwise degree $0$ on $\cC/S$.

\begin{remark}
Unlike the case of $\mathrm{Pic}^0_{C/S}$ for $C/S$ a (non-stacky) genus $g$ curve, the geometric fibers of $\mathrm{Pic}^0_{\cC/S}$ need not be connected for $\cC/S$ a twisted curve. In fact, when $\cC/S$ is a $1$-marked genus $1$ twisted stable curve, $\mathrm{Pic}^0_{\cC/S}$ behaves like the smooth part of a generalized elliptic curve over $S$. For instance, if $\cC = [C/\mu_N]$ for $C/S$ a N\'eron $1$-gon as in Example \ref{stacky1gon}, it is easily verified that $\mathrm{Pic}^0_{\cC/S} \cong \mathbb{G}_m \times \mathbb{Z}/(N)$. Standard $\mu_N$-stacky N\'eron $1$-gons will play an analogous role in this paper to that of N\'eron polygons in \cite{DR} and \cite{C}. In particular, we have:
\end{remark}

\begin{lemma}\label{singulartwistedcurves}
Let $k$ be an algebraically closed field, and $\cC/k$ a $1$-marked genus $1$ twisted stable curve, with no stacky structure at its marking, such that the coarse space $C/k$ is not smooth. Then $\cC$ is a standard $\mu_N$-stacky N\'eron $1$-gon (Example \ref{stacky1gon}) for some (unique) positive integer $N$.
\end{lemma}

\begin{proof}
It follows from \cite[II.1.15]{DR} that the $1$-marked genus $1$ stable curve $C/k$ is a N\'eron $1$-gon. Write $\pi: \cC \rightarrow C$ for the map exhibiting $C$ as the coarse space of $\cC$. Write $\widetilde{C} = \mathbb{P}^1 \rightarrow C$ for the normalization of $C$, and write $\widetilde{\cC}$ for the following fiber product:
\begin{align*}
\xymatrix{ 
  \widetilde{\cC} \ar[r]^{\tau} \ar[d]_{\nu} & \mathbb{P}^1 \ar[d] \\
  \cC \ar[r]_{\pi} & C }
\end{align*}
We have a short exact sequence of fppf sheaves on $\mathbb{P}^1$ 
\begin{displaymath}
0 \rightarrow \cO^\times_{\mathbb{P}^1} \rightarrow \tau_* \cO^\times_{\widetilde{\cC}} \rightarrow \cF \rightarrow 0,
\end{displaymath}
where $\cF$ is the pushforward to $\mathbb{P}^1$ of $\cO^\times_Z$, for $Z \cong \cB\mu_N \sqcup \cB\mu_N$ the preimage in $\widetilde{\cC}$ of the (stacky) node of $\cC$. This induces an exact sequence 
\begin{displaymath}
0 \rightarrow \mathrm{Pic}(\mathbb{P}^1) \rightarrow \mathrm{Pic}(\widetilde{\cC}) \rightarrow \mathrm{Pic}(Z) \rightarrow H^2(\mathbb{P}^1,\cO^\times_{\mathbb{P}^1}) = 0,
\end{displaymath}
inducing an isomorphism $\mathrm{Pic}^0_{\widetilde{\cC}/k} \cong \mathrm{Pic}(Z) \cong \mathbb{Z}/(N) \times \mathbb{Z}/(N)$ since $\mathrm{Pic}^0_{\mathbb{P}^1/k} = 0$.

Now consider the short exact sequence of fppf sheaves on $\cC$ 
\begin{displaymath}
0 \rightarrow \cO^\times_{\cC} \rightarrow \nu_* \cO^\times_{\widetilde{\cC}} \rightarrow \cG \rightarrow 0,
\end{displaymath}
where $\cG$ is the pushforward to $\cC$ of $\cO^\times_p$ for $p \cong \cB\mu_N$ the (stacky) node of $\cC$. This induces an exact sequence 
\begin{displaymath}
0 \rightarrow k^\times \stackrel{\mathrm{id}}{\rightarrow} k^\times \stackrel{0}{\rightarrow} k^\times \rightarrow \mathrm{Pic}(\cC) \rightarrow \mathrm{Pic}(\widetilde{\cC}) \rightarrow \mathrm{Pic}(\cB\mu_N).
\end{displaymath}
We have $\mathrm{Pic}(\cB\mu_N) \cong \mathbb{Z}/(N)$ and $\mathrm{Pic}^0_{\cC/k} \cong \mathbb{Z}/(N) \times \mathbb{Z}/(N)$, and the map 
\begin{displaymath}
\mathrm{Pic}^0_{\cC/k} \cong \mathbb{Z}/(N) \times \mathbb{Z}/(N) \rightarrow \mathbb{Z}/(N) \cong \mathrm{Pic}(\cB\mu_N)
\end{displaymath}
is given by $(a,b) \mapsto a-b$. In particular the kernel is isomorphic to $\mathbb{Z}/(N)$, so we have a short exact sequence 
\begin{displaymath}
0 \rightarrow k^\times \rightarrow \mathrm{Pic}^0_{\cC/k} \stackrel{f}{\rightarrow} \mathbb{Z}/(N) \rightarrow 0.
\end{displaymath}
This sequence splits (noncanonically) since $k^\times$ is divisible (as $k$ is algebraically closed), so $\mathrm{Pic}^0_{\cC/k} \cong \mathbb{G}_m \times \mathbb{Z}/(N)$.

The fppf exact sequence of sheaves on $\cC$ 
\begin{displaymath}
0 \rightarrow \mu_N \rightarrow \mathbb{G}_m \rightarrow \mathbb{G}_m \rightarrow 0
\end{displaymath}
then gives us an isomorphism 
\begin{align*}
H^1 (\cC, \mu_N) & \cong \mathrm{ker}(\mathrm{Pic}_{\cC/k} \stackrel{\times N}{\longrightarrow} \mathrm{Pic}_{\cC/k}) \\
& \cong \mu_N \times \mathbb{Z}/(N).
\end{align*}

Let $C' \rightarrow \cC$ be the $\mu_N$-torsor over $\cC$ corresponding to the class $(1,1) \in \mu_N \times \mathbb{Z}/(N) \cong H^1 (\cC, \mu_N)$. Then $C'$ is representable by \cite[2.3.10]{AH}, since $(1,1)$ is the class in $H^1 (\cC, \mu_N)$ of a uniformizing line bundle on $\cC$. Let 
\begin{displaymath}
V = C^{\mathrm{sm}} \cong \cC \times_C C^{\mathrm{sm}} \stackrel{\iota}{\hookrightarrow} \cC, 
\end{displaymath}
and consider the resulting $\mu_N$-torsor $C'_V \rightarrow V \cong \mathbb{G}_m$. The pullback map 
\begin{displaymath}
\iota^*: H^1(\cC,\mu_N) \cong \mu_N \times \mathbb{Z}/(N) \rightarrow \mathbb{Z}/(N) \cong H^1(\mathbb{G}_m,\mu_N)
\end{displaymath}
is given by $(\zeta,a) \mapsto a$. It therefore follows that $C'_V \cong \mathbb{G}_m$, with the $\mu_N$-action given by the standard multiplication action on $\mathbb{G}_m$. The quotient map $\mathbb{G}_m \rightarrow [\mathbb{G}_m/\mu_N] \cong \mathbb{G}_m$ is the map ``$x\mapsto x^N$.''

Fix an \'etale neighborhood $W$ of the node of $C$ of the form $W = \mathrm{Spec}(k[z,w]/(zw))$, such that 
\begin{displaymath}
\cC \times_C W \cong [D/\mu_N]
\end{displaymath}
for $D = \mathrm{Spec}(k[x,y]/(xy))$, with $\zeta \in \mu_N$ acting by $x \mapsto \zeta x$ and $y \mapsto \zeta^{-1} y$. The composite 
\begin{displaymath}
D \rightarrow [D/\mu_N] \rightarrow W
\end{displaymath}
is given by the ring homomorphism $z \mapsto x^N, w \mapsto y^N$. Since $C'$ is representable and $C' \times_C W \rightarrow \cC \times_C W$ is a $\mu_N$-torsor, it follows that $C' \times_C W \cong D$ with the above $\mu_N$-action. In particular we see that $C'$ is a nodal curve with one node.

Composing our original $C' \rightarrow \cC$ with the coarse space map $\cC \rightarrow C$ gives us a finite morphism of nodal curves $C' \rightarrow C$, which restricts to the $\mu_N$-torsor $\mathbb{G}_m \rightarrow \mathbb{G}_m \cong C^{\mathrm{sm}}$ and which is totally ramified over the node of $C$. Riemann-Hurwitz for nodal curves implies that $C'$ has arithmetic genus $1$, so $C'$ is a N\'eron $1$-gon with smooth locus $C'^{\mathrm{sm}} = \mathbb{G}_m$. The multiplication action of $\mu_N$ on $\mathbb{G}_m$ extends uniquely to an action on $C'$, and by assumption $\cC = [C'/\mu_N]$. Thus $\cC$ is a standard $\mu_N$-stacky $1$-gon.
\end{proof}

\subsection*{Relationship to moduli of elliptic curves with level structure}

\begin{notation}
For any finite flat commutative group scheme $G$ over a base scheme $S$, $\scr H(G)$ is the stack over $S$ associating to an $S$-scheme $T$ the groupoid of pairs $(E,\phi)$, where $E/T$ is an elliptic curve and $\phi: G^* \rightarrow E[N]$ is a homomorphism of group schemes over $T$ (for $G^*$ the Cartier dual of $G$). For the $S$-scheme $G=\mu_N$ we write $\scr H_1(N)$ for $\scr H(\mu_N)$, and for the $S$-scheme $G = \mu_N^2$ we write $\scr H(N)$ for $\scr H(\mu_N^2)$. 
\end{notation}

As in \cite[\S2.3]{AOV1}, we say a finite commutative group scheme is \textit{diagonalizable} if its Cartier dual is constant, and \textit{locally diagonalizable} if its Cartier dual is \'etale. 

\begin{lemma}[$\textrm{\cite[5.7]{AOV2}}$]\label{bigpicard}
Let $G/S$ be a finite diagonalizable commutative group scheme, so $\cB G$ is tame (since $G$ is linearly reductive) and $X := G^*$ is constant. For any twisted curve $\cC/S$ there is an equivalence of categories between the stack $\mathrm{TORS}_{\cC/S}(G)$ classifying $G$-torsors on $\cC/S$ and the Picard stack $\scr Pic_{\cC/S}[X]$ of morphisms of Picard stacks $X \rightarrow \scr Pic_{\cC/S}$. 
\end{lemma}

The construction for the above equivalence is as follows:
\begin{quotation}
  Let $\pi:P \rightarrow \cC$ be a $G$-torsor over $\cC$. $G$ acts on the sheaf $\pi_* \cO_P$, yielding a decomposition 
\begin{displaymath}
\pi_* \cO_P = \bigoplus_{\chi \in X} \cL_\chi.
\end{displaymath}
Each $\cL_\chi$ is invertible since $P$ is a torsor over $E$, so this determines a morphism of Picard stacks
\begin{align*}
\phi_P: X & \rightarrow \scr Pic_{\cC/S} \\
\chi & \mapsto [\cL_\chi].
\end{align*}
And conversely, such a morphism $\phi: X \rightarrow \scr Pic_{\cC/S}$ naturally determines an algebra structure on 
\begin{displaymath}
\bigoplus_{\chi \in X} \phi(\chi),
\end{displaymath}
giving a $G$-torsor 
\begin{displaymath}
\underline{\mathrm{Spec}}_{\cC} \big( \bigoplus_{\chi \in X} \phi(\chi) \big) \rightarrow \cC
\end{displaymath}
with the $G$-action determined by the $X$-grading.
\end{quotation}
This defines an open immersion from $\cK_{g,n}(\cB G)$ into the algebraic stack over $S$ associating to an $S$-scheme $T$ the groupoid of pairs $(\cC,\phi)$, where $\cC/T$ is an $n$-marked genus $g$ twisted stable curve and $\phi \in \scr Pic_{\cC/T}[X]$. Rigidifying $\cK_{g,n}(\cB G)$ and $\scr Pic_{\cC/T}$ with respect to the group schemes $G$ and $\mathbb{G}_m$ respectively, we have an open immersion from $\overline{\cK}_{g,n}(\cB G)$ into the stack classifying pairs $(\cC,\phi)$, where $\cC/T$ is an $n$-marked genus $g$ twisted stable curve and $\phi: X \rightarrow \mathrm{Pic}_{\cC/T}$ is a homomorphism of group schemes over $T$.

Writing $f:\cC\rightarrow S$ for the structural morphism, we have $\mathrm{Pic}_{\cC/S} = \mathbf{R}^1 f_* \mathbb{G}_m$. Therefore fppf-locally on $T$, $\phi$ corresponds to the choice of an $X$-torsor $P \rightarrow \cC$, with $P$ representable if and only if $(\cC,\phi)$ comes from an object of $\overline{\cK}_{g,n}(\cB G)$ (refer for example to \cite[2.3.10]{AH} to see that a morphism from a twisted curve $\cC$ to $\cB G$ is representable if and only if the corresponding $G$-torsor over $\cC$ is representable). This gives us:

\begin{corollary}\label{picard}
For a finite flat diagonalizable group scheme $G/S$, the above construction gives an equivalence between $\overline{\mc K}_{g,n}(\cB G)$ and the stack over $S$ associating to $T/S$ the groupoid of pairs $(\cC,\phi)$, where $\cC/T$ is an $n$-marked genus $g$ twisted stable curve and $\phi:X \rightarrow \mathrm{Pic}_{\cC/T}$ is a group scheme homomorphism such that, fppf-locally on $T$, the $G$-torsor over $\cC$ corresponding to $\phi$ as above is representable. 
\end{corollary}

In particular, we get an equivalence 
\begin{displaymath}
\overline{\cK}^\circ_{1,1}(\cB G) \simeq \mc H(G),
\end{displaymath}
since for an elliptic curve $E/S$ every $G$-torsor over $E$ is representable and we canonically have $E \cong \mathrm{Pic}^0_{E/S}$. This isomorphism sends $Q \in E(S)$ to the class of the line bundle $\cL((Q)-(0_E))$, so as a special case we see that if $\phi: \mathbb{Z}/(N) \rightarrow E$ is a group scheme homomorphism with $\phi(1) = Q$, the corresponding $\mu_N$-torsor over $E$ is of the form
\begin{displaymath}
P = \underline{\mathrm{Spec}}_E \big( \bigoplus_{a = 0}^{N-1} \cL((a\cdot Q)-(0_E)) \big).
\end{displaymath}

Since $\scr H_1(N)$ naturally contains a closed substack isomorphic to the stack $\scr Y_1(N)$ over $S$ classifying $[\Gamma_1(N)]$-structures on elliptic curves, we see that the algebraic stack $\overline{\cK}_{1,1}(\cB \mu_N)$ is a modular compactification of $\scr Y_1(N)$. Similarly, the algebraic stack $\overline{\cK}_{1,1}(\cB \mu_N^2)$ is a modular compactification of the stack $\scr Y(N)$ classifying (not necessarily symplectic) full level $N$ structures on elliptic curves. The task in both cases is to give a moduli interpretation of the closure of these classical moduli stacks in $\overline{\cK}_{1,1}(\cB G)$, and we address this in the following sections.

\begin{lemma}\label{dense}
$\overline{\cK}^\circ_{1,1}(\cB\mu_N)$ is dense in $\overline{\cK}'_{1,1}(\cB\mu_N)$, and $\overline{\cK}^\circ_{1,1}(\cB\mu_N^2)$ is dense in $\overline{\cK}'_{1,1}(\cB\mu_N^2)$.
\end{lemma}
\begin{proof}
Let $\cC_0$ be a standard $\mu_d$-stacky $1$-gon over an algebraically closed field $k$, and let $\phi_0: \mathbb{Z}/(N) \rightarrow \mathrm{Pic}^0_{\cC_k} \cong \mathbb{G}_m \times \mathbb{Z}/(d)$ be a group scheme homomorphism such that $(\cC_0, \phi_0) \in \overline{\cK}'_{1,1}(\cB \mu_N)(k)$ (so $d|N$ and $\phi_0$ meets every component of $\mathrm{Pic}^0_{\cC_0/k}$). We need to lift $(\cC_0/k, \phi_0)$ to a pair $(\cC, \phi) \in \overline{\cK}'_{1,1}(\cB \mu_N)(A)$ for a local ring $A$ with residue field $k$, such that the generic fiber of $\cC$ is a smooth elliptic curve.

Let $\cC_0'/k$ be a standard $\mu_N$-stacky $1$-gon, and consider the morphism $\cC_0' \rightarrow \cB \mu_N$ corresponding to the group scheme homomorphism $\mathbb{Z}/(N) \rightarrow \mathrm{Pic}^0_{\cC_0'/k} \cong \mathbb{G}_m \times \mathbb{Z}/(N)$ sending $1$ to $(1,N/d)$. This morphism is not representable; the corresponding $\mu_N$-torsor over $\cC_0'$ is as follows:
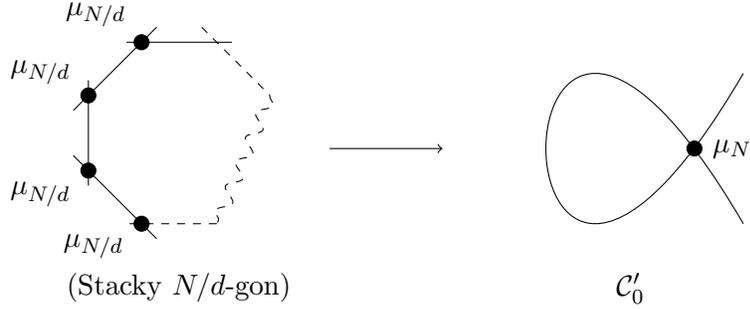
\begin{center}
\begin{tikzpicture}
\clip (-2,2.9) rectangle (9,7);
\draw (0.2,3.8) -- (-0.907,4.907);
\draw (-0.707,4.507) -- (-0.707,5.907);
\draw (-0.907,5.507) -- (0.2,6.614);
\draw (-0.2,6.414) -- (1.2,6.414);
\begin{scope}[dashed]
\draw (0.8,6.614) -- (1.707,5.707);
\draw (1,4) -- (-0.2,4);
\draw [decorate,decoration=snake] (1.707,5.707) -- (1,4);
\end{scope}
\node at (0,4) [circle,draw,fill,inner sep=2pt,label=185:$\mu_{N/d}$] {};
\node at (-0.707,4.707) [circle,draw,fill,inner sep=2pt,label=190:$\mu_{N/d}$] {};
\node at (-0.707,5.707) [circle,draw,fill,inner sep=2pt,label=170:$\mu_{N/d}$] {};
\node at (0,6.414) [circle,draw,fill,inner sep=2pt,label=135:$\mu_{N/d}$] {};
\draw [name path=curve] (8,6) .. controls (4.5,0) and (4.5,10) .. (8,4);
\node at (7.35,5) [circle,draw,fill,inner sep=2pt,label=0:$\mu_{N}$] {};
\node at (6.5,3.5) [inner sep=0pt,label=270:\textrm{$\cC_0'$}] {};
\node at (0.5,3.5) [inner sep=0pt,label=270:(Stacky $N/d$-gon)] {};
\draw [->] (2.5,5) -- (4,5);
\end{tikzpicture}
\end{center}
We may factor the morphism $\cC_0' \rightarrow \cB\mu_N$ as $\cC_0' \rightarrow \cC_0 \rightarrow \cB \mu_N$, where $\cC_0 \rightarrow \cB \mu_N$ is the relative coarse moduli space (\cite[Thm. 3.1]{AOV2}) of $\cC_0' \rightarrow \cB \mu_N$:
\begin{center}
\begin{tikzpicture}
\clip (-2,2.9) rectangle (9,7);
\draw (0.2,3.8) -- (-0.907,4.907);
\draw (-0.707,4.507) -- (-0.707,5.907);
\draw (-0.907,5.507) -- (0.2,6.614);
\draw (-0.2,6.414) -- (1.2,6.414);
\begin{scope}[dashed]
\draw (0.8,6.614) -- (1.707,5.707);
\draw (1,4) -- (-0.2,4);
\draw [decorate,decoration=snake] (1.707,5.707) -- (1,4);
\end{scope}
\draw [name path=curve] (8,6) .. controls (4.5,0) and (4.5,10) .. (8,4);
\node at (7.35,5) [circle,draw,fill,inner sep=2pt,label=0:$\mu_{d}$] {};
\node at (6.5,3.5) [inner sep=0pt,label=270:\textrm{$\cC_0$}] {};
\node at (0.5,3.5) [inner sep=0pt,label=270:($N/d$-gon)] {};
\draw [->] (2.5,5) -- (4,5);
\end{tikzpicture}
\end{center}
Since this $\cC_0$ is a standard $\mu_d$-stacky $1$-gon, we may identify it with our original twisted curve $\cC_0$. This gives us a morphism $\cC_0' \rightarrow \cC_0$, and the resulting pullback map $\mathrm{Pic}^0_{\cC_0/k} \rightarrow \mathrm{Pic}^0_{\cC'_0/k}$ is the monomorphism 
\begin{align*}
\mathbb{G}_m \times \mathbb{Z}/(d) & \rightarrow \mathbb{G}_m \times \mathbb{Z}/(N) \\
(\zeta, a) & \mapsto (\zeta, \frac{N}{d} \cdot a).
\end{align*}

Let $\cE/k[\![q^{1/N}]\!]$ be an $N$-gon Tate curve, so the special fiber of $\cE$ is a N\'eron $N$-gon, $\cE \otimes k((q^{1/N}))$ is a smooth elliptic curve, and we have an isomorphism $\cE^{\mathrm{sm}}[N] \cong \mu_N \times \mathbb{Z}/(N)$ of finite flat commutative group schemes over $k[\![q^{1/N}]\!]$. Let $Q = (1,1) \in \cE^{\mathrm{sm}}[N]$. The relative effective Cartier divisor 
\begin{displaymath}
D := \sum_{a \in \mathbb{Z}/(N)} [a\cdot Q]
\end{displaymath}
in $\cE^{\mathrm{sm}}$ is a subgroup scheme, \'etale over $k[\![q^{1/N}]\!]$, and the quotient $\overline{\cE} := \cE/D$ is naturally a generalized elliptic curve whose special fiber is a $1$-gon. The stack quotient $\cC := [\cE/\cE^{\mathrm{sm}}[N]]$ is naturally a twisted curve, whose generic fiber is an elliptic curve and whose special fiber is $\cC_0'$. Writing $\pi: \overline{\cE} \rightarrow \cC$ for the natural quotient map, we will see in \S\ref{compare} that for any line bundle $\cL$ on $\overline{\cE}$ there is a canonical decomposition 
\begin{displaymath}
\pi_* \cL \cong \bigoplus_{a \in \mathbb{Z}/(N)} \cL_a,
\end{displaymath}
where each $\cL_a$ is a line bundle on $\cC$.

For a section $R \in \cE^{\mathrm{sm}} (k[\![q^{1/N}]\!])$, we write $\overline{R}$ for its image in $\overline{\cE}^{\mathrm{sm}} (k[\![q^{1/N}]\!])$. Then we have the degree $0$ line bundle $\cL_R := \cL((\overline{R}) - (0_{\overline{\cE}}))$ on $\overline{\cE}$, hence a canonical decomposition 
\begin{displaymath}
\pi_* \cL_R \cong \bigoplus_{a \in \mathbb{Z}/(N)} \cL_{R,a}.
\end{displaymath}
We will see in \S\ref{compare} that the map $\cE^{\mathrm{sm}}[N] \rightarrow \mathrm{Pic}^0_{\cC/k[\![q^{1/N}]\!]} [N]$ sending $R = (\zeta, a) \in \mu_N \times \mathbb{Z}/(N) \cong \cE^{\mathrm{sm}}[N]$ to $\cL_{R,a}$ is an isomorphism of group schemes over $k[\![q^{1/N}]\!]$.

Returning to our original pair $(\cC_0, \phi_0) \in \overline{\cK}'_{1,1}(\cB\mu_N)(k)$, write $\phi(1) = (\zeta, a) \in \mu_N \times \mathbb{Z}/(d) \cong \mathrm{Pic}^0_{\cC_0/k}[N]$. Via the map $\cC_0' \rightarrow \cC_0$ constructed above, this corresponds to $(\zeta, (N/d) \cdot a) \in \mu_N \times \mathbb{Z}/(N) \cong \mathrm{Pic}^0_{\cC'_0/k}[N]$. $\cC'_0$ is the special fiber of the twisted curve $\cC$, and we have an isomorphism over $k[\![q^{1/N}]\!]$ 
\begin{displaymath}
\mathrm{Pic}^0_{\cC/k[\![q^{1/N}]\!]}[N] \cong \cE^{\mathrm{sm}}[N] \cong \mu_N \times \mathbb{Z}/(N).
\end{displaymath}
Now $(\zeta, (N/d)\cdot a)$ lifts to a section of $\mathrm{Pic}^0_{\cC/k[\![q^{1/N}]\!]}[N]$, corresponding to a group scheme homomorphism $\mathbb{Z}/(N) \rightarrow \mathrm{Pic}^0_{\cC/k[\![q^{1/N}]\!]}$, hence to a morphism $\cC \rightarrow \cB \mu_N$. Writing $\overline{\cC} \rightarrow \cB\mu_N$ for the relative coarse moduli space and $\phi: \mathbb{Z}/(N) \rightarrow \mathrm{Pic}^0_{\overline{\cC}/k[\![q^{1/N}]\!]}$ for the corresponding group scheme homomorphism, we see that $\overline{\cC}/k[\![q^{1/N}]\!]$ is a twisted curve with special fiber $\cC_0$ and generic fiber an elliptic curve. $\phi$ extends $\phi_0$ and $(\overline{\cC}, \phi) \in \overline{\cK}'_{1,1}(\cB\mu_N)(k[\![q^{1/N}]\!])$ as desired.

A similar argument of course applies to $\overline{\cK}^\circ_{1,1}(\cB\mu_N^2) \subset \overline{\cK}'_{1,1}(\cB\mu_N^2)$.
\end{proof}

We will require a concrete description of the $\mu_N$-torsor corresponding to a particular sort of $[\Gamma_1(N)]$-structure on an elliptic curve:

\begin{lemma}\label{quotient}
Let $K$ be a field and $E/K$ an elliptic curve. Let $Q \in E(K)$ be a $[\Gamma_1(N)]$-structure on $E$ such that the relative effective Cartier divisor 
\begin{displaymath}
D := \sum_{a = 0}^{N-1} [a\cdot Q]
\end{displaymath}
in $E$ is \'etale over $\mathrm{Spec}(K)$. Let $P \rightarrow E$ be the $\mu_N$-torsor corresponding to $Q$ as in \ref{picard}:
\begin{displaymath}
P = \underline{\mathrm{Spec}}_E \big( \bigoplus_{a = 0}^{N-1} \cL( (a \cdot Q)-(0_E)) \big).
\end{displaymath}
Then $P$ can be naturally identified with the quotient $E/D$, where $D$ is viewed as a subgroup scheme of $E$, \'etale of rank $N$ over $\mathrm{Spec}(K)$, with the quotient map $P \rightarrow E$ corresponding to the natural map $E/D \rightarrow E/E[N] \cong E$.
\end{lemma}

\begin{proof}
Consider the $e_N$-pairing on $E[N]$, a nondegenerate bilinear pairing of finite flat group schemes over $\mathrm{Spec}(K)$: 
\begin{displaymath}
e_N: E[N] \times E[N] \rightarrow \mu_N.
\end{displaymath}
Under our assumptions, the composite map 
\begin{displaymath}
E[N]/D = \{Q\}\times E[N]/D \hookrightarrow D \times E[N]/D \stackrel{e_N}{\rightarrow} \mu_N
\end{displaymath}
is an isomorphism of group schemes over $\mathrm{Spec}(K)$. Then via this isomorphism, $\mu_N$ acts on the quotient $E/D$ through the natural action of the subgroup scheme $E[N]/D \subset (E/D)[N]$, making $E/D$ a $\mu_N$-torsor over $E/E[N] \cong E$ with quotient map the obvious one induced from the factorization of $[N]$ as $E \rightarrow E/D \rightarrow E/E[N] \cong E$.

By Lemma \ref{bigpicard}, we may express the $\mu_N$-torsor $E/D \rightarrow E$ as 
\begin{displaymath}
E/D = \underline{\mathrm{Spec}}_E \big( \bigoplus_{a = 0}^{N-1} \cL_a \big)
\end{displaymath}
for some line bundles $\cL_a \in \mathrm{Pic}^0_{E/K}[N]$, with the algebra structure determined by isomorphisms $\cL_a \otimes \cL_b \cong \cL_{\textrm{$a+b$ mod $N$}}$ and the $\mu_N$-action corresponding to the grading. We have a natural isomorphism of group schemes over $K$ 
\begin{align*}
E & \rightarrow \mathrm{Pic}^0_{E/K} \\
R \in E(K) & \mapsto \cL((R)-(0_E)),
\end{align*}
so we conclude that $\cL_1 \cong \cL((Q_0)-(0_E))$ for some $Q_0 \in E[N]$, and $\cL_a \cong \cL((a\cdot Q_0)-(0_E))$.

Let $\hat{\pi}: E \rightarrow E/D$ be the isogeny dual to $\pi: E/D \rightarrow E/E[N] \cong E$. Identifying $E \cong \mathrm{Pic}^0_{E/K}$ and $E/D \cong \mathrm{Pic}^0_{(E/D)/K}$, $\hat{\pi}$ is simply given by the pullback map $\pi^*: \mathrm{Pic}^0_{E/K} \rightarrow \mathrm{Pic}^0_{(E/D)/K}$. For any line bundle $\cL$ on $E/K$ we have 
\begin{displaymath}
\pi^*(\cL) = \bigoplus_{a = 0}^{N-1} \cL \otimes \cL((a\cdot Q_0)-(0_E)),
\end{displaymath}
viewing the direct sum of line bundles on $E$ as a line bundle on $E/D = \underline{\mathrm{Spec}}_E( \oplus \cL((a\cdot Q_0)-(0_E)))$. In particular, for our original $[\Gamma_1(N)]$-structure $Q$,
\begin{displaymath}
\pi^* (\cL((Q)-(0_E))) = \bigoplus_{a = 0}^{N-1} \cL((Q+a\cdot Q_0)-(0_E)).
\end{displaymath}
But the dual isogeny to $\pi: E/D \rightarrow E/E[N]\cong E$ is the natural quotient map $E \rightarrow E/D$, and this maps $Q$ to $0_E$. Therefore the line bundle $\oplus \cL((Q+a\cdot Q_0)-(0_E))$ on $E/D$ is the trivial line bundle on $E/D$:
\begin{displaymath}
\bigoplus_{a = 0}^{N-1} \cL((Q+a\cdot Q_0)-(0_E)) \cong \bigoplus_{a = 0}^{N-1} \cL((a\cdot Q_0)-(0_E)).
\end{displaymath}
Therefore $Q$ is contained in the subgroup scheme of $E$ generated by $Q_0$: $Q = b\cdot Q_0$ for some $b$. Since $Q_0 \in E[N]$ and $N$ is the minimal positive integer with $N\cdot Q = 0_E$, this implies that $b \in (\mathbb{Z}/(N))^\times$, and in fact by the definition of the $e_N$-pairing we have $b = 1$, i.e. $Q_0 = Q$. Thus
\begin{displaymath}
\underline{\mathrm{Spec}}_E \big( \bigoplus_{a = 0}^{N-1} \cL((a\cdot Q_0)-(0_E)) \big) = \underline{\mathrm{Spec}}_E \big( \bigoplus_{a = 0}^{N-1} \cL((a\cdot Q)-(0_E)) \big),
\end{displaymath}
that is, $E/D \cong P$ with the quotient map $P \rightarrow E$ of the given $\mu_N$-action becoming identified with the natural quotient map $E/D \rightarrow E/E[N] \cong E$.
\end{proof}

\section{Moduli of elliptic curves in $\overline{\cK}_{1,1}(\cB \mu_N)$}

\subsection*{Reduction mod $p$ of $\scr H_1(N)$}

We first describe how the different components of $\overline{\cK}_{1,1}^\circ (\cB \mu_N) \cong \scr H_1(N)$ interact. These results are direct corollaries of \cite[\S13.5]{KM1}. Continue working over an arbitrary base scheme $S$. First consider the case where $N=p^n$ is a prime power. For each $0 \leq m \leq n$ we get a closed immersion 
\begin{displaymath}
\iota^{(p^m)}: \scr Y_1(p^m) \hookrightarrow \scr H_1(p^n)
\end{displaymath}
sending a pair $(E,\phi) \in \scr Y_1(p^m)(T)$, where $\phi:\mathbb{Z}/(p^m) \rightarrow E[p^m]$ is a $[\Gamma_1(p^m)]$-structure on $E/T$, to the pair $(E,\widetilde{\phi}) \in \scr H_1(p^n)(T)$, where $\widetilde{\phi}: \mathbb{Z}/(p^n) \rightarrow E[p^n]$ is obtained from $\phi$ by precomposing with the canonical projection $\mathbb{Z}/(p^n) \rightarrow \mathbb{Z}/(p^m)$. These yield a proper surjection of algebraic stacks
\begin{displaymath}
\scr Y_1(p^n) \sqcup \scr Y_1 (p^{n-1}) \sqcup ... \sqcup \scr Y_1(p) \sqcup \scr Y_1(1) \rightarrow \scr H_1(p^n)
\end{displaymath}
which is an isomorphism over $S[1/p]$.

But this is not an isomorphism over $S \otimes \mathbb{F}_p$. Recall that by Theorem \ref{y1-char-p}, for any perfect field $k$ of characteristic $p$, $\scr Y_1(p^m)_{k}$ is the disjoint union, with crossings at the supersingular points, of components 
\begin{align*}
\scr Y_1(p^m)_{k}^{(m-b,b)}\:\: (0\leq b \leq m),
\end{align*}
where an object of $\scr Y_1(p^m)_{k}^{(m-b,b)}(T)$ is a pair $(E,\phi)$ where $E/T$ is an elliptic curve and $\phi:\mathbb{Z}/(p^m) \rightarrow E[p^m]$ is a $[\Gamma_1(p^m)]$-$(m-b,b)$-cyclic structure on $E/T$ (Definition \ref{abcyclic}). Such an object corresponds via $\iota^{(m)}$ to the pair $(E,\widetilde{\phi}) \in \scr H_1(p^n)(T)$ as described above.

The key observation is Lemma \ref{gamma1comps}: if $\phi: \mathbb{Z}/(p^m) \rightarrow E$ is a $[\Gamma_1(p^m)]$-$(m-b,b)$-cyclic structure on an ordinary elliptic curve $E/T/\mathbb{F}_p$, then $\widetilde{\phi}:=\phi\circ\pi: \mathbb{Z}/(p^n) \rightarrow E$ is a $[\Gamma_1(p^n)]$-$(n-b,b)$-cyclic structure on $E$, where $\pi: \mathbb{Z}/(p^n) \rightarrow \mathbb{Z}/(p^m)$ is the natural projection. \'Etale locally on $T$ such that $E[p^n] \cong \mu_{p^n} \times \mathbb{Z}/(p^n)$ and $E[p^m] \cong \mu_{p^m} \times \mathbb{Z}/(p^m)$, the $[\Gamma_1(p^m)]$-$(m-b,b)$-cyclic structure $\phi$ corresponds to a section of 
\begin{align*}
\mu_{p^m}^\times \times (\mathbb{Z}/(p^b))^\times &\:\:\: \textrm{if $b<m$} \\
\mu_{p^m} \times (\mathbb{Z}/(p^b))^\times &\:\:\: \textrm{if $b=m$},
\end{align*}
and the reason that such a section also gives a $[\Gamma_1(p^n)]$-structure is that in characteristic $p$, unlike in characteristic $\neq p$, if $c <n$ and $\mathbb{Z}/(p^c) \rightarrow \mu_{p^c}$ is a $\mathbb{Z}/(p^c)$-generator then the composite 
\begin{displaymath}
\mathbb{Z}/(p^n) \twoheadrightarrow \mathbb{Z}/(p^c) \rightarrow \mu_{p^c} \hookrightarrow \mu_{p^n}
\end{displaymath}
is a $\mathbb{Z}/(p^n)$-generator. This gives us:
\begin{proposition}\label{y1comps}
Let $k$ be a perfect field of characteristic $p$. $\scr H_1(p^n)_k$ is the disjoint union, with crossings at the supersingular points, of components $Z_b$ for $0 \leq b \leq n$, where 
\begin{displaymath}
Z_b = \bigcup_{b\leq m \leq n} \scr Y_1(p^m)^{(m-b,b)}_k,
\end{displaymath}
identifying each $\scr Y_1(p^m)^{(m-b,b)}_k$ with a closed substack of $\scr H_1(p^n)_k$ via $\iota^{(p^m)}$. Each $\scr Y_1(p^m)^{(m-b,b)}_k$ is ``set-theoretically identified with $Z_b$'' in the sense that $(Z_b)_{\mathrm{red}} = \scr Y_1(p^m)^{(m-b,b)}_{k,\mathrm{red}}$ as substacks of $\scr H_1(p^n)_{k,\mathrm{red}}$ for all $b\leq m \leq n$.
\end{proposition}

To illustrate, visualize $\scr H_1(p^n)_{k,\mathrm{red}}$ as follows:

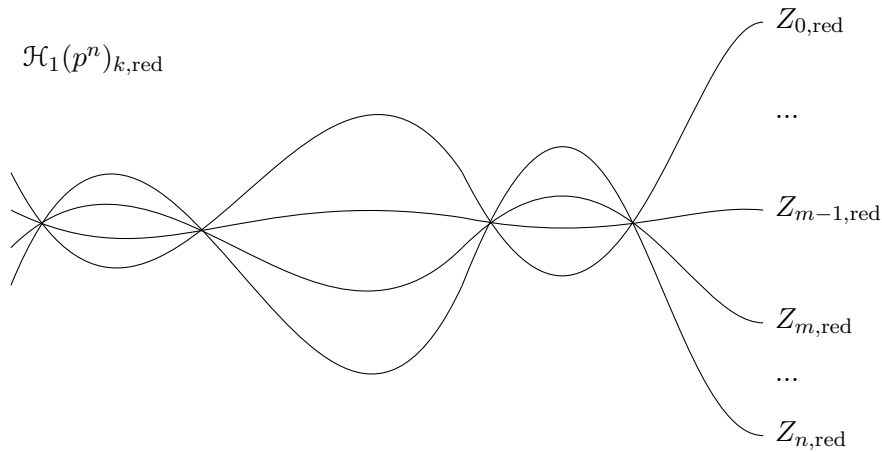
\begin{center}
\begin{tikzpicture}
\clip (-1,-2.7) rectangle (12,3.2);
\draw (0,0) .. controls (2,2) and (4,-2) .. (6,0);
\draw (6,0) .. controls (8,2) and (9,-1) .. (10,-1);
\draw (0,1) .. controls (2,-3) and (4,4) .. (6,1);
\draw (6,1) .. controls (8,-3) and (9,3) .. (10,3);
\draw (0,0.5) .. controls (2,-0.5) and (3,0.85) .. (6,0.4);
\draw (6,0.4) .. controls (8,0) and (9,0.6) .. (10,0.5);
\draw (0,-0.5) .. controls (2,4.4) and (4,-4.8) .. (6,-0.5);
\draw (6,-0.5) .. controls (8,4.6) and (8.6,-2.5) .. (10,-2.5);
\node at (10,-2.5) [inner sep=0pt,label=0:\textrm{$Z_{n,\mathrm{red}}$}] {};
\node at (10,-1.75) [inner sep=0pt,label=0:...] {};
\node at (10,-1) [inner sep=0pt,label=0:\textrm{$Z_{m,\mathrm{red}}$}] {};
\node at (10,0.5) [inner sep=0pt,label=0:\textrm{$Z_{m-1,\mathrm{red}}$}] {};
\node at (10,1.75) [inner sep=0pt,label=0:...] {};
\node at (10,3) [inner sep=0pt,label=0:\textrm{$Z_{0,\mathrm{red}}$}] {};
\node at (0,2.5) [inner sep=0pt,label=0:\textrm{$\scr H_1(p^n)_{k,\mathrm{red}}$}] {};
\end{tikzpicture}
\end{center}
The closed immersion $\iota^{(p^m)}$ sends the following copy of $\scr Y_1(p^m)_k$ to the obvious closed substack of $\scr H_1(p^n)_k$, contributing nilpotent structure to the components $Z_0,...,Z_m$:

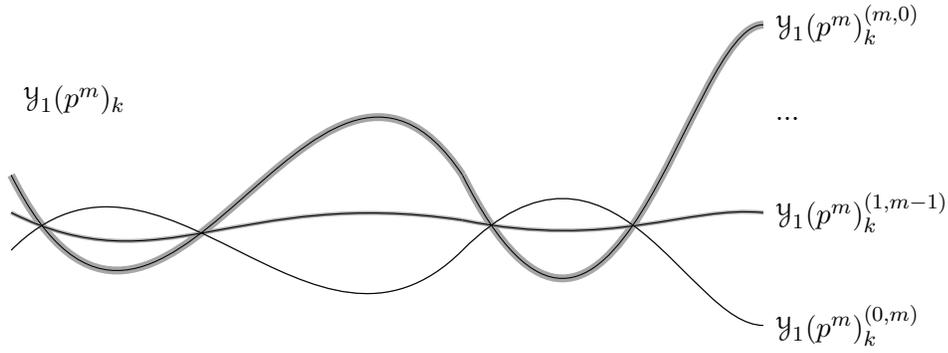
\begin{center}
\begin{tikzpicture}
\clip (-1,-8.3) rectangle (12.5,-3.7);
\draw [gray!70,line width=0.5pt] (0,-7) .. controls (2,-5) and (4,-9) .. (6,-7);
\draw [gray!70,line width=0.5pt] (6,-7) .. controls (8,-5) and (9,-8) .. (10,-8);
\draw [gray!70,line width=3pt] (0,-6) .. controls (2,-10) and (4,-3) .. (6,-6);
\draw [gray!70,line width=3pt] (6,-6) .. controls (8,-10) and (9,-4) .. (10,-4);
\draw [gray!70,line width=1.5pt] (0,-6.5) .. controls (2,-7.5) and (3,-6.15) .. (6,-6.6);
\draw [gray!70,line width=1.5pt] (6,-6.6) .. controls (8,-7) and (9,-6.4) .. (10,-6.5);
\draw (0,-7) .. controls (2,-5) and (4,-9) .. (6,-7);
\draw (6,-7) .. controls (8,-5) and (9,-8) .. (10,-8);
\draw (0,-6) .. controls (2,-10) and (4,-3) .. (6,-6);
\draw (6,-6) .. controls (8,-10) and (9,-4) .. (10,-4);
\draw (0,-6.5) .. controls (2,-7.5) and (3,-6.15) .. (6,-6.6);
\draw (6,-6.6) .. controls (8,-7) and (9,-6.4) .. (10,-6.5);
\node at (10,-8) [inner sep=0pt,label=0:$\scr Y_1(p^m)^{(0,m)}_k$] {};
\node at (10,-6.5) [inner sep=0pt,label=0:$\scr Y_1(p^m)^{(1,m-1)}_k$] {};
\node at (10,-5.25) [inner sep=0pt,label=0:...] {};
\node at (10,-4) [inner sep=0pt,label=0:$\scr Y_1(p^m)^{(m,0)}_k$] {};
\node at (0,-5) [inner sep=0pt,label=0:$\scr Y_1(p^m)_k$] {};
\end{tikzpicture}
\end{center}
The result is that $\scr H_1(p^n)_{k,\mathrm{red}} \cup \scr Y_1(p^m)_k \subseteq \scr H_1(p^n)_k$ looks something like this:

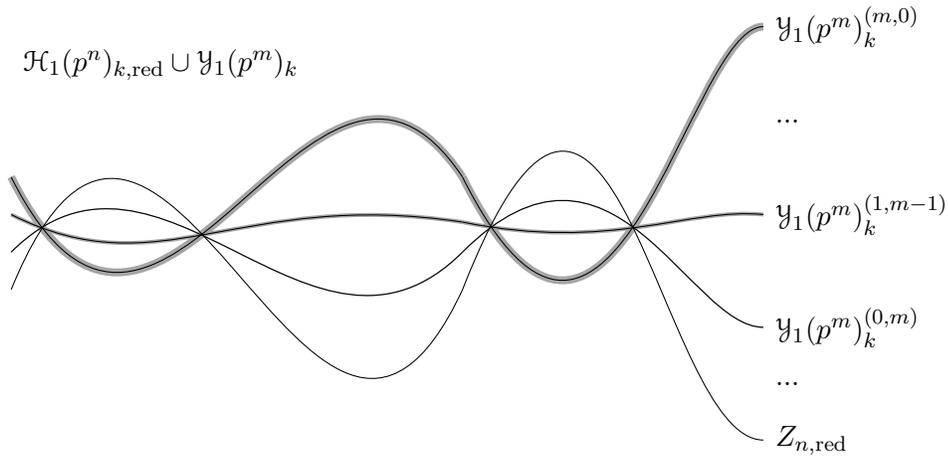
\begin{center}
\begin{tikzpicture}
\clip (-1,-3) rectangle (12.5,3.4);
\draw [gray!70,line width=0.5pt] (0,0) .. controls (2,2) and (4,-2) .. (6,0);
\draw [gray!70,line width=0.5pt] (6,0) .. controls (8,2) and (9,-1) .. (10,-1);
\draw [gray!70,line width=3pt] (0,1) .. controls (2,-3) and (4,4) .. (6,1);
\draw [gray!70,line width=3pt] (6,1) .. controls (8,-3) and (9,3) .. (10,3);
\draw [gray!70,line width=1.5pt] (0,0.5) .. controls (2,-0.5) and (3,0.85) .. (6,0.4);
\draw [gray!70,line width=1.5pt] (6,0.4) .. controls (8,0) and (9,0.6) .. (10,0.5);
\draw (0,0) .. controls (2,2) and (4,-2) .. (6,0);
\draw (6,0) .. controls (8,2) and (9,-1) .. (10,-1);
\draw (0,1) .. controls (2,-3) and (4,4) .. (6,1);
\draw (6,1) .. controls (8,-3) and (9,3) .. (10,3);
\draw (0,0.5) .. controls (2,-0.5) and (3,0.85) .. (6,0.4);
\draw (6,0.4) .. controls (8,0) and (9,0.6) .. (10,0.5);
\draw (0,-0.5) .. controls (2,4.4) and (4,-4.8) .. (6,-0.5);
\draw (6,-0.5) .. controls (8,4.6) and (8.6,-2.5) .. (10,-2.5);
\node at (10,-2.5) [inner sep=0pt,label=0:\textrm{$Z_{n,\mathrm{red}}$}] {};
\node at (10,-1.75) [inner sep=0pt,label=0:...] {};
\node at (10,-1) [inner sep=0pt,label=0:\textrm{$\scr Y_1(p^m)^{(0,m)}_k$}] {};
\node at (10,0.5) [inner sep=0pt,label=0:\textrm{$\scr Y_1(p^m)^{(1,m-1)}_k$}] {};
\node at (10,1.75) [inner sep=0pt,label=0:...] {};
\node at (10,3) [inner sep=0pt,label=0:\textrm{$\scr Y_1(p^m)^{(m,0)}_k$}] {};
\node at (0,2.5) [inner sep=0pt,label=0:\textrm{$\scr H_1(p^n)_{k,\mathrm{red}} \cup \scr Y_1(p^m)_k$}] {};
\end{tikzpicture}
\end{center}
Each $\scr Y_1(p^m)_k$ (for $0\leq m \leq n$) contributes additional nilpotent structure, giving us:

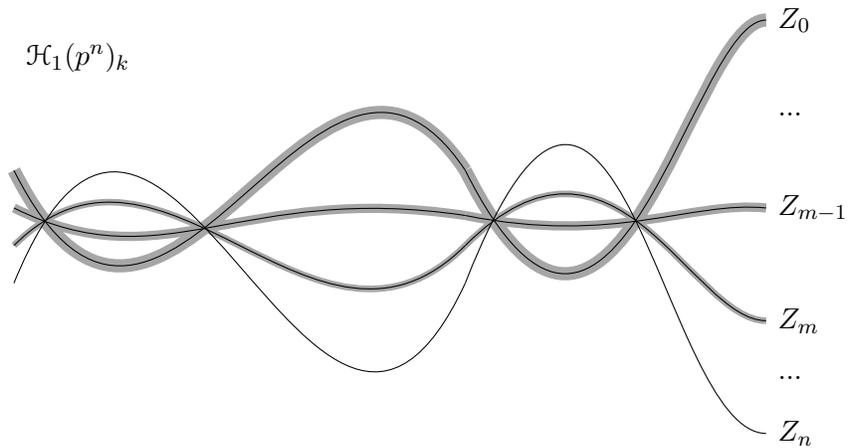
\begin{center}
\begin{tikzpicture}
\clip (-1,-3) rectangle (12,3.4);
\draw [gray!70,line width=2.5pt] (0,0) .. controls (2,2) and (4,-2) .. (6,0);
\draw [gray!70,line width=2.5pt] (6,0) .. controls (8,2) and (9,-1) .. (10,-1);
\draw [gray!70,line width=5pt] (0,1) .. controls (2,-3) and (4,4) .. (6,1);
\draw [gray!70,line width=5pt] (6,1) .. controls (8,-3) and (9,3) .. (10,3);
\draw [gray!70,line width=3.5pt] (0,0.5) .. controls (2,-0.5) and (3,0.85) .. (6,0.4);
\draw [gray!70,line width=3.5pt] (6,0.4) .. controls (8,0) and (9,0.6) .. (10,0.5);
\draw (0,0) .. controls (2,2) and (4,-2) .. (6,0);
\draw (6,0) .. controls (8,2) and (9,-1) .. (10,-1);
\draw (0,1) .. controls (2,-3) and (4,4) .. (6,1);
\draw (6,1) .. controls (8,-3) and (9,3) .. (10,3);
\draw (0,0.5) .. controls (2,-0.5) and (3,0.85) .. (6,0.4);
\draw (6,0.4) .. controls (8,0) and (9,0.6) .. (10,0.5);
\draw (0,-0.5) .. controls (2,4.4) and (4,-4.8) .. (6,-0.5);
\draw (6,-0.5) .. controls (8,4.6) and (8.6,-2.5) .. (10,-2.5);
\node at (10,-2.5) [inner sep=0pt,label=0:\textrm{$Z_n$}] {};
\node at (10,-1.75) [inner sep=0pt,label=0:...] {};
\node at (10,-1) [inner sep=0pt,label=0:\textrm{$Z_m$}] {};
\node at (10,0.5) [inner sep=0pt,label=0:\textrm{$Z_{m-1}$}] {};
\node at (10,1.75) [inner sep=0pt,label=0:...] {};
\node at (10,3) [inner sep=0pt,label=0:\textrm{$Z_0$}] {};
\node at (0,2.5) [inner sep=0pt,label=0:\textrm{$\scr H_1(p^n)_k$}] {};
\end{tikzpicture}
\end{center}

More generally, for arbitrary $N$ we get a closed immersion
\begin{displaymath}
\iota^{(d)}:\scr Y_1(d)\hookrightarrow \scr H_1(N)
\end{displaymath}
for each $d$ dividing $N$, and the resulting map 
\begin{displaymath}
\bigsqcup_{d|N} \scr Y_1(d) \rightarrow \scr H_1(N)
\end{displaymath}
is an isomorphism over $S[1/N]$. It follows immediately from \cite[\S13.5]{KM1} that if $(r,p) = 1$, then for a perfect field $k$ of characteristic $p$, $\scr Y_1(p^mr)_k$ is the disjoint union, with crossings at the supersingular points, of $m+1$ components $\scr Y_1(p^mr)_k^{(m-b,b)}$ ($0\leq b \leq m$), where 
\begin{displaymath}
\scr Y_1(p^mr)_k^{(m-b,b)} := \scr Y_1(r)_k \times_{\overline{\cM}_{1,1,k}} \scr Y_1(p^m)^{(m-b,b)}_k. 
\end{displaymath}
Now let $N = p^nN'$ where $(N',p) = 1$, and for any $r|N'$ let $\scr H_1(N)^r_k \subset \scr H_1(N)_k$ denote the union of the components $\scr Y_1(p^mr)_k$ for $0\leq m \leq n$. In summary:
\begin{corollary}
Let $k$ be a perfect field of characteristic $p$. For any $r$ dividing $N'$, $\scr H_1(N)^r_k$ is the disjoint union, with crossings at the supersingular points, of components $Z^r_b$ for $0 \leq b \leq n$, where 
\begin{displaymath}
Z_b^r = \bigcup_{b\leq m \leq n} \scr Y_1(p^mr)^{(m-b,b)}_k,
\end{displaymath}
identifying each $\scr Y_1(p^mr)^{(m-b,b)}_k$ with a closed substack of $\scr H_1(N)_k$ via $\iota^{(p^mr)}$. Each $\scr Y_1(p^mr)^{(m-b,b)}_k$ is ``set-theoretically identified with $Z_b^r$'' in the sense that $(Z_b)_{\mathrm{red}} = \scr Y_1(p^mr)^{(m-b,b)}_{k,\mathrm{red}}$ as substacks of $\scr H_1(N)_{k,\mathrm{red}}$. $\scr H_1(N)_k$ is the disjoint union of the closed substacks 
\begin{displaymath}
\{\scr H_1(N)^r_k\}_{r|N'}.
\end{displaymath}
\end{corollary}

\subsection*{Closure of $\scr Y_1(N)$ in $\overline{\cK}_{1,1}(\cB \mu_N)$}

We now want to describe the closure of $\scr Y_1(N)$ in $\overline{\cK}_{1,1} (\cB \mu_N)$, as a moduli stack classifying twisted curves with extra structure. This is accomplished in \cite[\S5.2]{ACV} over $\mathbb{Z}[1/N]$; let us briefly recall how this is done. 

\begin{definition}
Let $G$ be a finite group, viewed as a finite \'etale group scheme over $\mathbb{Z}[1/|G|]$. Fix an isomorphism $G^* \cong G$, after adjoining the necessary roots of unity to $\mathbb{Z}[1/|G|]$. $\scr B_{g,n}^{\textrm{tei}}(G)$ is defined as the substack of $\overline{\cK}_{1,1}(\cB G)$ over $\mathbb{Z}[1/|G|]$ whose objects are \textit{twisted Teichm\"uller $G$-structures} on twisted curves. An object of $\overline{\cK}_{g,n}(\cB G)(T)$ is a pair $(\cC,\phi)$, where $\cC/T$ is a $1$-marked genus $1$-twisted stable curve with non-stacky marking, and $\phi$ is a group scheme homomorphism $G \rightarrow \mathrm{Pic}^0_{\cC/T}$. By definition, $(\cC,\phi) \in \scr B_{g,n}^{\textrm{tei}}(G)(T)$ if and only if, whenever $P \rightarrow \cC$ is a $G$-torsor corresponding to $\phi$ (fppf-locally on $T$), the geometric fibers of $P/T$ are connected.
\end{definition}

$\scr B_{g,n}^{\textrm{tei}}(G)$ is naturally a closed substack of $\overline{\cK}_{g,n}(\cB G)$. Working over $\mathbb{Z}[1/N]$, the choice of a primitive $N^{\mathrm{th}}$ root of unity $\zeta_N$ determines an isomorphism of group schemes $G := \mathbb{Z}/(N) \cong \mu_N$ over $\mathbb{Z}[\zeta_N,1/N]$. Applying the resulting isomorphism $\overline{\cK}_{1,1}(\cB G) \cong \overline{\cK}_{1,1}(\cB\mu_N)$, we may view $\scr Y_1(N)$ as a substack of $\overline{\cK}_{1,1}(\cB G)$, and the closure of $\scr Y_1(N)$ in $\overline{\cK}_{1,1}(\cB G)$ over $\mathbb{Z}[\zeta_N,1/N]$ can be shown to be $\scr B_{1,1}^{\textrm{tei}}(G)$ (indeed, this follows from \ref{gamma1} below). Thus:

\begin{corollary}
The closure of $\scr Y_1(N)$ in $\overline{\cK}_{1,1}(\cB \mu_N)$ over $\mathbb{Z}[1/N]$ is the stack whose objects over a scheme $T/\mathbb{Z}[1/N]$ are pairs $(\cC,\phi)$, where $\cC/T$ is a $1$-marked genus $1$ twisted stable curve with non-stacky marking, $\phi:\mathbb{Z}/(N) \rightarrow \mathrm{Pic}^0_{\cC/S}$ is a group scheme homomorphism, and whenever $P \rightarrow \cC$ is a $\mu_N$-torsor corresponding to $\phi$ (fppf-locally on $T$) the geometric fibers of $P/T$ are connected.
\end{corollary}

However, in characteristics dividing $N$, simply requiring the $\mu_N$-torsors to have connected geometric fibers will not give us a moduli stack isomorphic to the closure of $\scr Y_1(N)$ in $\overline{\cK}_{1,1}(\cB\mu_N)$. For example, the group scheme $\mu_{p^n}$ is itself connected over any field of characteristic $p$; so any $\mu_{p^n}$-torsor over an $n$-marked genus $g$ twisted stable curve over a field of characteristic $p$ will automatically be connected. So if the above result held over a base scheme $S \in \mathrm{Sch}/\mathbb{F}_p$, it would say the closure of $\scr Y_1(p^n)$ in $\overline{\cK}_{1,1}(\cB\mu_{p^n})$ is all of $\overline{\cK}'_{1,1}(\cB\mu_{p^n})$ (the substack of $\overline{\cK}_{1,1}(\cB\mu_{p^n})$ where the marking is representable). We will see that this is not true; the closure of $\scr Y_1(p^n)$ turns out to be finite and flat of constant rank $p^{2n}-p^{2n-2}$ over $\overline{\cM}_{1,1}$, while $\overline{\cK}'_{1,1}(\cB\mu_{p^n})$ turns out to be finite and flat of constant rank $p^{2n}$ over $\overline{\cM}_{1,1}$.

A $\mu_N$-torsor over $\cC/S$ determines a group scheme homomorphism $\mathbb{Z}/(N) \rightarrow \mathrm{Pic}^0_{\cC/S}$ (cf. \ref{bigpicard}), and over $\mathbb{Z}[1/N]$ a $\mu_N$-torsor with connected geometric fibers corresponds to a group scheme homomorphism which is injective. We are therefore led to consider Drinfeld structures on our twisted curves:

\begin{definition}
Let $\cC/S$ be a $1$-marked genus $1$ twisted stable curve with no stacky structure at its marking. A \textit{$[\Gamma_1(N)]$-structure} on $\cC$ is a group scheme homomorphism $\phi: \mathbb{Z}/(N) \rightarrow \mathrm{Pic}^0_{\cC/S}$ such that:
\begin{itemize}
  \item the relative effective Cartier divisor 
  \begin{displaymath}
  D := \sum_{a \in \mathbb{Z}/(N)} [\phi(a)]
  \end{displaymath}
  in $\mathrm{Pic}^0_{\cC/S}$ is an $S$-subgroup scheme, and 
  \item for every geometric point $\overline{p} \rightarrow S$, $D_{\overline{p}}$ meets every irreducible component of $(\mathrm{Pic}^0_{\cC/S})_{\overline{p}} = \mathrm{Pic}^0_{\cC_{\overline{p}}/k(\overline{p})}$.
\end{itemize}
Over any scheme $S$, we define $\scr X^{\mathrm{tw}}_1(N)$ to be the substack of $\overline{\cK}_{1,1}(\cB \mu_N)$ whose objects over an $S$-scheme $T$ are pairs $(\cC,\phi) \in \overline{\cK}_{1,1}(\cB\mu_N)(T)$ such that $\phi$ is a $[\Gamma_1(N)]$-structure on $\cC$.
\end{definition}

Note that if $\cC/S$ is a $1$-marked genus $1$ stable curve (non-twisted), so that every geometric fiber of $\cC/S$ is either a smooth elliptic curve or a N\'eron $1$-gon, then this agrees with the definition of a $[\Gamma_1(N)]$-structure on $\cC/S$ as given in \cite[2.4.1]{C} and in Definition \ref{gamma1elliptic} above, identifying a $[\Gamma_1(N)]$-structure $P \in \cC^{\mathrm{sm}}(S)$ with its corresponding group scheme homomorphism $\mathbb{Z}/(N) \rightarrow \cC^{\mathrm{sm}},1\mapsto P$, since $\mathrm{Pic}^0_{\cC/S} \cong \cC^\mathrm{sm}$ in this case.

\begin{theorem}[Restatement of \ref{firstgamma1}]\label{gamma1}
Let $S$ be a scheme and $\scr X^{\mathrm{tw}}_1(N)$ the stack over $S$ classifying $[\Gamma_1(N)]$-structures on $1$-marked genus $1$ twisted stable curves with non-stacky marking. Then $\scr X^{\mathrm{tw}}_1(N)$ is a closed substack of $\overline{\cK}_{1,1}(\cB \mu_N)$, which contains $\scr Y_1(N)$ as an open dense substack.
\end{theorem}
In particular $\scr X^{\mathrm{tw}}_1(N)$ is flat over $S$ with local complete intersection fibers, and is proper and quasi-finite over $\overline{\cM}_{1,1}$.

\begin{remark}
Although this gives a new modular compactification of $\scr Y_1(N)$, it should be noted that the proof of the theorem relies in several places on the proof in \cite{C} that the moduli stack classifying $[\Gamma_1(N)]$-structures on generalized elliptic curves is a proper algebraic stack over $\overline{\cM}_{1,1}$.
\end{remark}

\begin{proof}[Proof of \ref{gamma1}]
The main point is to verify the valuative criterion of properness for $\scr X^{\mathrm{tw}}_1(N)$, which implies $\scr X^{\mathrm{tw}}_1(N)$ is a closed substack of $\overline{\cK}_{1,1}(\cB \mu_N)$. It follows from Lemma \ref{dense} that $\scr Y_1(N)$ is dense.

Let $R$ be a discrete valuation ring with $T := \mathrm{Spec}(R) \in \mathrm{Sch}/S$; write $\eta = \mathrm{Spec}(K)$ for the generic point of $T$ and $s = \mathrm{Spec}(k)$ for the closed point. Let $(\cC_\eta,\phi_\eta) \in \scr X^{\mathrm{tw}}_1(N)(\eta)$, so $\cC_\eta/K$ is a $1$-marked genus $1$ twisted stable curve with non-stacky marking and $\phi_\eta:\mathbb{Z}/(N) \rightarrow \mathrm{Pic}^0_{\cC_\eta/K}$ is a $[\Gamma_1(N)]$-structure on $\cC_\eta$. Since $\overline{\cK}_{1,1}(\cB\mu_N)$ is proper over $S$, there is a discrete valuation ring $R_1$ containing $R$ as a local subring, with corresponding morphism of spectra $T_1 \rightarrow T$ over $S$, such that the pair $(\cC_\eta \times_T T_1, (\phi_\eta)_{T_1})$ extends to a pair $(\cC_{T_1}, \phi_{T_1}) \in \overline{\cK}_{1,1}(\cB \mu_N)(T)$; and whenever such an extension exists, it is unique. Therefore it suffices to show that for some such $R_1$ and $T_1$, there exists a $[\Gamma_1(N)]$-structure $\phi_{T_1}$ on a $1$-marked genus $1$ twisted stable curve $\cC_{T_1}$ extending the $[\Gamma_1(N)]$-structure $(\phi_\eta)_{T_1}$ on $\cC_\eta \times_T T_1$. This is accomplished in Lemmas \ref{completelysmooth}-\ref{mixedchar} below.

\begin{terminology}
In the following lemmas and their proofs, ``base change on $R$'' will refer to replacing $R$ with a discrete valuation ring $R_1$ as above.
\end{terminology}

By \cite[IV.1.6]{DR}, after a base change on $R$ we may assume that the minimal proper regular model of the coarse space $C_\eta$ of $\cC_\eta$ is a generalized elliptic curve. We maintain this assumption for the rest of the proof.

\begin{lemma}\label{completelysmooth}
Suppose $\cC_\eta/K$ is a smooth elliptic curve whose minimal proper regular model over $R$ is smooth, and $\phi_\eta$ is a $[\Gamma_1(N)]$-structure on $\cC_\eta$. Then after base change on $R$, $(\cC_\eta,\phi_\eta)$ extends to a $[\Gamma_1(N)]$-structure $\phi$ on a smooth elliptic curve $\cC/R$.
\end{lemma}
\begin{proof}
This is immediate from the fact that the stack $\scr X_1(N)$ classifying $[\Gamma_1(N)]$-structures on generalized elliptic curves is proper (\cite[1.2.1]{C}).
\end{proof}

\begin{lemma}\label{genericnode}
Suppose $(\cC_\eta,\phi_\eta) \in \scr X^{\mathrm{tw}}_1(N)(\eta)$, such that the coarse space $C_\eta$ of $\cC_\eta$ is singular. Then after base change on $R$, $(\cC_\eta,\phi_\eta)$ extends to $(\cC,\phi) \in \scr X^{\mathrm{tw}}_1(N)(R)$.
\end{lemma}
\begin{proof}
After a base change on $R$, we may assume that $\cC_\eta/K$ is the standard $\mu_d$-stacky N\'eron $1$-gon as in Example \ref{stacky1gon}, for some $d\geq 1$; since $\cC_\eta$ admits a $[\Gamma_1(N)]$-structure and $\mathrm{Pic}^0_{\cC_\eta/K} \cong \mathbb{G}_{m,K} \times \mathbb{Z}/(d)$, we have $d|N$. 

\begin{center}
\begin{tikzpicture}
\clip (-1,-1.6) rectangle (5,1.2);
\draw [name path=curve] (3,1) .. controls (-0.5,-5) and (-0.5,5) .. (3,-1);
\node at (2.35,0) [circle,draw,fill,inner sep=2pt,label=0:$\mu_{d}$] {};
\node at (1.8,-1) [inner sep=0pt,label=270: Standard $\mu_d$-stacky N\'eron $1$-gon] {};
\end{tikzpicture}
\end{center}
Let $\cC/R$ be the standard $\mu_d$-stacky N\'eron $1$-gon over $R$, so $\cC$ is a genus $1$ twisted curve over $R$ extending $\cC_\eta$ (and of course the marking of $\cC_\eta$ extends to $\cC$). After further base change on $R$ we may assume $\mu_N(K) = \mu_N(R)$. Let $\phi_\eta(1) = (\zeta,a) \in \mathbb{G}_m(K) \times \mathbb{Z}/(d)$; since $\phi_\eta$ is a $[\Gamma_1(N)]$-structure on $\cC_\eta$, we have $\zeta \in \mu_N(K)$ and $a \in (\mathbb{Z}/(d))^\times$. Since $\mu_N(K) = \mu_N(R)$, this extends to a section $(\zeta,a) \in \mathbb{G}_m(R) \times \mathbb{Z}/(d)$, determining a group scheme homomorphism $\phi: \mathbb{Z}/(N) \rightarrow \mathbb{G}_{m,R} \times \mathbb{Z}/(d)$ with $\phi(1) = (\zeta,a)$. Since $a \in (\mathbb{Z}/(d))^\times$, $\phi$ is a $[\Gamma_1(N)]$-structure on $\cC$.
\end{proof}

\begin{lemma}\label{purechar}
Suppose $R$ has pure characteristic $p \geq 0$, and $\cC_\eta = E_\eta$ is an elliptic curve over $K$ whose minimal proper regular model over $R$ is not smooth; that is, $E_\eta/K$ is an ordinary elliptic curve with bad reduction. If $\phi_\eta: \mathbb{Z}/(N) \rightarrow \mathrm{Pic}^0_{\cC_\eta/K}$ is a $[\Gamma_1(N)]$-structure on $\cC_\eta$, then after base change on $R$, there exists a pair $(\cC,\phi) \in \scr X^{\mathrm{tw}}_1(N)(R)$ extending $(\cC_\eta,\phi_\eta)$.
\end{lemma}
\begin{proof}
Since for any coprime $N$ and $N'$ we obviously have 
\begin{displaymath}
\overline{\cK}_{1,1}(\cB\mu_{NN'}) \cong \overline{\cK}_{1,1}(\cB\mu_N) \times_{\overline{\cM}_{1,1}} \overline{\cK}_{1,1}(\cB\mu_{N'})
\end{displaymath}
and 
\begin{displaymath}
\scr X^{\mathrm{tw}}_1(NN') \cong \scr X^{\mathrm{tw}}_1(N) \times_{\overline{\cM}_{1,1}} \scr X^{\mathrm{tw}}_1(N'),
\end{displaymath}
it suffices to consider the cases where \textbf{(a)} $N$ is prime to $p$ (including the case $p=0$) and \textbf{(b)} $N = p^n$.

\textbf{(a)} First suppose $(N,p) = 1$ or $p=0$. After base extension on $R$ we may assume that the finite abelian group $E_\eta(K) [N]$ is isomorphic to $(\mathbb{Z}/(N))^2$. After further base extension on $R$, the map $\phi_\eta$ gives us a $\mu_N$-torsor $P_\eta \rightarrow E_\eta$, corresponding to the point $Q_\eta = \phi_\eta(1) \in \mathrm{Pic}^0_{E_\eta/K}(K) \cong E_\eta(K)$ of ``exact order $N$'' in the sense of \cite[\S1.4]{KM1} (and since $(N,p) = 1$, this just says $Q_\eta$ has exact order $N$ as an element of the group $E_\eta(K)$). Let $\scr E/R$ be the minimal proper regular model; after further base extension on $R$ and replacing $\scr E$ with the new minimal proper regular model of $E_\eta$, we may assume $\scr E$ has the structure of a generalized elliptic curve extending that of $E_\eta$ (\cite[IV.1.6]{DR}; also cf. \cite[proof of 3.2.6]{C}). Since $E_\eta(K) = \scr E^{\mathrm{sm}}(R)$, the finite flat $R$-group scheme $\scr E^{\mathrm{sm}}[N]$ has rank $N^2$, hence is \'etale over $\mathrm{Spec}(R)$ since $(N,p) = 1$. Therefore the special fiber $\scr E_s$ is geometrically a N\'eron $mN$-gon for some $m$. Replacing $\scr E$ with its contraction away from the subgroup scheme $\scr E^{\mathrm{sm}}[N] \subset \scr E^{\mathrm{sm}}$ (cf. \cite[IV.1]{DR}), we get a generalized elliptic curve $\scr E/R$ extending $E_\eta$, whose special fiber is geometrically a N\'eron $N$-gon, with $E_\eta(K)[N] \cong \scr E^{\mathrm{sm}}(R) [N]$. 

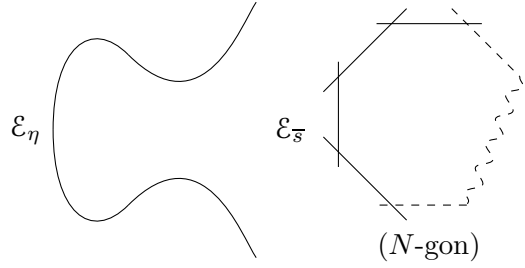
\begin{center}
\begin{tikzpicture}
\clip (-2.2,-0.8) rectangle (6,3);
\draw (1.2,2.7) .. controls (1,2.4) and (0.5,1) .. (-0.5,2);
\draw (-0.5,2) .. controls (-1,2.5) and (-1.5,2.1) .. (-1.5,1);
\draw (1.2,-0.7) .. controls (1,-0.4) and (0.5,1) .. (-0.5,0);
\draw (-0.5,0) .. controls (-1,-0.5) and (-1.5,-0.1) .. (-1.5,1);
\node at (-1.5,1) [inner sep=0pt,label=180: $\scr E_\eta$] {};
\draw (3.2,-0.2) -- (2.093,0.907);
\draw (2.293,0.507) -- (2.293,1.907);
\draw (2.093,1.507) -- (3.2,2.614);
\draw (2.8,2.414) -- (4.2,2.414);
\begin{scope}[dashed]
\draw (3.8,2.614) -- (4.707,1.707);
\draw (4,0) -- (2.8,0);
\draw [decorate,decoration=snake] (4.707,1.707) -- (4,0);
\end{scope}
\node at (2,1) [inner sep=0pt,label=180:$\scr E_{\overline{s}}$] {};
\node at (3.5,-0.2) [inner sep=0pt,label=270:($N$-gon)] {};
\end{tikzpicture}
\end{center}
Note that $Q_\eta$ extends uniquely to a section $Q \in \scr E^{\mathrm{sm}}(R)[N]$ of exact order $N$, which Conrad calls a ``possibly non-ample $[\Gamma_1(N)]$-structure on $\scr E$,'' meaning $Q$ satisfies all the conditions of a $[\Gamma_1(N)]$-structure except that the relative effective Cartier divisor $\sum_{a \in \mathbb{Z}/(N)} [a \cdot Q]$ might not meet every irreducible component of a geometric closed fiber $\scr E_{\overline{s}}$.

Recall that by Corollary \ref{picard}, 
\begin{displaymath}
P_\eta = \underline{\mathrm{Spec}}_{E_\eta} \big( \bigoplus_{a=0}^{N-1} \cL((a\cdot Q_\eta) - (0_{E_\eta}))\big)
\end{displaymath}
with the $\mu_N$-action on $\oplus \cL((a\cdot Q_\eta) - (0_{E_\eta}))$ induced by the $\mathbb{Z}/(N)$-grading. Since $N$ is invertible on $\mathrm{Spec}(R)$, the assumptions of \ref{quotient} are satisfied, so we may identify $P_\eta$ with $E_\eta/\langle Q_\eta \rangle$, with the quotient map $P_\eta \rightarrow E_\eta$ identified with the quotient map $E_\eta/\langle Q_\eta \rangle \rightarrow E_\eta/E_\eta[N] \cong E_\eta$. Then the $\mu_N$-action on $P_\eta = E_\eta/\langle Q_\eta \rangle$ is determined by the group scheme isomorphism $E_\eta[N]/\langle Q_\eta \rangle \cong \mu_N$ induced by the $e_N$-pairing and the choice of $Q_\eta$:
\begin{displaymath}
E_\eta[N]/\langle Q_\eta \rangle = \{Q_\eta\} \times E_\eta[N]/\langle Q_\eta \rangle \hookrightarrow \langle Q_\eta \rangle \times E_\eta[N]/\langle Q_\eta \rangle \stackrel{e_N}{\rightarrow} \mu_N.
\end{displaymath}
By \cite[Theorem 4.1.1]{C}, the $e_N$-pairing on $E_\eta/K$ extends (possibly after further base change on $R$) to a nondegenerate bilinear pairing of finite flat commutative group schemes over $R$ 
\begin{displaymath}
e_N: \scr E^{\mathrm{sm}}[N] \times \scr E^{\mathrm{sm}}[N] \rightarrow \mu_N.
\end{displaymath}
Therefore the isomorphism of group schemes $E_\eta[N]/\langle Q_\eta \rangle \cong \mu_N$ described above may be extended via the same formula to an isomorphism of group schemes $\scr E^{\mathrm{sm}}[N]/\langle Q \rangle \cong \mu_N$.

This isomorphism of group schemes makes $P' := [\scr E/\langle Q \rangle]$ a $\mu_N$-torsor over the twisted curve $\cC':= [\scr E/\scr E^{\mathrm{sm}}[N]]$, extending our original $\mu_N$-torsor over $E_\eta$ ($\cC'$ is indeed a twisted curve: by \cite[Prop. 2.3]{AOV2} we may detect this on the geometric fibers, where it is clear, since the geometric closed fiber of $\cC'$ is a standard $\mu_N$-stacky N\'eron $1$-gon). But $P'$ is not necessarily representable. Indeed, let $d$ be the minimal positive integer such that $d\cdot Q_{\overline{s}}$ lies in the identity component of the geometric closed fiber $\scr E_{\overline{s}}$. Then the coarse space $\overline{P}'_{\overline{s}}$ of $P'_{\overline{s}}$ is a N\'eron $N/d$-gon, and for any geometric point $\overline{q} \rightarrow \overline{P}'_{\overline{s}}$ mapping to a node of $\overline{P}'_{\overline{s}}$, we have $P'_{\overline{s}}\times_{\overline{P}'_{\overline{s}}} \overline{q} \cong (\cB\mu_{N/d})_{k(\overline{q})}$. 

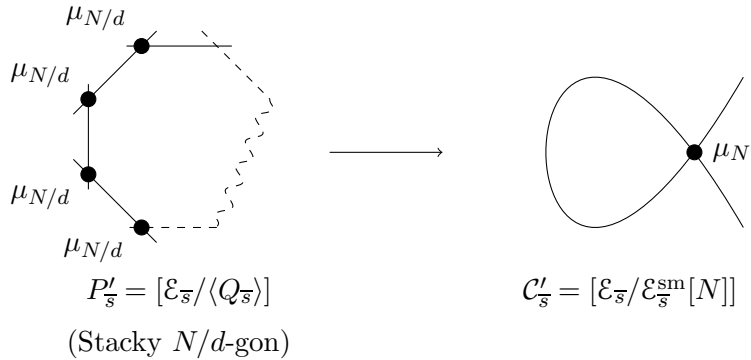
\begin{center}
\begin{tikzpicture}
\clip (-2,2.2) rectangle (9,7);
\draw (0.2,3.8) -- (-0.907,4.907);
\draw (-0.707,4.507) -- (-0.707,5.907);
\draw (-0.907,5.507) -- (0.2,6.614);
\draw (-0.2,6.414) -- (1.2,6.414);
\begin{scope}[dashed]
\draw (0.8,6.614) -- (1.707,5.707);
\draw (1,4) -- (-0.2,4);
\draw [decorate,decoration=snake] (1.707,5.707) -- (1,4);
\end{scope}
\node at (0,4) [circle,draw,fill,inner sep=2pt,label=185:$\mu_{N/d}$] {};
\node at (-0.707,4.707) [circle,draw,fill,inner sep=2pt,label=190:$\mu_{N/d}$] {};
\node at (-0.707,5.707) [circle,draw,fill,inner sep=2pt,label=170:$\mu_{N/d}$] {};
\node at (0,6.414) [circle,draw,fill,inner sep=2pt,label=135:$\mu_{N/d}$] {};
\draw [name path=curve] (8,6) .. controls (4.5,0) and (4.5,10) .. (8,4);
\node at (7.35,5) [circle,draw,fill,inner sep=2pt,label=0:$\mu_{N}$] {};
\node at (6.5,3.5) [inner sep=0pt,label=270:\textrm{$\cC'_{\overline{s}} = [\scr E_{\overline{s}}/\scr E^{\mathrm{sm}}_{\overline{s}}[N]]$}] {};
\node at (0.5,3.5) [inner sep=0pt,label=270:\textrm{$P'_{\overline{s}} = [\scr E_{\overline{s}}/\langle Q_{\overline{s}} \rangle]$}] {};
\node at (0.5,2.8) [inner sep=0pt,label=270:(Stacky $N/d$-gon)] {};
\draw [->] (2.5,5) -- (4,5);
\end{tikzpicture}
\end{center}
Let $P \rightarrow \cC$ be the $\mu_N$-torsor corresponding to the relative coarse moduli space (\cite[Thm. 3.1]{AOV2}) of the map $\cC' \rightarrow \cB\mu_N$ coming from the $\mu_N$-torsor $P' \rightarrow \cC'$. The generic fibers are the same as those of $P' \rightarrow \cC'$, $P_{\overline{s}}$ is a non-stacky N\'eron $N/d$-gon, and $\cC_{\overline{s}}$ is a standard $\mu_d$-stacky N\'eron $1$-gon, say with coarse space $\pi: \cC_{\overline{s}} \rightarrow C_{\overline{s}}$. 

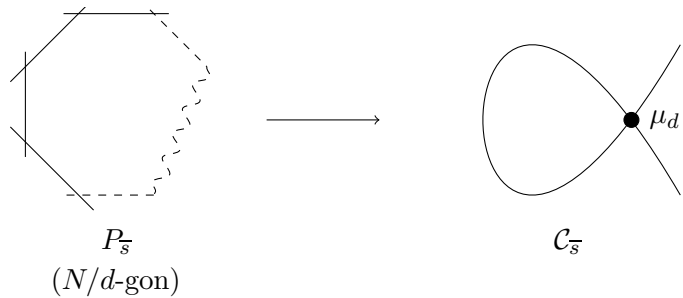
\begin{center}
\begin{tikzpicture}
\clip (-2,2.5) rectangle (9,6.5);
\draw (0.2,3.8) -- (-0.907,4.907);
\draw (-0.707,4.507) -- (-0.707,5.907);
\draw (-0.907,5.507) -- (0.2,6.614);
\draw (-0.2,6.414) -- (1.2,6.414);
\begin{scope}[dashed]
\draw (0.8,6.614) -- (1.707,5.707);
\draw (1,4) -- (-0.2,4);
\draw [decorate,decoration=snake] (1.707,5.707) -- (1,4);
\end{scope}
\draw [name path=curve] (8,6) .. controls (4.5,0) and (4.5,10) .. (8,4);
\node at (7.35,5) [circle,draw,fill,inner sep=2pt,label=0:$\mu_{d}$] {};
\node at (6.5,3.7) [inner sep=0pt,label=270:\textrm{$\cC_{\overline{s}}$}] {};
\node at (0.5,3.7) [inner sep=0pt,label=270:\textrm{$P_{\overline{s}}$}] {};
\node at (0.5,3.2) [inner sep=0pt,label=270:($N/d$-gon)] {};
\draw [->] (2.5,5) -- (4,5);
\end{tikzpicture}
\end{center}
Let $\overline{q} \rightarrow C_{\overline{s}}$ be a geometric point mapping to the node of $C_{\overline{s}}$. Then $\cC_{\overline{s}} \times_{C_{\overline{s}}} \overline{q} \cong (\cB\mu_d)_{k(\overline{q})}$ and $P_{\overline{s}} \times_{C_{\overline{s}}} \overline{q} = \mu_{N/d} \times \overline{q}$ (which as a $k(\overline{q})$-scheme is just $N/d$ disjoint copies of $\overline{q}$ since $(N,p) = 1$), so the resulting $\mu_N$-torsor over $(\cB\mu_d)_{k(\overline{q})}$ corresponds to a generator of $H^1(\cB\mu_d,\mu_N) \cong \mathbb{Z}/(d)$. Therefore, with respect to the decomposition 
\begin{displaymath}
\mathrm{Pic}^0_{\cC_{\overline{s}}/k(\overline{s})} \cong \mathrm{Pic}^0_{C_{\overline{s}}/k(\overline{s})} \times H^0(C_{\overline{s}}, \mathbf{R}^1\pi_*\mathbb{G}_m) \cong \mathbb{G}_m \times \mathbb{Z}/(d),
\end{displaymath}
the class of the $\mu_N$-torsor $P_{\overline{s}} \rightarrow \cC_{\overline{s}}$ projects in the second factor to a generator of $\mathbb{Z}/(d)$, so the map $\phi: \mathbb{Z}/(N) \rightarrow \mathrm{Pic}^0_{\cC/R}$ induced by $P$ is a $[\Gamma_1(N)]$-structure on the twisted curve $\cC/R$ extending the $[\Gamma_1(N)]$-structure $\phi_\eta$ on $\cC_\eta$.

\textbf{(b)} Now suppose $N = p^n$. After base change on $R$, the $[\Gamma_1(N)]$-structure $\phi_\eta$ gives us a $\mu_N$-torsor $P_\eta \rightarrow E_\eta$ corresponding to the point $Q_\eta = \phi_\eta(1) \in \mathrm{Pic}^0_{E_\eta/K}(K) \cong E_\eta(K)$ of ``exact order $N$'' (in the sense of \cite[\S1.4]{KM1}, but not necessarily as an element of the group $E_\eta(K)$). Since $E_\eta/K$ is ordinary, after base change on $R$ we have an isomorphism of group schemes over $K$ 
\begin{displaymath}
E_\eta[N] \cong \mu_N \times \mathbb{Z}/(N),
\end{displaymath}
so $E_\eta(K)[N] \cong \mu_N(K) \times \mathbb{Z}/(N) = \{1\} \times \mathbb{Z}/(N) = \mathbb{Z}/(N)$. The $\mu_N$-torsor $P_\eta \rightarrow E_\eta$ is 
\begin{displaymath}
P_\eta = \underline{\mathrm{Spec}}_{E_\eta} \big( \bigoplus_{a=0}^{N-1} \cL((a\cdot Q_\eta) - (0_{E_\eta}))\big)
\end{displaymath}
with the $\mu_N$-action on $\oplus \cL((Q_\eta) - (0_{E_\eta}))$ induced by the $\mathbb{Z}/(N)$-grading.

Choose $m\geq 0$ minimal such that the image of $p^m \cdot Q_\eta$ in $\mathbb{Z}/(N)$ is zero. So $p^m \cdot Q_\eta = 0_{E_\eta}$ is viewed as a point of ``exact order $p^{n-m}$'' on $E_\eta$. Since $\cL((p^m \cdot Q_\eta) - (0_{E_\eta})) \cong \cO_{E_\eta}$, the corresponding $\mu_{p^{n-m}}$-torsor over $E_\eta$ is trivial. So if $\cC/R$ is a twisted curve extending $E_\eta/K$, the $\mu_{p^{n-m}}$-torsor corresponding to $p^m \cdot Q_\eta$ will automatically extend to the trivial $\mu_{p^{n-m}}$-torsor over $\cC$.

Now view $Q_\eta$ as a point of ``exact order $p^m$'' on $E_\eta/K$. The relative effective Cartier divisor 
\begin{displaymath}
D_\eta := \sum_{a \in \mathbb{Z}/(p^m)} [a\cdot Q_\eta] 
\end{displaymath}
in $E_\eta$ is a subgroup scheme which is \'etale over $K$. So by \ref{quotient}, $E_\eta/D_\eta$ is the underlying scheme of the $\mu_{p^m}$-torsor corresponding to $Q_\eta$. The $\mu_{p^m}$-action on $E_\eta/D_\eta$ is given by the group scheme isomorphism 
\begin{displaymath}
E_\eta[p^m]/D_\eta \cong \mu_{p^m}
\end{displaymath}
induced by the $e_{p^m}$-pairing and the choice of $Q_\eta$:
\begin{displaymath}
E_\eta[p^m]/D_\eta = \{Q_\eta\} \times E_\eta[p^m]/D_\eta \hookrightarrow D_\eta \times E_\eta[p^m]/D_\eta \stackrel{e_{p^m}}{\rightarrow} \mu_{p^m}.
\end{displaymath}

After a base change on $R$ if necessary, let $\scr E/R$ be a generalized elliptic curve extending $E_\eta$, with a $[\Gamma_1(p^m)]$-structure $Q$ on $\scr E$ extending $Q_\eta$. The special fiber $\scr E_s/k$ is geometrically a N\'eron $p^m$-gon, and $\scr E^{\mathrm{sm}}[p^m] \cong \mu_{p^m} \times \mathbb{Z}/(p^m)$. 

\begin{center}
\begin{tikzpicture}
\clip (-2.2,-0.8) rectangle (6,3);
\draw (1.2,2.7) .. controls (1,2.4) and (0.5,1) .. (-0.5,2);
\draw (-0.5,2) .. controls (-1,2.5) and (-1.5,2.1) .. (-1.5,1);
\draw (1.2,-0.7) .. controls (1,-0.4) and (0.5,1) .. (-0.5,0);
\draw (-0.5,0) .. controls (-1,-0.5) and (-1.5,-0.1) .. (-1.5,1);
\node at (-1.5,1) [inner sep=0pt,label=180: $\scr E_\eta$] {};
\draw (3.2,-0.2) -- (2.093,0.907);
\draw (2.293,0.507) -- (2.293,1.907);
\draw (2.093,1.507) -- (3.2,2.614);
\draw (2.8,2.414) -- (4.2,2.414);
\begin{scope}[dashed]
\draw (3.8,2.614) -- (4.707,1.707);
\draw (4,0) -- (2.8,0);
\draw [decorate,decoration=snake] (4.707,1.707) -- (4,0);
\end{scope}
\node at (2,1) [inner sep=0pt,label=180:$\scr E_{\overline{s}}$] {};
\node at (3.5,-0.2) [inner sep=0pt,label=270:($p^m$-gon)] {};
\end{tikzpicture}
\end{center}
By \cite[Theorem 4.1.1]{C}, the $e_{p^m}$-pairing on $E_\eta/K$ extends (possibly after further base change on $R$) to a nondegenerate bilinear pairing of finite flat commutative group schemes over $R$
\begin{displaymath}
e_{p^m}: \scr E^{\mathrm{sm}}[p^m] \times \scr E^{\mathrm{sm}}[p^m] \rightarrow \mu_{p^m}.
\end{displaymath}
Therefore the isomorphism of group schemes $E_\eta[p^m]/D_\eta \cong \mu_{p^m}$ described above may be extended via the same formula to an isomorphism of group schemes $\scr E^{\mathrm{sm}}[p^m]/D \cong \mu_{p^m}$, where $D$ is the relative effective Cartier divisor 
\begin{displaymath}
D = \sum_{a \in \mathbb{Z}/(p^m)} [a \cdot Q]
\end{displaymath}
in $\scr E^{\mathrm{sm}}$. This makes the quotient $[\scr E/D] = \scr E/D$ a representable $\mu_{p^m}$-torsor over the twisted curve $\cC := [\scr E/\scr E^{\mathrm{sm}}[p^m]]$, which extends the given $\mu_{p^m}$-torsor over $E_\eta$.

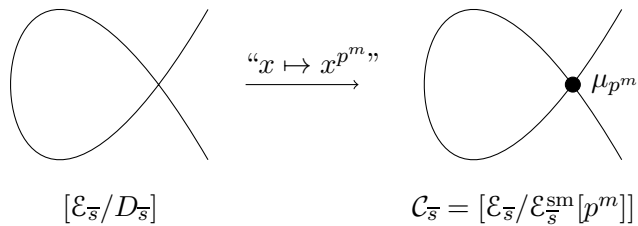
\begin{center}
\begin{tikzpicture}
\clip (-2,3.1) rectangle (9,6.2);
\draw [name path=curve] (2,6) .. controls (-1.5,0) and (-1.5,10) .. (2,4);
\draw [name path=curve] (7.5,6) .. controls (4,0) and (4,10) .. (7.5,4);
\draw [->] (2.5,5) -- (4,5);
\node at (6.85,5) [circle,draw,fill,inner sep=2pt,label=0:$\mu_{p^m}$] {};
\node at (0.7,3.7) [inner sep=0pt,label=270:\textrm{$[\scr E_{\overline{s}}/D_{\overline{s}}]$}] {};
\node at (6.2,3.7) [inner sep=0pt,label=270:\textrm{$\cC_{\overline{s}} = [\scr E_{\overline{s}}/\scr E^{\mathrm{sm}}_{\overline{s}}[p^m]]$}] {};
\node at (3.4,5) [inner sep=0pt,label=90:``$x\mapsto x^{p^m}$''] {};
\end{tikzpicture}
\end{center}
The geometric special fiber $\cC_{\overline{s}}$ is a standard $\mu_{p^m}$-stacky N\'eron $1$-gon, say with coarse space $\pi: \cC_{\overline{s}} \rightarrow C_{\overline{s}}$; $[\scr E_{\overline{s}}/D_{\overline{s}}] = \scr E_{\overline{s}} /D_{\overline{s}}$ is a (non-stacky) N\'eron $1$-gon, and the quotient map $P_{\overline{s}} := \scr E_{\overline{s}} /D_{\overline{s}} \rightarrow \cC_{\overline{s}}$ extends the map $P_{\overline{s}}^{\mathrm{sm}} = \mathbb{G}_m \rightarrow \cC_{\overline{s}}^{\mathrm{sm}} = \mathbb{G}_m,x\mapsto x^{p^m}$. In particular, if $\overline{q} \rightarrow C_{\overline{s}}$ is a geometric point mapping to the node of $C_{\overline{s}}$, then $\cC\times_C \overline{q} \cong (\cB\mu_{p^m})_{k(\overline{q})}$ and $P_{\overline{s}}\times_C \overline{q} = \overline{q}$, so the resulting $\mu_{p^m}$-torsor over $(\cB\mu_{p^m})_{k(\overline{q})}$ corresponds to a generator of $H^1(\cB\mu_{p^m}, \mu_{p^m}) \cong \mathbb{Z}/(p^m)$. Therefore, with respect to the decomposition 
\begin{displaymath}
\mathrm{Pic}^0_{\cC_{\overline{s}}/k(\overline{s})} \cong \mathrm{Pic}^0_{C_{\overline{s}}/k(\overline{s})} \times H^0 (C_{\overline{s}}, \mathbf{R}^1\pi_* \mathbb{G}_m) \cong \mathbb{G}_m \times \mathbb{Z}/(p^m),
\end{displaymath}
the class of the $\mu_{p^m}$-torsor $P_{\overline{s}} \rightarrow \cC_{\overline{s}}$ projects in the second factor to a generator of $\mathbb{Z}/(p^m)$, so the group scheme homomorphism $\phi: \mathbb{Z}/(p^m) \rightarrow \mathrm{Pic}^0_{\cC/R}$ corresponding to $P := [\scr E/D] = \scr E/D$ is a $[\Gamma_1(p^m)]$-structure on the twisted curve $\cC/R$.

Finally, write 
\begin{displaymath}
P = \underline{\mathrm{Spec}}_{\cC}\big( \bigoplus_{a = 0}^{p^m-1} \cL_a \big),
\end{displaymath}
with the grading determined by the $\mu_{p^m}$-action on $P$. Since $P$ extends the $\mu_{p^m}$-torsor on $E_\eta$ determined by $Q_\eta$, we have $\cL_a|_\eta \cong \cL((Q_\eta)-(0_{E_\eta}))$. Then 
\begin{displaymath}
P \times \mu_{p^{n-m}} = \underline{\mathrm{Spec}}_{\cC} \big( \bigoplus_{a = 0}^{p^n-1} \cL_{(\textrm{$a$ mod $p^m$})} \big)
\end{displaymath}
is a $\mu_{p^n}$-torsor over $\cC$ with the $\mu_{p^n}$-action determined by the grading, extending the original $\mu_{p^n}$-torsor over $E_\eta$ and representable since $P$ is representable. Since the group scheme homomorphism $\phi: \mathbb{Z}/(p^m) \rightarrow \mathrm{Pic}^0_{\cC/R}$ corresponding to $P$ is a $[\Gamma_1(p^m)]$-structure on $\cC$, and the group scheme homomorphism $\phi': \mathbb{Z}/(p^n) \rightarrow \mathrm{Pic}^0_{\cC/R}$ corresponding to $P \times \mu_{p^{n-m}}$ is $\phi \circ \pi$ for the canonical projection $\pi: \mathbb{Z}/(p^n) \rightarrow \mathbb{Z}/(p^m)$, it follows immediately that $\phi'$ is a $[\Gamma_1(p^n)]$-structure on $\cC$.
\end{proof}

\begin{lemma}\label{mixedchar}
Suppose $R$ has mixed characteristic $(0,p)$, and $\cC_\eta = E_\eta$ is an elliptic curve over $K$ whose minimal proper regular model over $R$ is not smooth; that is, $E_\eta/K$ is an ordinary elliptic curve with bad reduction. If $\phi_\eta: \mathbb{Z}/(N) \rightarrow \mathrm{Pic}^0_{\cC_\eta/K}$ is a $[\Gamma_1(N)]$-structure on $\cC_\eta$, then after base change on $R$, there exists a pair $(\cC,\phi) \in \scr X^{\mathrm{tw}}_1(N)(R)$ extending $(\cC_\eta,\phi_\eta)$.
\end{lemma}
\begin{proof}
As before, we can restrict to the two separate cases where $(N,p) = 1$ and where $N = p^n$.

\textbf{(a)} If $(p,N) = 1$, the same argument as in part (a) of the proof of Lemma \ref{purechar} carries through.

\textbf{(b)} Suppose $N = p^n$. As in the proof of Lemma \ref{purechar}, after base extension on $R$ we can extend $E_\eta$ to a generalized elliptic curve $\scr E/R$ whose special fiber is geometrically a N\'eron $N$-gon, such that $\scr E^{\mathrm{sm}}(R)[N] \cong E_\eta(K)[N] \cong (\mathbb{Z}/(N))^2$ (the latter isomorphism being a noncanonical isomorphism of abelian groups). After further base change on $R$, the $[\Gamma_1(p^n)]$-structure $\phi_\eta$ gives us a $\mu_N$-torsor $P_\eta \rightarrow E_\eta$, corresponding to $Q_\eta = \phi_\eta(1) \in \mathrm{Pic}^0_{E_\eta/K}(K) \cong E_\eta(K)$ of ``exact order $N$'' in the sense of \cite[\S1.4]{KM1}. $Q_\eta$ extends to a ``possibly non-ample $[\Gamma_1(N)]$-structure'' on $\scr E/R$. Since the special fiber $\scr E_s$ is geometrically a N\'eron $N$-gon, after further base change on $R$ we may assume $\scr E^{\mathrm{sm}}_s[N] \cong \mu_N \times \mathbb{Z}/(N)$, so 
\begin{displaymath}
\scr E^{\mathrm{sm}}_s(k(s))[N] \cong \mu_N(k(s)) \times \mathbb{Z}/(N) = \{1\} \times \mathbb{Z}/(N) \cong \mathbb{Z}/(N).
\end{displaymath}

Suppose $d\geq 1$ is minimal such that $d\cdot Q_s$ maps to $0$ in $\mathbb{Z}/(N)$. Then we can choose $Q^1,Q^2 \in \scr E^{\mathrm{sm}}(R)[N]$ such that:
\begin{itemize}
  \item $Q = Q^1+Q^2$ in $\scr E^{\mathrm{sm}}(R)$; 
  \item $Q^1$ has exact order $d$ in the abelian group $\scr E^{\mathrm{sm}}(R)[N]$, and the relative effective Cartier divisor 
  \begin{displaymath}
    \sum_{a = 0}^{d-1} [a\cdot Q] 
  \end{displaymath}
  in $\scr E^{\mathrm{sm}}$ is \'etale over $\mathrm{Spec}(R)$; and 
  \item $Q^2_s$ maps to $0$ in $\mathbb{Z}/(N)$ via the above isomorphism.
\end{itemize}
In the abelian group $\scr E^{\mathrm{sm}}(R)[N] \cong (\mathbb{Z}/(N))^2$, $Q^2$ has exact order $e$ for some $e|N$. Therefore $Q^1_\eta$ is a $[\Gamma_1(d)]$-structure on $E_\eta$, and $Q^2_\eta$ is a $[\Gamma_1(e)]$-structure on $E_\eta$. They correspond via Lemma \ref{bigpicard} to the $\mu_d$-torsor 
\begin{displaymath}
P^1_\eta := \underline{\mathrm{Spec}}_{E_\eta} \big( \bigoplus_{a = 0}^{d-1} \cL( (a\cdot Q^1_\eta)-(0_{E_\eta}))\big)
\end{displaymath}
and the $\mu_e$-torsor 
\begin{displaymath}
P^2_\eta := \underline{\mathrm{Spec}}_{E_\eta} \big( \bigoplus_{a = 0}^{e-1} \cL( (a\cdot Q^2_\eta)-(0_{E_\eta}))\big)
\end{displaymath}
respectively, with the gradings determining the actions of $\mu_d$ and $\mu_e$. The $\mu_N$-torsor corresponding to $Q_\eta$ via Lemma \ref{bigpicard} is 
\begin{displaymath}
P_\eta := \underline{\mathrm{Spec}}_{E_\eta} \big( \bigoplus_{a = 0}^{N-1} \cL( (a\cdot Q_\eta)-(0_{E_\eta}))\big).
\end{displaymath}
The group law on $E_\eta$ tells us that 
\begin{displaymath}
\big( (Q^1_\eta)-(0_{E_\eta})\big) + \big( (Q^2_\eta)-(0_{E_\eta}) \big) \sim (Q^1_\eta+Q^2_{\eta})-(0_{E_\eta}) = (Q_\eta)-(0_{E_\eta}),
\end{displaymath}
so we conclude that 
\begin{displaymath}
P_\eta = \underline{\mathrm{Spec}}_{E_\eta} \big( \bigoplus_{a = 0}^{N-1} \cL( (a\cdot Q^1_\eta)-(0_{E_\eta})) \otimes \cL( (a\cdot Q^2_\eta)-(0_{E_\eta})) \big),
\end{displaymath}
with the $\mu_N$-action induced by the grading.

Consider the $\mu_d$-torsor $P^1_\eta \rightarrow E_\eta$. As in Lemma \ref{purechar}, factoring the isogeny $[d]$ on $E_\eta$ as $E_\eta \rightarrow E_\eta/\langle Q^1_\eta \rangle \rightarrow E_\eta$, we have $P^1_\eta = E_\eta/\langle Q^1_\eta \rangle$ with $\mu_d$ acting on $P^1_\eta$ through the isomorphism with the group scheme $E_\eta[d]/\langle Q^1_\eta \rangle$ induced via 
\begin{displaymath}
E_\eta[d]/\langle Q^1_\eta \rangle \cong \{Q^1_\eta\} \times E_\eta[d]/\langle Q^1_\eta \rangle \hookrightarrow \langle Q^1_\eta \rangle \times E_\eta[d]/\langle Q^1_\eta \rangle \stackrel{e_d}{\rightarrow} \mu_d.
\end{displaymath}
Let $\scr E_1/R$ (possibly after base change on $R$) be a generalized elliptic curve extending $E_\eta$, whose closed fiber is geometrically a N\'eron $d$-gon, with $\scr E_1^{\mathrm{sm}}(R)[d] \cong E_\eta(K)[d]$. 

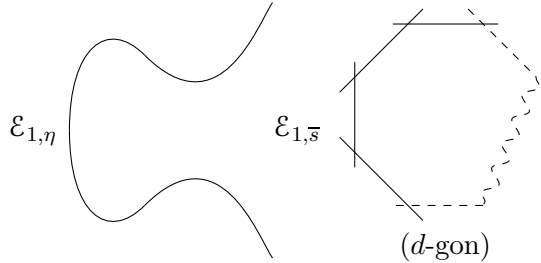
\begin{center}
\begin{tikzpicture}
\clip (-2.4,-0.8) rectangle (6,3);
\draw (1.2,2.7) .. controls (1,2.4) and (0.5,1) .. (-0.5,2);
\draw (-0.5,2) .. controls (-1,2.5) and (-1.5,2.1) .. (-1.5,1);
\draw (1.2,-0.7) .. controls (1,-0.4) and (0.5,1) .. (-0.5,0);
\draw (-0.5,0) .. controls (-1,-0.5) and (-1.5,-0.1) .. (-1.5,1);
\node at (-1.5,1) [inner sep=0pt,label=180: $\scr E_{1,\eta}$] {};
\draw (3.2,-0.2) -- (2.093,0.907);
\draw (2.293,0.507) -- (2.293,1.907);
\draw (2.093,1.507) -- (3.2,2.614);
\draw (2.8,2.414) -- (4.2,2.414);
\begin{scope}[dashed]
\draw (3.8,2.614) -- (4.707,1.707);
\draw (4,0) -- (2.8,0);
\draw [decorate,decoration=snake] (4.707,1.707) -- (4,0);
\end{scope}
\node at (2,1) [inner sep=0pt,label=180:$\scr E_{1,\overline{s}}$] {};
\node at (3.5,-0.2) [inner sep=0pt,label=270:($d$-gon)] {};
\end{tikzpicture}
\end{center}
As in Lemma \ref{purechar}, after further base change on $R$ we may assume that the $e_d$-pairing on $E_\eta$ extends to a nondegenerate bilinear pairing of finite flat commutative group schemes over $R$ 
\begin{displaymath}
e_d: \scr E_1^{\mathrm{sm}}[d] \times \scr E_1^{\mathrm{sm}}[d] \rightarrow \mu_d.
\end{displaymath}
$Q^1_\eta$ extends to a $[\Gamma_1(d)]$-structure $Q^1$ on $\scr E_1/R$, and the relative effective Cartier divisor 
\begin{displaymath}
D:=\sum_{a = 0}^{d-1} [a\cdot Q^1]
\end{displaymath}
in $\scr E_1^{\mathrm{sm}}$ is \'etale over $R$, so via the same formula as above we see that the isomorphism $E_\eta[d]/\langle Q^1_\eta \rangle \cong \mu_d$ extends to a group scheme isomorphism $\scr E_1^{\mathrm{sm}}[d]/\langle Q^1 \rangle \cong \mu_d$. This makes $P^1 := [\scr E_1/\langle Q^1 \rangle] = \scr E_1/\langle Q_1 \rangle$ a representable $\mu_d$-torsor over the twisted curve $\cC := [\scr E/\scr E^{\mathrm{sm}}[d]]$, extending the given $\mu_d$-torsor $P^1_\eta \rightarrow E_\eta$.

The special fiber $\cC_s$ is geometrically a standard $\mu_d$-stacky N\'eron $1$-gon, say with coarse space $\pi: \cC_s \rightarrow C_s$; $[\scr E_{1,s}/D_s] = \scr E_{1,s}/D_s$ is geometrically a (non-stacky) N\'eron $1$-gon, and the quotient map $P^1_s := \scr E_{1,s}/D_s \rightarrow \cC_s$ extends the map $P_s^{1,\mathrm{sm}} = \mathbb{G}_m \rightarrow \cC_s^{\mathrm{sm}} = \mathbb{G}_m,x\mapsto x^d$. 

\begin{center}
\begin{tikzpicture}
\clip (-2,3.1) rectangle (9,6.2);
\draw [name path=curve] (2,6) .. controls (-1.5,0) and (-1.5,10) .. (2,4);
\draw [name path=curve] (7.5,6) .. controls (4,0) and (4,10) .. (7.5,4);
\draw [->] (2.5,5) -- (4,5);
\node at (6.85,5) [circle,draw,fill,inner sep=2pt,label=0:$\mu_{d}$] {};
\node at (0.7,3.7) [inner sep=0pt,label=270:\textrm{$P^1_{\overline{s}} = [\scr E_{1,\overline{s}}/D_{\overline{s}}]$}] {};
\node at (6.2,3.7) [inner sep=0pt,label=270:\textrm{$\cC_{\overline{s}} = [\scr E_{1,\overline{s}}/\scr E^{\mathrm{sm}}_{1,\overline{s}}[d]]$}] {};
\node at (3.4,5) [inner sep=0pt,label=90:``$x\mapsto x^{d}$''] {};
\end{tikzpicture}
\end{center}
In particular, if $\overline{q} \rightarrow C_s$ is a geometric point mapping to the node of $C_s$, then $\cC\times_C \overline{q} \cong (\cB\mu_d)_{k(\overline{q})}$ and $P^1_s\times_C \overline{q} = \overline{q}$, so the resulting $\mu_d$-torsor over $(\cB\mu_d)_{k(\overline{q})}$ corresponds to a generator of $H^1(\cB\mu_d, \mu_d) \cong \mathbb{Z}/(d)$. Therefore, with respect to the decomposition 
\begin{displaymath}
\mathrm{Pic}^0_{\cC_s/k(s)} \cong \mathrm{Pic}^0_{C_s/k(s)} \times H^0 (C_s, \mathbf{R}^1\pi_* \mathbb{G}_m) \cong \mathbb{G}_m \times \mathbb{Z}/(d),
\end{displaymath}
the class of the $\mu_d$-torsor $P^1_s \rightarrow \cC_s$ projects in the second factor to a generator of $\mathbb{Z}/(d)$. Therefore the group scheme homomorphism $\mathbb{Z}/(d) \rightarrow \mathrm{Pic}^0_{\cC/R}$ defined by the $\mu_d$-torsor $P^1 := [\scr E_1/D] = \scr E_1/D$ over $\cC$ is a $[\Gamma_1(d)]$-structure on the twisted curve $\cC/R$.

Next consider the $\mu_e$-torsor $P^2_\eta \rightarrow E_\eta$. Let $\pi: \cC \rightarrow C$ be the coarse space of the twisted curve $\cC/R$ described above. By \cite[IV.1.6]{DR}, after further base change on $R$ we may assume that $C/R$ is a generalized elliptic curve, with structure extending that of $E_\eta$; note that $C_\eta = E_\eta$ and that $C_s$ is geometrically a N\'eron $1$-gon. 

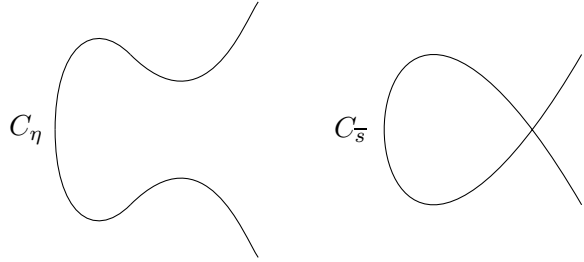
\begin{center}
\begin{tikzpicture}
\clip (-2.2,-0.8) rectangle (6,2.8);
\draw (1.2,2.7) .. controls (1,2.4) and (0.5,1) .. (-0.5,2);
\draw (-0.5,2) .. controls (-1,2.5) and (-1.5,2.1) .. (-1.5,1);
\draw (1.2,-0.7) .. controls (1,-0.4) and (0.5,1) .. (-0.5,0);
\draw (-0.5,0) .. controls (-1,-0.5) and (-1.5,-0.1) .. (-1.5,1);
\node at (-1.5,1) [inner sep=0pt,label=180: $C_\eta$] {};
\draw [name path=curve] (5.5,2) .. controls (2,-4) and (2,6) .. (5.5,0);
\node at (2.8,1) [inner sep=0pt,label=180:$C_{\overline{s}}$] {};
\end{tikzpicture}
\end{center}
We may take the scheme-theoretic closure of the section $Q^2_\eta \in E_\eta(K)$ to get a unique section $Q^2 \in C^{\mathrm{sm}}(R)$; necessarily $e\cdot Q^2 = 0_C$ since $e\cdot Q^2_\eta = 0_{E_\eta}$. The isomorphisms 
\begin{displaymath}
\cL((a\cdot Q^2_\eta)-(0_{E_\eta})) \otimes \cL((b\cdot Q^2_\eta)-(0_{E_\eta})) \cong \cL(((a+b)\cdot Q^2_\eta)-(0_{E_\eta}))
\end{displaymath}
of line bundles on $E_\eta$ extend uniquely to isomorphisms 
\begin{displaymath}
\cL((a\cdot Q^2)-(0_C))\otimes \cL((b\cdot Q^2)-(0_C)) \cong \cL(((a+b)\cdot Q^2) - (0_C))
\end{displaymath}
of line bundles on $C$; therefore the $\mu_e$-torsor 
\begin{displaymath}
P^2_\eta = \underline{\mathrm{Spec}}_{E_\eta} \big( \bigoplus_{a = 0}^{e-1} \cL( (a\cdot Q^2_\eta)-(0_{E_\eta}))\big) 
\end{displaymath}
extends uniquely to a $\mu_e$-torsor 
\begin{displaymath}
P^2 := \underline{\mathrm{Spec}}_C \big( \bigoplus_{a = 0}^{e-1} \cL( (a\cdot Q^2) - (0_C)) \big)
\end{displaymath}
over $C$, with the $\mu_e$-action induced by the grading. Since $\mathrm{Pic}^0_{C/R}$ has irreducible geometric fibers, this is a $[\Gamma_1(e)]$-structure on the generalized elliptic curve $C/R$. Pulling this back to $\cC$ via the coarse moduli space map $\pi: \cC \rightarrow C$, we get a $\mu_e$-torsor $\cP^2 \rightarrow \cC$ extending $P^2_\eta \rightarrow E_\eta$, such that the corresponding map $\phi: \mathbb{Z}/(N) \rightarrow \mathrm{Pic}^0_{\cC/R}$ lands in the identity component of every geometric fiber.

Finally, we return to the $\mu_N$-torsor $P_\eta \rightarrow E_\eta$. Write the $\mu_d$-torsor $P^1 \rightarrow \cC$ as 
\begin{displaymath}
P^1 = \underline{\mathrm{Spec}}_{\cC} \big( \bigoplus_{a = 0}^{d-1} \cL_a \big),
\end{displaymath}
so $\cL_a|_\eta = \cL((a\cdot Q^1_\eta)-(0_{E_\eta}))$; and write the $\mu_e$-torsor $P^2 \rightarrow \cC$ as 
\begin{displaymath}
P^2 = \underline{\mathrm{Spec}}_{\cC} \big( \bigoplus_{a = 0}^{e-1} \cL'_a \big),
\end{displaymath}
so $\cL'_a|_\eta = \cL((a\cdot Q^2_\eta)-(0_{E_\eta}))$. Consider the $\mu_N$-torsor 
\begin{displaymath}
P := \underline{\mathrm{Spec}}_{\cC} \big( \bigoplus_{a = 0}^{N-1} \cL_{\textrm{$a$ mod $d$}} \otimes \cL'_{\textrm{$a$ mod $e$}} \big)
\end{displaymath}
over $\cC$, with the $\mu_N$-action induced by the grading. Since 
\begin{displaymath}
\cL((a\cdot Q^1_\eta)-(0_{E_\eta}))\otimes \cL((a\cdot Q^2_\eta)-(0_{E_\eta})) \cong \cL((a\cdot Q_\eta)-(0_{E_\eta})),
\end{displaymath}
we conclude that $P \rightarrow \cC$ extends the original $\mu_N$-torsor $P_\eta \rightarrow E_\eta$. Furthermore, with respect to the decomposition 
\begin{displaymath}
\mathrm{Pic}^0_{\cC_s/k(s)} \cong \mathrm{Pic}^0_{C_s/k(s)} \times H^0 (C_s, \mathbf{R}^1\pi_* \mathbb{G}_m) \cong \mathbb{G}_m \times \mathbb{Z}/(d),
\end{displaymath}
the line bundle $\cL_1|_s$ projects to a generator of $\mathbb{Z}/(d)$, and the line bundle $\cL'_1|_s$ projects to $0 \in \mathbb{Z}/(d)$; therefore the line bundle $(\cL_1 \otimes \cL'_1)|_s$ projects to a generator of $\mathbb{Z}/(d)$, so the group scheme homomorphism $\mathbb{Z}/(N) \rightarrow \mathrm{Pic}^0_{\cC/R}$ corresponding to the $\mu_N$-torsor $P \rightarrow \cC$ is a $[\Gamma_1(N)]$-structure on $\cC/R$, extending our original $[\Gamma_1(N)]$-structure $\phi_\eta: \mathbb{Z}/(N) \rightarrow \mathrm{Pic}^0_{\cC_\eta/K}$.
\end{proof}

This concludes the proof of our final lemma and thus of Theorem \ref{gamma1}.
\end{proof}

\subsection*{Reduction mod $p$ of $\overline{\cK}'_{1,1}(\cB \mu_N)$}

The analysis of $\scr H_1(N)$ above immediately generalizes to the compactified case. Recall (\ref{globalindex}) that $\overline{\cK}'_{1,1}(\cB \mu_N) \subset \overline{\cK}_{1,1}(\cB \mu_N)$ denotes the closed substack classifying rigidified twisted stable $\mu_N$-covers of twisted curves with non-stacky marking; $\overline{\cK}'_{1,1}(\cB \mu_N)$ is the closure of $\overline{\cK}^\circ_{1,1} (\cB \mu_N) \simeq \scr H_1(N)$ in $\overline{\cK}_{1,1}(\cB \mu_N)$ by Lemma \ref{dense}.

We have a natural closed immersion
\begin{displaymath}
\iota^{(d)}:\scr X^{\mathrm{tw}}_1(d)\hookrightarrow \overline{\cK}'_{1,1}(\cB \mu_N)
\end{displaymath}
for each $d$ dividing $N$, precomposing a map $\mathbb{Z}/(d) \rightarrow \mathrm{Pic}^0_{\cC/S}$ with the canonical projection $\mathbb{Z}/(N) \rightarrow \mathbb{Z}/(d)$. The resulting map 
\begin{displaymath}
\bigsqcup_{d|N} \scr X^{\mathrm{tw}}_1(d) \rightarrow \overline{\cK}'_{1,1}(\cB \mu_N)
\end{displaymath}
is an isomorphism over $S[1/N]$.

\begin{definition}\label{abcyclictwisted}
Let $p$ be prime and $S \in \mathrm{Sch}/\mathbb{F}_p$. Let $\cC/S$ be a $1$-marked genus $1$ twisted stable curve with no stacky structure at its marking. Let $n\geq 1$ and $a,b\geq 0$ with $a+b=n$. A \textit{$[\Gamma_1(p^n)]$-$(a,b)$-cyclic structure} on $\cC$ is a $[\Gamma_1(p^n)]$-structure $\phi: \mathbb{Z}/(p^n) \rightarrow \mathrm{Pic}^0_{\cC/S}$, such that: 
\begin{itemize}
  \item if $S_1 \subset S$ is the maximal Zariski open subset such that $\cC_{S_1} \rightarrow S_1$ is smooth, $\phi_{S_1}: \mathbb{Z}/(p^n) \rightarrow \cC_{S_1}$ is a $[\Gamma_1(p^n)]$-$(a,b)$-cyclic structure in the sense of \cite{KM1}; and 
  \item if $S_2 \subset S$ is the complement of the supersingular locus of $\cC \rightarrow S$, then the relative effective Cartier divisor 
  \begin{displaymath}
  D := \sum_{m=1}^{p^b} [\phi(m)]
  \end{displaymath}
  in $\mathrm{Pic}^0_{\cC_{S_2}/S_2}$ is a subgroup scheme of $\mathrm{Pic}^0_{\cC_{S_2}/S_2}$ which is \'etale over $S_2$. 
\end{itemize}
Over the base scheme $S \in \mathrm{Sch}/\mathbb{F}_p$, we define $\scr X^{\mathrm{tw}}_1(p^n)^{(a,b)} \subset \scr X^{\mathrm{tw}}_1(p^n)$ to be the closed substack associating to $T/S$ the groupoid of pairs $(\cC,\phi)$, where $\cC/S$ is a $1$-marked genus $1$ twisted stable curve with non-stacky marking, and $\phi$ is a $[\Gamma_1(p^n)]$-$(a,b)$-cyclic structure on $\cC$.

If $N = p^nN'$ with $(N',p) = 1$, we define $\scr X^{\mathrm{tw}}_1(N)^{(a,b)} := \scr X^{\mathrm{tw}}_1(N') \times_{\overline{\cM}_{1,1}} \scr X^{\mathrm{tw}}_1(p^n)^{(a,b)}$.
\end{definition}

The same argument as that used to prove Lemma \ref{gamma1comps} immediately gives us:
\begin{lemma}
Let $\cC/S/\mathbb{F}_p$ be a $1$-marked genus $1$ twisted stable curve with non-stacky marking, and let $\phi:\mathbb{Z}/(p^n) \rightarrow \mathrm{Pic}^0_{\cC/S}$ be a $[\Gamma_1(p^n)]$-$(a,b)$-cyclic structure on $\cC$. Then for the canonical projection $\pi: \mathbb{Z}/(p^{n+1}) \twoheadrightarrow \mathbb{Z}/(p^n)$, the composite $\phi \circ \pi: \mathbb{Z}/(p^{n+1}) \rightarrow \mathrm{Pic}^0_{\cC/S}$ is a $[\Gamma_1(p^{n+1})]$-structure on $\cC$, and is $[\Gamma_1(p^{n+1})]$-$(a+1,b)$-cyclic.
\end{lemma}

If $N = p^nN'$ with $(N',p) = 1$, for any $r|N'$ write $\overline{\cK}'_{1,1}(\cB\mu_N)^r \subset \overline{\cK}'_{1,1}(\cB\mu_N)$ for the union of the components $\scr X^{\mathrm{tw}}_1(p^mr)$ over $0 \leq m \leq n$. 

\begin{corollary}\label{twistedgamma1comps}
Let $k$ be a perfect field of characteristic $p$, and let $N = p^nN'$ where $(N',p) = 1$. For any $r|N'$, $\overline{\cK}'_{1,1}(\cB\mu_N)^r_k$ is the disjoint union, with crossings at the supersingular points, of components $\scr Z^r_b$ for $0 \leq b \leq n$, where 
\begin{displaymath}
\scr Z_b^r = \bigcup_{b\leq m \leq n} \scr X^{\mathrm{tw}}_1(p^mr)^{(m-b,b)}_k,
\end{displaymath}
identifying each $\scr X^{\mathrm{tw}}_1(p^mr)^{(m-b,b)}_k$ with a closed substack of $\overline{\cK}'_{1,1}(\cB\mu_N)_k$ via $\iota^{(p^mr)}$. Each substack $\scr X^{\mathrm{tw}}_1(p^mr)^{(m-b,b)}_k$ is ``set-theoretically identified with $\scr Z_b^r$'' in the sense that $(\scr Z^r_b)_{\mathrm{red}} = \scr X_1(p^mr)^{(m-b,b)}_{k,\mathrm{red}}$ as substacks of $\overline{\cK}'_{1,1}(\cB\mu_N)_{k,\mathrm{red}}$. $\overline{\cK}'_{1,1}(\cB\mu_N)_k$ is the disjoint union of the open and closed substacks 
\begin{displaymath}
\{\overline{\cK}'_{1,1}(\cB\mu_N)^r_k\}_{r|N'}.
\end{displaymath}
\end{corollary}

Strictly speaking, to apply the crossings theorem (Theorem \ref{crossings}) to get the above corollary, we need to know that the morphism $\overline{\cK}'_{1,1}(\cB\mu_N) \rightarrow \overline{\cM}_{1,1}$ is finite. This follows from Lemma \ref{finite} below.

The picture in the case $N = p^n$ is essentially the same as the picture for $\overline{\cK}^\circ_{1,1}(\cB\mu_{p^n})_k$ (as discussed after Proposition \ref{y1comps}), except now each component is proper:

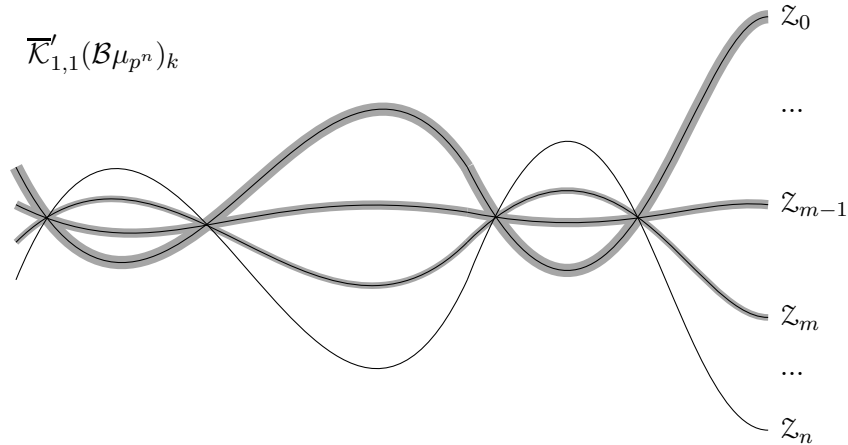
\begin{center}
\begin{tikzpicture}
\clip (-1,-3) rectangle (12,3.4);
\draw [gray!70,line width=2.5pt] (0,0) .. controls (2,2) and (4,-2) .. (6,0);
\draw [gray!70,line width=2.5pt] (6,0) .. controls (8,2) and (9,-1) .. (10,-1);
\draw [gray!70,line width=5pt] (0,1) .. controls (2,-3) and (4,4) .. (6,1);
\draw [gray!70,line width=5pt] (6,1) .. controls (8,-3) and (9,3) .. (10,3);
\draw [gray!70,line width=3.5pt] (0,0.5) .. controls (2,-0.5) and (3,0.85) .. (6,0.4);
\draw [gray!70,line width=3.5pt] (6,0.4) .. controls (8,0) and (9,0.6) .. (10,0.5);
\draw (0,0) .. controls (2,2) and (4,-2) .. (6,0);
\draw (6,0) .. controls (8,2) and (9,-1) .. (10,-1);
\draw (0,1) .. controls (2,-3) and (4,4) .. (6,1);
\draw (6,1) .. controls (8,-3) and (9,3) .. (10,3);
\draw (0,0.5) .. controls (2,-0.5) and (3,0.85) .. (6,0.4);
\draw (6,0.4) .. controls (8,0) and (9,0.6) .. (10,0.5);
\draw (0,-0.5) .. controls (2,4.4) and (4,-4.8) .. (6,-0.5);
\draw (6,-0.5) .. controls (8,4.6) and (8.6,-2.5) .. (10,-2.5);
\node at (10,-2.5) [inner sep=0pt,label=0:\textrm{$\scr Z_n$}] {};
\node at (10,-1.75) [inner sep=0pt,label=0:...] {};
\node at (10,-1) [inner sep=0pt,label=0:\textrm{$\scr Z_m$}] {};
\node at (10,0.5) [inner sep=0pt,label=0:\textrm{$\scr Z_{m-1}$}] {};
\node at (10,1.75) [inner sep=0pt,label=0:...] {};
\node at (10,3) [inner sep=0pt,label=0:\textrm{$\scr Z_0$}] {};
\node at (0,2.5) [inner sep=0pt,label=0:\textrm{$\overline{\cK}'_{1,1}(\cB\mu_{p^n})_k$}] {};
\end{tikzpicture}
\end{center}

\section{Moduli of elliptic curves in $\overline{\cK}_{1,1}(\cB \mu_N^2)$}

\subsection*{Reduction mod $p$ of $\scr H(N)$}

Next we turn our attention to $\overline{\cK}_{1,1} (\cB  \mu_N^2)$, working as before over an arbitrary base scheme $S$. Recall that by Corollary \ref{picard}, the open substack $\overline{\cK}^\circ_{1,1} (\cB  \mu_N^2)$ classifying rigidified twisted $\mu_N^2$-covers of smooth elliptic curves is naturally equivalent to the stack $\scr H(N)$ associating to a scheme $T/S$ the groupoid of pairs $(E,\phi)$ where $E/T$ is an elliptic curve and $\phi:(\mathbb{Z}/(N))^2 \rightarrow E[N]$ is a homomorphism of group schemes over $T$.

For any subgroup $K \leq (\mathbb{Z}/(N))^2$ with corresponding quotient $G_K = (\mathbb{Z}/(N))^2/K$ of $(\mathbb{Z}/(N))^2$, recall (Definition \ref{gkstructures}) that $\scr Y_K$ denotes the moduli stack associating to a scheme $T/S$ the groupoid of pairs $(E,\psi)$, where $E/T$ is an elliptic curve and $\psi: G_K \rightarrow E$ is a $G_K$-structure (in the sense of \cite[\S1.5]{KM1}). So for example, if $G_K \cong \mathbb{Z}/(d)$ for some $d|N$, $\scr Y_K$ is isomorphic to the stack $\scr Y_1(d)$, and if $G_K \cong (\mathbb{Z}/(d))^2$ then $\scr Y_K$ is isomorphic to the stack $\scr Y(d)$ classifying (not necessarily symplectic) $[\Gamma(d)]$-structures on elliptic curves. For every such $K$, we have a closed immersion 
\begin{displaymath}
\iota^K: \scr Y_K  \hookrightarrow \scr H(N), 
\end{displaymath}
given by precomposing a $G_K$-structure $\phi:G_K \rightarrow E[N]$ with the canonical projection $(\mathbb{Z}/(N))^2 \rightarrow G_K$. Together, these give a proper surjection 
\begin{displaymath}
\bigsqcup_{K \leq (\mathbb{Z}/(N))^2} \scr Y_K \rightarrow \scr H(N)
\end{displaymath}
which is an isomorphism over $S[1/N]$.

But in characteristics dividing $N$ this is not an isomorphism. First we consider the case where $N=p^n$ for some prime $p$. Any quotient $G_K = (\mathbb{Z}/(p^n))^2/K$ is isomorphic as an abelian group to $\mathbb{Z}/(p^m) \times \mathbb{Z}/(p^l)$ for some $l\leq m\leq n$. The corresponding moduli stack $\scr Y_K$ classifies $G_K$-structures on elliptic curves, and we saw in Theorem \ref{gkstructurescharp} that over a perfect field $k$ of characteristic $p$, $\scr Y_{K,k}$ is the disjoint union, with crossings at the supersingular points, of substacks $\scr Y^H_{K,k}$ indexed by the set 
\begin{displaymath}
L_K := \{H \leq G_K|\textrm{$H$ and $G_K/H$ are both cyclic}\}.
\end{displaymath}
The component $\scr Y^H_{K,k}$ classifies $G_K$-structures of component label $H$.

Now consider two subgroups $K' \leq K \leq (\mathbb{Z}/(p^n))^2$, and write $\pi: G_{K'} \rightarrow G_K$ for the canonical surjection. If $\phi:G_K \rightarrow E[p^n]$ is a $G_K$-structure on an ordinary elliptic curve $E/T/k$, then $\phi \circ \pi$ may or may not be a $G_{K'}$-structure on $E$. Indeed, we saw in Lemma \ref{gammacomps} that $\phi \circ \pi$ is a $G_{K'}$ structure if and only if $\pi^{-1}(H) \in L_{K'}$, i.e. if and only if $\pi^{-1}(H) \subseteq G_{K'}$ is cyclic.

Consider the set $\{(K,H)|K \leq (\mathbb{Z}/(p^n))^2,H \in L_K\}$. Let $\sim$ be the equivalence relation on this set generated by requiring that $(K,H) \sim (K',H')$ if $K' \leq K$ and $H,H'$ are as above, and let $\Lambda = \{(K,H)\}/\sim$.

\begin{proposition}\label{hcomps}
Let $k$ be a perfect field of characteristic $p$. $\scr H(p^n)_k$ is the disjoint union, with crossings at the supersingular points, of components $\scr H(p^n)^\lambda_k$ for $\lambda \in \Lambda$, where 
\begin{displaymath}
\scr H(p^n)^\lambda_k := \bigcup_{[(K,H)] = \lambda} \scr Y^H_{K,k},
\end{displaymath}
identifying each $\scr Y^H_{K,k}$ with a closed substack of $\scr H(p^n)^\lambda_k$ via $\iota^K$. Each $\scr Y^H_{K,k}$ is ``set-theoretically identified with $\scr H(p^n)^\lambda_k$'' in the sense that $(\scr H(p^n)^\lambda_k)_{\mathrm{red}} = \scr Y^H_{K,k,\mathrm{red}}$ as substacks of $\scr H(p^n)_{k,\mathrm{red}}$ for all $(K,H)$ with $[(K,H)] = \lambda \in \Lambda$.

If $N = p^nN'$ with $(p,N') = 1$, for any $A \leq (\mathbb{Z}/(N))^2$ of order prime to $p$, let $\scr H(N)^A_k \subset \scr H(N)_k$ be the union of the substacks $\scr Y_{K,k}$ for $A \leq K \leq (\mathbb{Z}/(N))^2$ with $(K:A)$ a power of $p$. Then similarly $\scr H(N)^A_k$ is the disjoint union, with crossings at the supersingular points, of components $\scr H(N)^{A,\lambda}_k$ for $\lambda \in \Lambda$, and $\scr H(N)_k$ is the disjoint union of the open and closed substacks $\scr H(N)^A_k$ for $A \leq (\mathbb{Z}/(N))^2$ of order prime to $p$. 
\end{proposition}

As in the case of $\scr H_1(p^n)$, the reduction mod $p$ of $\scr H(p^n)$ has an appealing geometric description. To keep our pictures from getting unreasonably large, we restrict our attention to the case $n=1$. The group $K := (\mathbb{Z}/(p))^2$ has $p+3$ subgroups, namely the entire group $K$, $K_0 := 0$, and $p+1$ subgroups $K_1,...,K_{p+1}$ isomorphic to $\mathbb{Z}/(p)$. The corresponding moduli stacks are 
\begin{align*}
\scr Y_K & = \scr Y(1) \\
\scr Y_{K_0} & = \scr Y(p) \\
\scr Y_{K_i} & \cong \scr Y_1(p)\:\:\:\textrm{for $i = 1,...,p+1$,}
\end{align*}
so we see that over $\mathbb{Z}[1/p]$, $\scr H(p)$ is the disjoint union of $\scr Y(1)$, $\scr Y(p)$ and $p+1$ copies of $\scr Y_1(p)$.

By definition we have 
\begin{align*}
L_K & = \{0\} \\
L_{K_0} & = \{K_1,...,K_{p+1}\} \\
L_{K_i} & = \{G_{K_i},0\}\:\:\:\textrm{for $i = 1,...,p+1$,} 
\end{align*}
where as usual $G_{K_i} = (\mathbb{Z}/(p))^2/K_i$. The set of labels $\Lambda$ is built by putting an equivalence relation on the set consisting of the following pairs:
\begin{align*}
(K,0)\:\:\: & \\
(K_0,K_i)\:\:\: & \textrm{for $i = 1,...,p+1$} \\
(K_i,G_{K_i})\:\:\: & \textrm{for $i = 1,...,p+1$} \\
(K_i,0)\:\:\: & \textrm{for $i = 1,...,p+1$}.
\end{align*}
By definition, working over a perfect field $k$ of characteristic $p$, the pair $(K,0)$ corresponds to $\scr Y(1)_k$; the pair $(K_0,K_i)$ corresponds to the component $\scr Y(p)^{K_i}_k$ of $\scr Y(p)_k$; the pair $(K_i,G_{K_i})$ corresponds to the component $\scr Y_1(p)^{(1,0)}_k$ in $\scr Y_{K_i,k} \cong \scr Y_1(p)_k$; and the pair $(K_i,0)$ corresponds to the component $\scr Y_1(p)^{(0,1)}_k$ in $\scr Y_{K_i,k} \cong \scr Y_1(p)_k$. Unwinding the definition, we see that the equivalence relation defining $\Lambda$ just says that $(K,0) \sim (K_i,G_{K_i})$ for $i = 1,...,p+1$, and that $(K_0,K_i) \sim (K_i,0)$ for $i = 1,...,p+1$. Thus the set of labels is $\Lambda = \{\lambda_0,\lambda_1, ..., \lambda_{p+1}\}$, where 
\begin{align*}
\lambda_0 & = [(K,0)] = [(K_i,G_{K_i})]\:\:\:\textrm{for $i = 1,...,p+1$} \\
\lambda_i & = [(K_0,K_i)] = [(K_i,0)]\:\:\:\:\textrm{for $i = 1,...,p+1$.}
\end{align*}

Visualize $\scr H(p)_{k,\mathrm{red}}$ as follows:

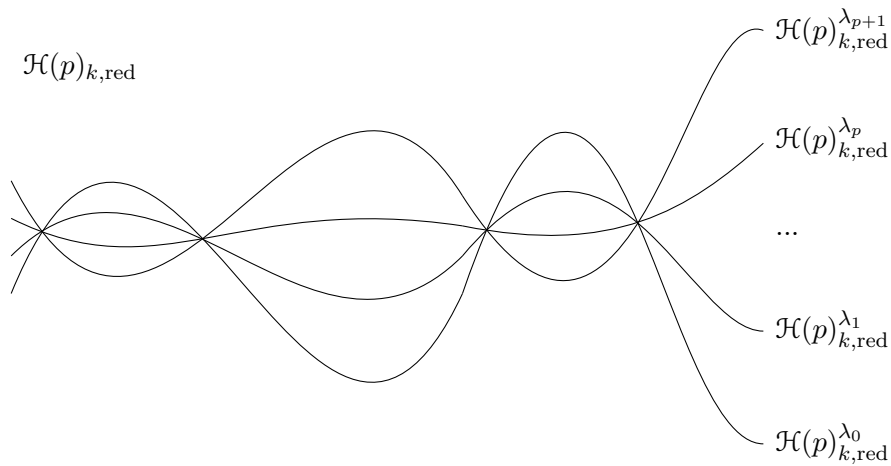
\begin{center}
\begin{tikzpicture}
\clip (-0.5,-2.8) rectangle (12,3.4);
\draw (0,0) .. controls (2,2) and (4,-2) .. (6,0);
\draw (6,0) .. controls (8,2.4) and (9,-1) .. (10,-1);
\draw (0,1) .. controls (2,-3) and (4,4) .. (6,0.8);
\draw (6,0.8) .. controls (8.3,-2.8) and (9,3.4) .. (10,3);
\draw (0,0.5) .. controls (2,-0.5) and (3,0.85) .. (6,0.4);
\draw (6,0.4) .. controls (8,0) and (9,0.6) .. (10,1.5);
\draw (0,-0.5) .. controls (2,4.4) and (4,-4.8) .. (6,-0.5);
\draw (6,-0.5) .. controls (8,5.3) and (8.6,-2.5) .. (10,-2.5);
\node at (10,-2.5) [inner sep=0pt,label=0:\textrm{$\scr H(p)^{\lambda_0}_{k,\mathrm{red}}$}] {};
\node at (10,-1) [inner sep=0pt,label=0:\textrm{$\scr H(p)^{\lambda_1}_{k,\mathrm{red}}$}] {};
\node at (10,0.3) [inner sep=0pt,label=0:...] {};
\node at (10,1.55) [inner sep=0pt,label=0:\textrm{$\scr H(p)^{\lambda_p}_{k,\mathrm{red}}$}] {};
\node at (10,3) [inner sep=0pt,label=0:\textrm{$\scr H(p)^{\lambda_{p+1}}_{k,\mathrm{red}}$}] {};
\node at (0,2.5) [inner sep=0pt,label=0:\textrm{$\scr H(p)_{k,\mathrm{red}}$}] {};
\end{tikzpicture}
\end{center}
The component $\scr H(p)^{\lambda_0}_k$ is ``set-theoretically identified'' with the component $\scr Y(1)_k$ and with the component $\scr Y_1(p)^{(1,0)}_k$ in each copy of $\scr Y_1(p)_k$; each of these contributes additional nilpotent structure to the component $\scr H(p)^{\lambda_0}_k$ (each $\scr Y_1(p)^{(1,0)}_k$ has length $p-1$ over $\scr H(p)^{\lambda_0}_{k,\mathrm{red}} = \scr Y(1)_k$, so $\scr H(p)^{\lambda_0}_k$ has length $(p+1)(p-1)+1 = p^2$ over $\scr H(p)^{\lambda_0}_{k,\mathrm{red}} = \scr Y(1)_k$).

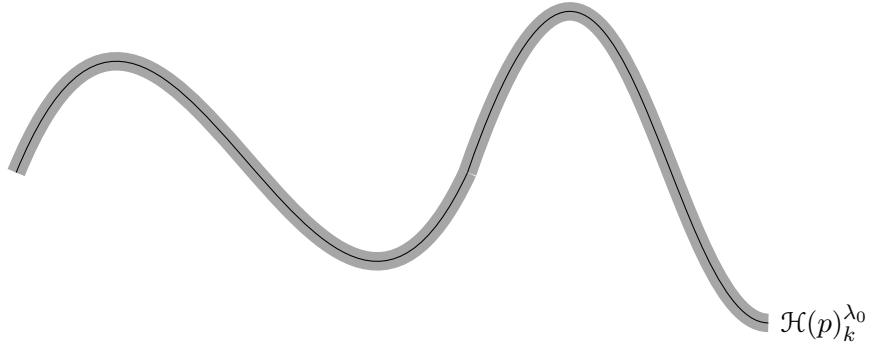
\begin{center}
\begin{tikzpicture}
\clip (-0.5,-2.8) rectangle (12,2);
\draw [gray!70,line width=7pt] (0,-0.5) .. controls (2,4.4) and (4,-4.8) .. (6,-0.5);
\draw [gray!70,line width=7pt] (6,-0.5) .. controls (8,5.3) and (8.6,-2.5) .. (10,-2.5);
\draw (0,-0.5) .. controls (2,4.4) and (4,-4.8) .. (6,-0.5);
\draw (6,-0.5) .. controls (8,5.3) and (8.6,-2.5) .. (10,-2.5);
\node at (10,-2.5) [inner sep=0pt,label=0:\textrm{$\scr H(p)^{\lambda_0}_k$}] {};
\end{tikzpicture}
\end{center}
For $i = 1,...,p+1$, the component $\scr H(p)^{\lambda_i}_k$ is ``set-theoretically identified'' with the component $\scr Y(p)^{K_i}_k$ of $\scr Y(p)_k$, contributing nilpotent structure to $\scr H(p)^{\lambda_i}_k$ (each $\scr Y(p)^{K_i}_k$ has length $p-1$ over $\scr H(p)^{\lambda_i}_{k,\mathrm{red}}$). The component $\scr H(p)^{\lambda_i}_k$ is also ``set-theoretically identified'' with the component $\scr Y_1(p)^{(0,1)}_k$ of $\scr Y_{K_i,k} \cong \scr Y_1(p)_k$; each of these is reduced, adding $1$ to the length of the component $\scr H(p)^{\lambda_i}_k$ over $\scr H(p)^{\lambda_i}_{k,\mathrm{red}}$. Thus $\scr H(p)^{\lambda_i}_k$ has length $p$ over $\scr H(p)^{\lambda_i}_{k,\mathrm{red}}$, which is isomorphic to $\scr Y_1(p)^{(0,1)}_k$ and hence has degree $p^2-p$ over $\scr Y(1)_k$ (cf. \cite[13.5.6]{KM1}). The result is that the component $\scr H(p)^{\lambda_i}_k$ has length $p$ over the underlying reduced stack $\scr Y_1(p)^{(0,1)}_k$, which has degree $p^2-p$ over $\scr Y(1)_k$:

\begin{center}
\begin{tikzpicture}
\clip (-0.5,-2.8) rectangle (12,3.4);
\draw [gray!70,line width=5pt] (0,0) .. controls (2,2) and (4,-2) .. (6,0);
\draw [gray!70,line width=5pt] (6,0) .. controls (8,2.4) and (9,-1) .. (10,-1);
\draw [gray!70,line width=5pt] (0,1) .. controls (2,-3) and (4,4) .. (6,0.8);
\draw [gray!70,line width=5pt] (6,0.8) .. controls (8.3,-2.8) and (9,3.4) .. (10,3);
\draw [gray!70,line width=5pt] (0,0.5) .. controls (2,-0.5) and (3,0.85) .. (6,0.4);
\draw [gray!70,line width=5pt] (6,0.4) .. controls (8,0) and (9,0.6) .. (10,1.5);
\draw (0,0) .. controls (2,2) and (4,-2) .. (6,0);
\draw (6,0) .. controls (8,2.4) and (9,-1) .. (10,-1);
\draw (0,1) .. controls (2,-3) and (4,4) .. (6,0.8);
\draw (6,0.8) .. controls (8.3,-2.8) and (9,3.4) .. (10,3);
\draw (0,0.5) .. controls (2,-0.5) and (3,0.85) .. (6,0.4);
\draw (6,0.4) .. controls (8,0) and (9,0.6) .. (10,1.5);
\node at (10,-1) [inner sep=0pt,label=0:\textrm{$\scr H(p)^{\lambda_1}_k$}] {};
\node at (10,0.3) [inner sep=0pt,label=0:...] {};
\node at (10,1.55) [inner sep=0pt,label=0:\textrm{$\scr H(p)^{\lambda_p}_k$}] {};
\node at (10,3) [inner sep=0pt,label=0:\textrm{$\scr H(p)^{\lambda_{p+1}}_k$}] {};
\end{tikzpicture}
\end{center}
This gives us the following picture of $\scr H(p)_k$:

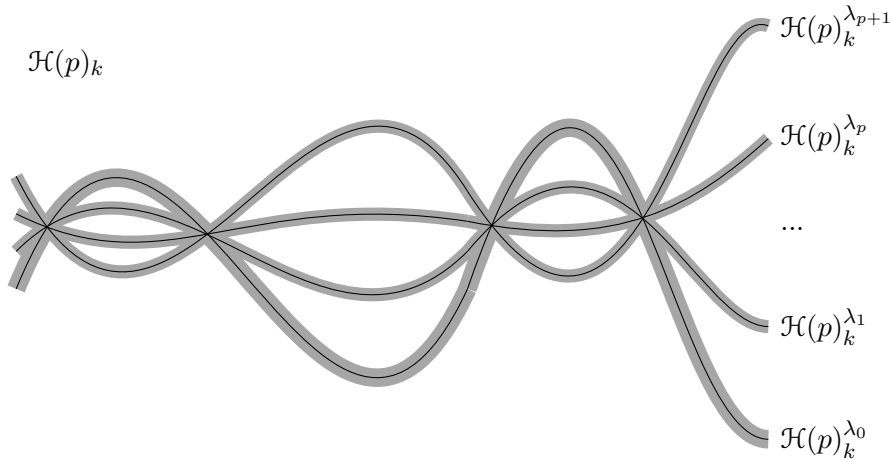
\begin{center}
\begin{tikzpicture}
\clip (-0.5,-2.8) rectangle (12,3.4);
\draw [gray!70,line width=5pt] (0,0) .. controls (2,2) and (4,-2) .. (6,0);
\draw [gray!70,line width=5pt] (6,0) .. controls (8,2.4) and (9,-1) .. (10,-1);
\draw [gray!70,line width=5pt] (0,1) .. controls (2,-3) and (4,4) .. (6,0.8);
\draw [gray!70,line width=5pt] (6,0.8) .. controls (8.3,-2.8) and (9,3.4) .. (10,3);
\draw [gray!70,line width=5pt] (0,0.5) .. controls (2,-0.5) and (3,0.85) .. (6,0.4);
\draw [gray!70,line width=5pt] (6,0.4) .. controls (8,0) and (9,0.6) .. (10,1.5);
\draw [gray!70,line width=7pt] (0,-0.5) .. controls (2,4.4) and (4,-4.8) .. (6,-0.5);
\draw [gray!70,line width=7pt] (6,-0.5) .. controls (8,5.3) and (8.6,-2.5) .. (10,-2.5);
\draw (0,0) .. controls (2,2) and (4,-2) .. (6,0);
\draw (6,0) .. controls (8,2.4) and (9,-1) .. (10,-1);
\draw (0,1) .. controls (2,-3) and (4,4) .. (6,0.8);
\draw (6,0.8) .. controls (8.3,-2.8) and (9,3.4) .. (10,3);
\draw (0,0.5) .. controls (2,-0.5) and (3,0.85) .. (6,0.4);
\draw (6,0.4) .. controls (8,0) and (9,0.6) .. (10,1.5);
\draw (0,-0.5) .. controls (2,4.4) and (4,-4.8) .. (6,-0.5);
\draw (6,-0.5) .. controls (8,5.3) and (8.6,-2.5) .. (10,-2.5);
\node at (10,-2.5) [inner sep=0pt,label=0:\textrm{$\scr H(p)^{\lambda_0}_k$}] {};
\node at (10,-1) [inner sep=0pt,label=0:\textrm{$\scr H(p)^{\lambda_1}_k$}] {};
\node at (10,0.3) [inner sep=0pt,label=0:...] {};
\node at (10,1.55) [inner sep=0pt,label=0:\textrm{$\scr H(p)^{\lambda_p}_k$}] {};
\node at (10,3) [inner sep=0pt,label=0:\textrm{$\scr H(p)^{\lambda_{p+1}}_k$}] {};
\node at (0,2.5) [inner sep=0pt,label=0:\textrm{$\scr H(p)_k$}] {};
\end{tikzpicture}
\end{center}
Note that adding up the lengths calculated in the course of the above construction, we recover the fact that the stack $\scr H(p)$ has length $p^4$ over $\scr Y(1) = \cM_{1,1}$.

\subsection*{Closure of $\scr Y(N)$ in $\overline{\cK}_{1,1}(\cB \mu_N^2)$}

\begin{definition}\label{gammaNtwistedcurve}
Let $\cC/S$ be a $1$-marked genus $1$ twisted stable curve over a scheme $S$, with no stacky structure at its marking. A \textit{$[\Gamma(N)]$-structure} on $\cC$ is a group scheme homomorphism $\phi: (\mathbb{Z}/(N))^2 \rightarrow \mathrm{Pic}^0_{\cC/S}$ such that:
\begin{itemize}
  \item the relative effective Cartier divisor 
  \begin{displaymath}
    D := \sum_{a \in (\mathbb{Z}/(N))^2} [\phi(a)]
  \end{displaymath}
  in $\mathrm{Pic}^0_{\cC/S}$ is an $N$-torsion subgroup scheme, hence $D = \mathrm{Pic}^0_{\cC/S}[N]$; and 
  \item for every geometric point $\overline{p} \rightarrow S$, $D_{\overline{p}}$ meets every irreducible component of $(\mathrm{Pic}^0_{\cC/S})_{\overline{p}} = \mathrm{Pic}^0_{\cC_{\overline{p}}/k(\overline{p})}$.
\end{itemize}
We write $\scr X^{\mathrm{tw}}(N)$ for the substack of $\overline{\cK}_{1,1}(\cB\mu_N^2)$ associating to $T/S$ the groupoid of pairs $(\cC,\phi)$, where $\cC/S$ is a $1$-marked genus $1$ twisted stable curve with non-stacky marking, and $\phi$ is a $[\Gamma(N)]$-structure on $\cC$.
\end{definition}

Note that if $\cC/S$ is a twisted curve admitting a $[\Gamma(N)]$-structure and $\overline{p} \rightarrow S$ is a geometric point such that $C_{\overline{p}}$ is singular, then necessarily $\mathrm{Pic}^0_{\cC_{\overline{p}}/k(\overline{p})} \cong \mathbb{G}_m \times \mathbb{Z}/(N)$, so by Lemma \ref{singulartwistedcurves} $\cC_{\overline{p}}$ is a standard $\mu_N$-stacky N\'eron $1$-gon over $k(\overline{p})$, as in \ref{stacky1gon}.

Applying the methods of our study of $\scr X^{\mathrm{tw}}_1(N)$ to the stack $\scr X^{\mathrm{tw}}(N)$, we have:
\begin{theorem}\label{gamma}
Let $S$ be a scheme and $\scr X^{\mathrm{tw}}(N)$ the stack over $S$ classifying $[\Gamma(N)]$-structures on $1$-marked genus $1$ twisted stable curves with non-stacky marking. Then $\scr X^{\mathrm{tw}}(N)$ is a closed substack of $\overline{\cK}_{1,1}(\cB\mu_N^2)$, which contains $\scr Y(N)$ as an open dense substack.
\end{theorem}
In particular $\scr X^{\mathrm{tw}}(N)$ is flat over $S$ with local complete intersection fibers, and is proper and quasi-finite over $\overline{\cM}_{1,1}$.

\begin{remark}
In \cite{P} a direct proof is given that over $\mathbb{Z}[1/N]$, $\scr X^{\mathrm{tw}}(N)$ agrees with the stack $\scr X(N)$ classifying $[\Gamma(N)]$-structures on generalized elliptic curves; in this case a $\mu_N^2$-torsor over a $1$-marked genus $1$ twisted stable curve with non-stacky marking is in fact a generalized elliptic curve, and it is this observation that gives the desired equivalence. This argument does not generalize to characteristics dividing $N$, because, for example, if $N = p^n$ then the N\'eron $N$-gon in characteristic $p$ (which is a generalized elliptic curve admitting various $[\Gamma(N)]$-structures) cannot be realized as a $\mu_N^2$-torsor over a $1$-marked genus $1$ twisted stable curve.
\end{remark}

\begin{proof}[Proof of \ref{gamma}]
This is proved in exactly the same manner as Theorem \ref{gamma1}. An immediate consequence of Lemma \ref{quotient} is that if $E/K$ is an elliptic curve over a field $K$ in which $N$ is invertible, and $(Q_1,Q_2)$ is a $[\Gamma(N)]$-structure on $E$, then the $\mu_N^2$-torsor 
\begin{displaymath}
P = \underline{\mathrm{Spec}}_E \big( \bigoplus_{a,b \in \mathbb{Z}/(N)} \cL((a\cdot Q_1 + b\cdot Q_2)-(0_E)) \big) 
\end{displaymath}
over $E$ as in \ref{picard} may be identified with $E$ itself, with the quotient map $P \rightarrow E$ corresponding to the isogeny $[N]: E \rightarrow E$. We immediately deduce via the methods of Theorem \ref{gamma1} that over an algebraically closed field $k$ in which $N$ is invertible, the $\mu_N^2$-torsor obtained when we pass to the cusp of $\overline{\cM}_{1,1}$ is a N\'eron $N$-gon $P$ over a standard $\mu_N$-stacky N\'eron $1$-gon $\cC$, with the $\mu_N^2$-action on $P$ induced by some choice of isomorphism $P^{\mathrm{sm}}[N] \cong \mu_N^2$:

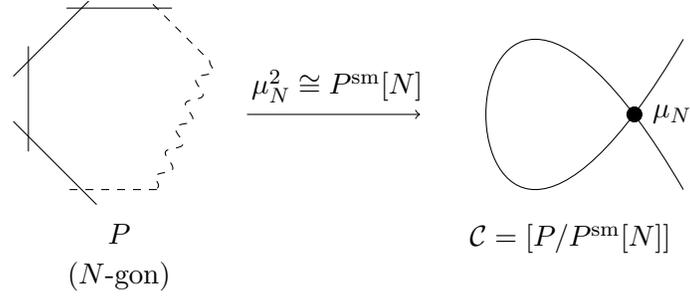
\begin{center}
\begin{tikzpicture}
\clip (-2,2.5) rectangle (9,6.5);
\draw (0.2,3.8) -- (-0.907,4.907);
\draw (-0.707,4.507) -- (-0.707,5.907);
\draw (-0.907,5.507) -- (0.2,6.614);
\draw (-0.2,6.414) -- (1.2,6.414);
\begin{scope}[dashed]
\draw (0.8,6.614) -- (1.707,5.707);
\draw (1,4) -- (-0.2,4);
\draw [decorate,decoration=snake] (1.707,5.707) -- (1,4);
\end{scope}
\draw [name path=curve] (8,6) .. controls (4.5,0) and (4.5,10) .. (8,4);
\node at (7.35,5) [circle,draw,fill,inner sep=2pt,label=0:$\mu_{N}$] {};
\node at (6.5,3.7) [inner sep=0pt,label=270:\textrm{$\cC = [P/P^{\mathrm{sm}}[N]]$}] {};
\node at (0.5,3.7) [inner sep=0pt,label=270:\textrm{$P$}] {};
\node at (0.5,3.2) [inner sep=0pt,label=270:($N$-gon)] {};
\node at (3.4,5) [inner sep=0pt,label=90:\textrm{$\mu_N^2 \cong P^{\mathrm{sm}}[N]$}] {};
\draw [->] (2.2,5) -- (4.5,5);
\end{tikzpicture}
\end{center}
And in the case of $N = p^n$, over an algebraically closed $k$ field of characteristic $p$, the $\mu_{p^n}^2$-torsor $P$ obtained in passing to the cusp of $\overline{\cM}_{1,1}$ may be realized as a trivial $\mu_{p^n}$-torsor over a standard N\'eron $1$-gon $C'$, which in turn is a $\mu_{p^n}$-torsor over a standard $\mu_{p^n}$-stacky N\'eron $1$-gon $\cC = [C'/\mu_{p^n}]$ via the choice of an isomorphism $(C')^{\mathrm{sm}}[p^n] \cong \mu_{p^n}$:

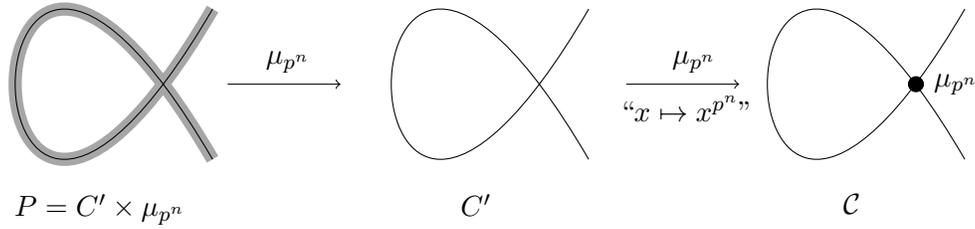
\begin{center}
\begin{tikzpicture}
\clip (-1.5,3) rectangle (13,6.5);
\draw [gray!70,line width=5pt] (2,6) .. controls (-1.5,0) and (-1.5,10) .. (2,4);
\draw (2,6) .. controls (-1.5,0) and (-1.5,10) .. (2,4);
\draw (7,6) .. controls (3.5,0) and (3.5,10) .. (7,4);
\draw (12,6) .. controls (8.5,0) and (8.5,10) .. (12,4);
\node at (11.35,5) [circle,draw,fill,inner sep=2pt,label=0:$\mu_{p^n}$] {};
\node at (5.5,3.7) [inner sep=0pt,label=270:\textrm{$C'$}] {};
\node at (0.5,3.7) [inner sep=0pt,label=270:\textrm{$P = C' \times \mu_{p^n}$}] {};
\node at (3,5) [inner sep=0pt,label=90:$\mu_{p^n}$] {};
\node at (8.4,5) [inner sep=0pt,label=90:$\mu_{p^n}$] {};
\node at (8.3,5) [inner sep=0pt,label=270:\textrm{``$x\mapsto x^{p^n}$''}] {};
\node at (10.5,3.7) [inner sep=0pt,label=270:\textrm{$\cC$}] {};
\draw [->] (2.2,5) -- (3.7,5);
\draw [->] (7.5,5) -- (9,5);
\end{tikzpicture}
\end{center}
In both of the above cases, it is immediately verified that for each of these $\mu_N^2$-torsors, the corresponding group scheme homomorphism $(\mathbb{Z}/(N))^2 \rightarrow \mathrm{Pic}^0_{\cC/k} \cong \mathbb{G}_m \times \mathbb{Z}/(N)$ is a $[\Gamma(N)]$-structure on the standard $\mu_N$-stacky N\'eron $1$-gon $\cC$ in the sense of Definition \ref{gammaNtwistedcurve}, giving the valuative criterion of properness for $\scr X^{\mathrm{tw}}(N)$, hence Theorem \ref{gamma}.
\end{proof}

\subsection*{Reduction mod $p$ of $\overline{\cK}'_{1,1}(\cB \mu_N^2)$}

Recall that $\overline{\cK}'_{1,1}(\cB \mu_N^2) \subset \overline{\cK}_{1,1}(\cB \mu_N^2)$ is the closed substack classifying rigidified twisted stable $\mu_N^2$-covers of twisted curves with non-stacky marking; so $\overline{\cK}'_{1,1}(\cB \mu_N^2)$ is the closure of 
\begin{displaymath}
\overline{\cK}^\circ_{1,1} (\cB \mu_N^2) \simeq \scr H(N) 
\end{displaymath}
in $\overline{\cK}_{1,1}(\cB \mu_N^2)$.

\begin{definition}
Let $\cC/S$ be a $1$-marked genus $1$ twisted stable curve with non-stacky marking, and let $G$ be a 2-generated finite abelian group, say $G \cong \mathbb{Z}/(n_1) \times \mathbb{Z}/(n_2)$, $n_1 \geq n_2$. A \textit{$G$-structure} on $\cC$ is a homomorphism $\phi: G \rightarrow \mathrm{Pic}^0_{\cC/S}$ of group schemes over $S$ such that:
\begin{itemize}
  \item the relative effective Cartier divisor 
  \begin{displaymath}
  D := \sum_{a \in G} [\phi(a)]
  \end{displaymath}
  in $\mathrm{Pic}^0_{\cC/S}$ is an $n_1$-torsion subgroup scheme; and 
  \item for every geometric point $\overline{p} \rightarrow S$, $D_{\overline{p}}$ meets every irreducible component of $(\mathrm{Pic}^0_{\cC/S})_{\overline{p}} = \mathrm{Pic}^0_{\cC_{\overline{p}}/k(\overline{p})}$.
\end{itemize}

For any subgroup $K \leq (\mathbb{Z}/(N))^2$ with corresponding quotient $G_K = (\mathbb{Z}/(N))^2/K$ of $(\mathbb{Z}/(N))^2$, we write $\scr X^{\mathrm{tw}}_K$ for the moduli stack over $S$ associating to a scheme $T/S$ the groupoid of pairs $(\cC,\psi)$, where $\cC/T$ is a $1$-marked genus $1$ twisted stable curve with non-stacky marking, and $\psi: G_K \rightarrow \mathrm{Pic}^0_{\cC/T}$ is a $G_K$-structure on $\mathrm{Pic}^0_{\cC/T}$. 

For a twisted curve $\cC/S/\mathbb{F}_p$, if $N = p^n$ and $G_K \cong \mathbb{Z}/(p^m) \times \mathbb{Z}/(p^l)$ with $m\geq l \geq 1$, we set 
\begin{displaymath}
L_K := \{H \leq G_K|\textrm{$H$ and $G_K/H$ are both cyclic}\},
\end{displaymath}
and for any $H \in L_K$ we say a $G_K$-structure $\phi: G_K \rightarrow \mathrm{Pic}^0_{\cC/S}$ has \textit{component label $H$} if $H$ maps to the kernel of the $n$-fold relative Frobenius $F^n$ on the group scheme $\mathrm{Pic}^0_{\cC/S}$ over $S$, and the resulting group scheme homomorphism $G_K/H \rightarrow \mathrm{Pic}^0_{\cC/S}[p^n]/\mathrm{ker}(F^n)$ is a $G_K/H$-structure in the sense of \cite[\S1.5]{KM1}.

If $G_K \cong \mathbb{Z}/(p^m)$ (i.e. $l=0$), then $\scr X^{\mathrm{tw}}_K \cong \scr X^{\mathrm{tw}}_1(p^m)$, and for $H \cong \mathbb{Z}/(p^a) \in L_K$ we define $\scr X^{\mathrm{tw},H}_K \subset \scr X^{\mathrm{tw}}_K$ to be the substack $\scr X^{\mathrm{tw}}_1(p^m)^{(a,m-a)} \subset \scr X^{\mathrm{tw}}_1(p^m)$ as in Definition \ref{abcyclictwisted}. We still say that $\scr X^{\mathrm{tw},H}_K$ classifies $G_K$-structures of \textit{component label $H$}.
\end{definition}

So for example, if $G_K \cong \mathbb{Z}/(d)$ for some $d|N$, $\scr X^{\mathrm{tw}}_K$ is isomorphic to the stack $\scr X^{\mathrm{tw}}_1(d)$, and if $G_K \cong (\mathbb{Z}/(d))^2$ then $\scr X^{\mathrm{tw}}_K$ is isomorphic to the stack $\scr X^{\mathrm{tw}}(d)$. For every such $K$, we have a closed immersion 
\begin{displaymath}
\iota^K: \scr X^{\mathrm{tw}}_K  \hookrightarrow \overline{\cK}'_{1,1}( \cB \mu_N^2), 
\end{displaymath}
given by precomposing a $G_K$-structure $\phi:G_K \rightarrow \mathrm{Pic}^0_{\cC/T}$ with the canonical projection $(\mathbb{Z}/(N))^2 \rightarrow G_K$. Together, these give a proper surjection 
\begin{displaymath}
\bigsqcup_{K \leq (\mathbb{Z}/(N))^2} \scr X^{\mathrm{tw}}_K \rightarrow \overline{\cK}'_{1,1}(\cB\mu_N^2)
\end{displaymath}
which is an isomorphism over $S[1/N]$.

Let $k$ be a perfect field of characteristic $p$. The same argument proving Theorem \ref{gammacomps} immediately gives us:

\begin{corollary}
If $K \leq (\mathbb{Z}/(p^n))^2$ such that $G_K := (\mathbb{Z}/(p^n))^2/K \cong \mathbb{Z}/(p^m) \times \mathbb{Z}/(p^l)$ with $l\leq m \leq n$, then $\scr X^{\mathrm{tw}}_{K,k}$ is the disjoint union, with crossings at the supersingular points, of closed substacks $\scr X^{\mathrm{tw},H}_{K,k}$ for $H \in L_K$. $\scr X^{\mathrm{tw},H}_{K,k}$ classifies $G_K$-structures with component label $H$. 
\end{corollary}

As before, we let $\Lambda$ denote the set $\{(K,H)|K\leq (\mathbb{Z}/(p^n))^2,H \in L_K\}$, modulo the equivalence relation generated by declaring $(K,H) \sim (K',\pi^{-1}(H))$ whenever $K' \leq K$ with corresponding quotient map $\pi:G_{K'} \rightarrow G_K$ such that $\pi^{-1}(H) \in L_{K'}$. We conclude:
\begin{corollary}
Let $k$ be a perfect field of characteristic $p$. $\overline{\cK}'_{1,1}(\cB \mu_{p^n}^2)_{k}$ is the disjoint union, with crossings at the supersingular points, of components $\overline{\cK}'_{1,1}(\cB\mu_{p^n}^2)^\lambda_k$ for $\lambda \in \Lambda$, where 
\begin{displaymath}
\overline{\cK}'_{1,1}(\cB\mu_{p^n}^2)^\lambda_k := \bigcup_{[(K,H)] = \lambda} \scr X^{\mathrm{tw},H}_{K,k},
\end{displaymath}
identifying each $\scr X^{\mathrm{tw},H}_{K,k}$ with a closed substack of $\overline{\cK}'_{1,1}(\cB\mu_{p^n}^2)^\lambda_k$ via $\iota^K$. Each $\scr X^{\mathrm{tw},H}_{K,k}$ is ``set-theoretically identified with $\overline{\cK}'_{1,1}(\cB\mu_{p^n}^2)^\lambda_k$'' in the sense that 
\begin{displaymath}
\overline{\cK}'_{1,1}(\cB\mu_{p^n}^2)^\lambda_{k,\mathrm{red}} = \scr X^{\mathrm{tw},H}_{K,k,\mathrm{red}} 
\end{displaymath}
as substacks of $\overline{\cK}'_{1,1}(\cB\mu_{p^n})_{k,\mathrm{red}}$ for all $(K,H)$ with $[(K,H)] = \lambda \in \Lambda$.

If $N = p^nN'$ with $(p,N') = 1$, for any $A \leq (\mathbb{Z}/(N))^2$ of order prime to $p$, let 
\begin{displaymath}
\overline{\cK}'_{1,1}(\cB\mu_N^2)^A_k \subset \overline{\cK}'_{1,1}(\cB\mu_N^2)_k
\end{displaymath}
be the union of the substacks $\scr X^{\mathrm{tw}}_{K,k}$ for $A \leq K \leq (\mathbb{Z}/(N))^2$ with $(K:A)$ a power of $p$. Then similarly $\overline{\cK}'_{1,1}(\cB\mu_N^2)^A_k$ is the disjoint union, with crossings at the supersingular points, of components $\overline{\cK}'_{1,1}(\cB\mu_N^2)^{A,\lambda}_k$ for $\lambda \in \Lambda$, and $\overline{\cK}'_{1,1}(\cB\mu_N^2)_k$ is the disjoint union of the open and closed substacks $\overline{\cK}'_{1,1}(\cB\mu_N^2)^A_k$ for $A \leq (\mathbb{Z}/(N))^2$ of order prime to $p$.
\end{corollary}

The picture in the case $N = p$ is essentially the same as the picture for $\overline{\cK}^\circ_{1,1}(\cB\mu_p^2)_k \simeq \scr H(p)_k$ (as discussed after Proposition \ref{hcomps}), except now each component is proper:

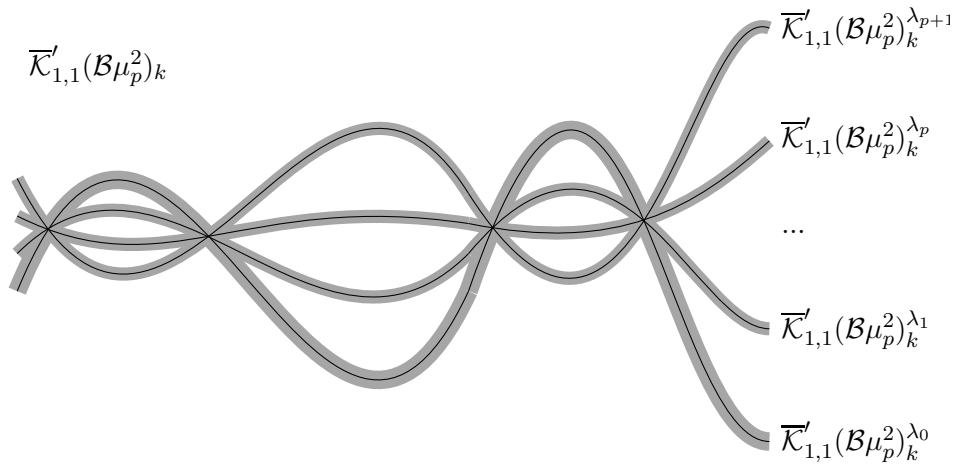
\begin{center}
\begin{tikzpicture}
\clip (-0.5,-2.8) rectangle (12.4,3.4);
\draw [gray!70,line width=5pt] (0,0) .. controls (2,2) and (4,-2) .. (6,0);
\draw [gray!70,line width=5pt] (6,0) .. controls (8,2.4) and (9,-1) .. (10,-1);
\draw [gray!70,line width=5pt] (0,1) .. controls (2,-3) and (4,4) .. (6,0.8);
\draw [gray!70,line width=5pt] (6,0.8) .. controls (8.3,-2.8) and (9,3.4) .. (10,3);
\draw [gray!70,line width=5pt] (0,0.5) .. controls (2,-0.5) and (3,0.85) .. (6,0.4);
\draw [gray!70,line width=5pt] (6,0.4) .. controls (8,0) and (9,0.6) .. (10,1.5);
\draw [gray!70,line width=7pt] (0,-0.5) .. controls (2,4.4) and (4,-4.8) .. (6,-0.5);
\draw [gray!70,line width=7pt] (6,-0.5) .. controls (8,5.3) and (8.6,-2.5) .. (10,-2.5);
\draw (0,0) .. controls (2,2) and (4,-2) .. (6,0);
\draw (6,0) .. controls (8,2.4) and (9,-1) .. (10,-1);
\draw (0,1) .. controls (2,-3) and (4,4) .. (6,0.8);
\draw (6,0.8) .. controls (8.3,-2.8) and (9,3.4) .. (10,3);
\draw (0,0.5) .. controls (2,-0.5) and (3,0.85) .. (6,0.4);
\draw (6,0.4) .. controls (8,0) and (9,0.6) .. (10,1.5);
\draw (0,-0.5) .. controls (2,4.4) and (4,-4.8) .. (6,-0.5);
\draw (6,-0.5) .. controls (8,5.3) and (8.6,-2.5) .. (10,-2.5);
\node at (10,-2.5) [inner sep=0pt,label=0:\textrm{$\overline{\cK}'_{1,1}(\cB\mu_p^2)^{\lambda_0}_k$}] {};
\node at (10,-1) [inner sep=0pt,label=0:\textrm{$\overline{\cK}'_{1,1}(\cB\mu_p^2)^{\lambda_1}_k$}] {};
\node at (10,0.3) [inner sep=0pt,label=0:...] {};
\node at (10,1.55) [inner sep=0pt,label=0:\textrm{$\overline{\cK}'_{1,1}(\cB\mu_p^2)^{\lambda_p}_k$}] {};
\node at (10,3) [inner sep=0pt,label=0:\textrm{$\overline{\cK}'_{1,1}(\cB\mu_p^2)^{\lambda_{p+1}}_k$}] {};
\node at (0,2.5) [inner sep=0pt,label=0:\textrm{$\overline{\cK}'_{1,1}(\cB\mu_p^2)_k$}] {};
\end{tikzpicture}
\end{center}

\section{Comparison with the classical moduli stacks}\label{compare}
As promised, we verify that the moduli stacks $\scr X_1^{\mathrm{tw}}(N)$ and $\scr X^{\mathrm{tw}}(N)$ are isomorphic to the corresponding classical moduli stacks, justifying the claim in \cite{AOV2} that we have recovered the Katz-Mazur regular models:
\begin{theorem}[Restatement of \ref{firstclassical}]\label{classical}
Over the base $S = \mathrm{Spec}(\mathbb{Z})$, there is a canonical isomorphism of algebraic stacks $\scr X_1^{\mathrm{tw}}(N) \cong \scr X_1(N)$ extending the identity map on $\scr Y_1(N)$, and a canonical isomorphism of algebraic stacks $\scr X^{\mathrm{tw}}(N) \cong \scr X(N)$ extending the identity map on $\scr Y(N)$.
\end{theorem}
We prove this after some preliminary results; the main point is to demonstrate that $\scr X^{\mathrm{tw}}_1(N)$ and $\scr X^{\mathrm{tw}}(N)$ are normal.

\begin{proposition}\label{gamma1auts}
The morphism $\pi:\scr X^{\mathrm{tw}}_1(N) \rightarrow \overline{\cM}_{1,1}$ sending $(\cC,\phi)$ to the coarse space $C$ of $\cC$ is a representable morphism of stacks. In particular, $\scr X^{\mathrm{tw}}_1(N)$ is Deligne-Mumford.

Similarly the natural morphism $\scr X^{\mathrm{tw}}(N) \rightarrow \overline{\cM}_{1,1}$ is representable, hence $\scr X^{\mathrm{tw}}(N)$ is Deligne-Mumford.
\end{proposition}

\begin{proof}
We have already seen that $\scr X^{\mathrm{tw}}_1(N)$ is an algebraic stack, so it suffices to show that for any object $(\cC,\phi) \in \scr X^{\mathrm{tw}}_1(N)(k)$ with $k$ an algebraically closed field, the natural map $\mathrm{Aut}(\cC,\phi) \rightarrow \mathrm{Aut}(C)$ is a monomorphism of group schemes. Here $C/k$ is the coarse space of $\cC$, and automorphisms are required to preserve the marking. It is obvious that $\mathrm{Aut}(\cC,\phi) \rightarrow \mathrm{Aut}(C)$ is a monomorphism if $\cC = C$ is a smooth elliptic curve over $k$, so by Lemma \ref{singulartwistedcurves} we reduce to the case where $\cC/k$ is a standard $\mu_d$-stacky N\'eron $1$-gon for some $d|N$.

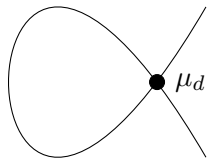
\begin{center}
\begin{tikzpicture}
\clip (-1,-1) rectangle (5,1.2);
\draw [name path=curve] (3,1) .. controls (-0.5,-5) and (-0.5,5) .. (3,-1);
\node at (2.35,0) [circle,draw,fill,inner sep=2pt,label=0:$\mu_{d}$] {};
\end{tikzpicture}
\end{center}

In this case an automorphism of $\cC$ is an automorphism of the coarse space $C$, together with an automorphism of the $\mu_d$-gerbe in $\cC$ lying over the node of $C$. Thus 
\begin{displaymath}
\mathrm{Aut}(\cC) \cong \mathrm{Aut}(C) \times \mathrm{Aut}(\cB \mu_{d,k}).
\end{displaymath}
The only nontrivial automorphism of $C$ preserving the marked point $1 \in C$ is the automorphism $\iota: C \rightarrow C$ induced by the inversion automorphism of $\mathbb{G}_m$. We have $\mathrm{Aut}(\cB \mu_d) \cong \mu_d$, and the automorphism of $\mathrm{Pic}^0_{\cC/k} \cong \mathbb{G}_m \times \mathbb{Z}/(d)$ induced by $(0, \zeta) \in \mathrm{Aut}(\cC) \cong \langle \iota \rangle \times \mu_d$ sends $(\eta, a)$ to $(\zeta^a \eta, a)$. Since $\phi: \mathbb{Z}/(N) \rightarrow \mathrm{Pic}^0_{\cC/k}$ meets every component, the only automorphisms of $\cC$ than can possibly preserve $\phi$ are the automorphisms $\langle \iota \rangle \times \{0\} \subset \mathrm{Aut}(\cC)$ (cf. \cite[proof of 3.1.8]{C}). Thus $\mathrm{Aut}(\cC,\phi) \subset \langle \iota \rangle \times \{0\} \cong \mathrm{Aut}(C)$.

The same argument applies to $\scr X^{\mathrm{tw}}(N)$.
\end{proof}

\begin{corollary}\label{finite}
$\scr X^{\mathrm{tw}}_1(N)$ and $\scr X^{\mathrm{tw}}(N)$ are finite over $\overline{\cM}_{1,1}$.
\end{corollary}

\begin{proof}
By Theorems \ref{gamma1} and \ref{gamma}, the natural maps $\scr X^{\mathrm{tw}}_1(N) \rightarrow \overline{\cM}_{1,1}$ and $\scr X^{\mathrm{tw}}(N) \rightarrow \overline{\cM}_{1,1}$ are proper and quasi-finite; since they are also representable, they are finite.
\end{proof}

Let $\infty \hookrightarrow \overline{\cM}_{1,1}$ denote the closed substack classifying $1$-marked genus $1$ stable curves whose geometric fibers are singular. Let $\scr X^{\mathrm{tw}}_1(N)^\infty = \scr X^{\mathrm{tw}}_1(N) \times_{\overline{\cM}_{1,1}} \infty$ and $\scr X^{\mathrm{tw}}(N)^\infty = \scr X^{\mathrm{tw}}(N) \times_{\overline{\cM}_{1,1}} \infty$. Exactly analogously to \cite[2.1.12]{C}, formation of these closed substacks is compatible with arbitrary base change.

\begin{proposition}\label{CM}
The proper flat morphisms $\scr X^{\mathrm{tw}}_1(N) \rightarrow \mathrm{Spec} (\mathbb{Z})$ and $\scr X^{\mathrm{tw}}(N) \rightarrow \mathrm{Spec}(\mathbb{Z})$ are Cohen-Macaulay (of pure relative dimension $1$).
\end{proposition}
\begin{proof}
Let $\scr X$ denote $\scr X^{\mathrm{tw}}_1(N)$ or $\scr X^{\mathrm{tw}}(N)$. The canonical morphism $\scr X \rightarrow \overline{\cM}_{1,1}$ is finite (by Corollary \ref{finite}) and flat (by Theorems \ref{gamma1} and \ref{gamma}), and the structural morphism $\overline{\cM}_{1,1} \rightarrow \mathrm{Spec}(\mathbb{Z})$ is Cohen-Macaulay (cf. \cite[3.3.1]{C}), so by \cite[2.7.9, Cor. 3]{B}, the composite $\scr X \rightarrow \mathrm{Spec}(\mathbb{Z})$ is Cohen-Macaulay.
\end{proof}

\begin{lemma}\label{cartier}
$\scr X^{\mathrm{tw}}_1(N)^\infty$ and $\scr X^{\mathrm{tw}}(N)^\infty$ are relative effective Cartier divisors over $\mathrm{Spec}(\mathbb{Z})$ in $\scr X^{\mathrm{tw}}_1(N)$ and $\scr X^{\mathrm{tw}}(N)$ respectively.
\end{lemma}
\begin{proof}
Here we are using the notion of a Cartier divisor on a Deligne-Mumford stack, cf. \cite[Ch. XIII]{ACG}. For $\scr X = \scr X^{\mathrm{tw}}_1(N)$ or $\scr X = \scr X^{\mathrm{tw}}(N)$, the closed substack $\scr X^\infty$ is the pullback $\scr X \times_{\overline{\cM}_{1,1}} \infty$. We know $\infty \subset \overline{\cM}_{1,1}$ is a relative effective Cartier divisor over $\mathrm{Spec}(\mathbb{Z})$ (meaning an effective Cartier divisor which is flat over $\mathrm{Spec}(\mathbb{Z})$), and by Theorems \ref{gamma1} and \ref{gamma} the morphism $\scr X \rightarrow \overline{\cM}_{1,1}$ is flat. Cartier divisors are preserved by flat morphisms (cf. \cite[\S1.7]{F}), so $\scr X^\infty$ is an effective Cartier divisor in $\scr X$. Since $\scr X \rightarrow \overline{\cM}_{1,1}$ is flat, so is $\scr X^\infty \rightarrow \infty$, so $\scr X^\infty$ is flat over $\mathrm{Spec}(\mathbb{Z})$, i.e. $\scr X^\infty$ is a relative effective Cartier divisor in $\scr X$ over $\mathrm{Spec}(\mathbb{Z})$.
\end{proof}

\begin{corollary}
$\scr X_1^{\mathrm{tw}}(N)$ and $\scr X^{\mathrm{tw}}(N)$ are normal.
\end{corollary}
\begin{proof}
This is proven in an identical manner to \cite[4.1.4]{C}. The stacks $\scr X_1^{\mathrm{tw}}(N)$ and $\scr X^{\mathrm{tw}}(N)$ are Deligne-Mumford, and from \cite[3.0.2]{ACV} we know $\scr X_1^{\mathrm{tw}}(N) \otimes_\mathbb{Z} \mathbb{Z}[1/N]$ and $\scr X^{\mathrm{tw}}(N) \otimes_\mathbb{Z} \mathbb{Z}[1/N]$ are smooth over $\mathrm{Spec}(\mathbb{Z}[1/N])$. In particular, they are regular at any characteristic $0$ points. Furthermore, by Proposition \ref{CM} they are Cohen-Macaulay over $\mathrm{Spec}(\mathbb{Z})$ of pure relative dimension $1$. As in \cite[4.1.4]{C}, we can conclude from Serre's criterion for normality that it suffices to prove that these stacks are regular away from some relative effective Cartier divisor, since such a divisor cannot contain any codimension $1$ points of positive residue characteristic. Use the divisors $\scr X_1^{\mathrm{tw}}(N)^{\infty}$ and $\scr X^{\mathrm{tw}}(N)^{\infty}$; their complements are $\scr Y_1(N)$ and $\scr Y(N)$, which are regular by \cite[5.1.1]{KM1}.
\end{proof}

\begin{proof}[Proof of \ref{classical}]
$\scr X_1^{\mathrm{tw}}(N)$ and $\scr X^{\mathrm{tw}}(N)$ are finite, flat and normal over $\overline{\cM}_{1,1}$, so they are naturally identified with the normalizations (in the sense of \cite[IV.3.3]{DR}) of $\overline{\cM}_{1,1}$ in $\scr X_1^{\mathrm{tw}}(N)|_{\cM_{1,1}} = \scr Y_1(N)$ and $\scr X^{\mathrm{tw}}(N)|_{\cM_{1,1}} = \scr Y(N)$ respectively; cf. \cite[4.1.5]{C}.
\end{proof}

We now give a moduli interpretation of the equivalence $\scr X_1(N) \simeq \scr X^{\mathrm{tw}}_1(N)$. Let $S$ be a scheme, $E/S$ a generalized elliptic curve, and $P \in E^{\mathrm{sm}}(S)[N]$ a $[\Gamma_1(N)]$-structure on $E$. From this data we want to construct a pair $(\cC_P,\phi_P)$, where $\cC_P/S$ is a $1$-marked genus $1$ twisted stable curve with non-stacky marking, and $\phi_P: \mathbb{Z}/(N) \rightarrow \mathrm{Pic}^0_{\cC_P/S}$ is a $[\Gamma_1(N)]$-structure on $\cC_P$.

If $E/S$ is a smooth elliptic curve, there is nothing to show: we simply take $(\cC_P, \phi_P) = (E,\phi_P)$ where $\phi_P: \mathbb{Z}/(N) \rightarrow \mathrm{Pic}^0_{E/S} \cong E$ sends $1 \mapsto P$. Therefore to construct $(\cC_P, \phi_P)$ in general, we may restrict to the open subscheme of $S$ where $E/S$ has no supersingular geometric fibers; once we have constructed $(\cC_P, \phi_P)$ in this case, we only need to check that it agrees with our previous construction for ordinary elliptic curves.

For the rest of the construction, assume that $E/S$ has no supersingular geometric fibers.

Note that by \cite[4.2.3]{C}, fppf-locally on $S$ there exists a generalized elliptic curve $E'/S$, whose singular geometric fibers are $N$-gons, together with an open $S$-immersion $\iota: E^{\mathrm{sm}} \hookrightarrow E'^{\mathrm{sm}}$ of group schemes over $S$. In particular, by \cite[II.1.20]{DR} the group scheme $E'^{\mathrm{sm}}[N]/S$ is finite and flat of constant rank $N^2$.

Since $E'/S$ has no supersingular geometric fibers and all its singular geometric fibers are $N$-gons, it follows that fppf-locally on $S$, there exists a $[\Gamma(N)]$-structure $(Q,R)$ on $E'$ such that: 
\begin{itemize}
  \item the relative effective Cartier divisor 
  \begin{displaymath}
  D := \sum_{a \in \mathbb{Z}/(N)} [a\cdot Q]
  \end{displaymath}
in $E'^{\mathrm{sm}}$ is \'etale over $S$, and 
   \item $R$ meets the identity component of every geometric fiber of $E'/S$.
\end{itemize}

The choice of $Q$ and the pairing 
\begin{displaymath}
e_N: E'^{\mathrm{sm}}[N] \times E'^{\mathrm{sm}}[N] \rightarrow \mu_N
\end{displaymath}
induce a canonical isomorphism $E'^{\mathrm{sm}}[N]/D \cong E'^{\mathrm{sm}}[N]/D \times \{Q\} \cong \mu_N$. Identifying $E'^{\mathrm{sm}}[N]/D$ with its image in the $N$-torsion of the generalized elliptic curve $C := E'/D$ (a generalized elliptic curve whose singular fibers are $1$-gons), the group law of $C$ and the above isomorphism give us an action of $\mu_N$ on $C$, making $C$ a $\mu_N$-torsor over the twisted curve $\cC := [C/\mu_N] = [E'/E'^{\mathrm{sm}}[N]] = [E/E^{\mathrm{sm}}[N]]$. Write 
\begin{displaymath}
C = \underline{\mathrm{Spec}}_{\cC} \big( \bigoplus_{a \in \mathbb{Z}/(N)} \cG_a \big) \stackrel{\pi}{\longrightarrow} \cC,
\end{displaymath}
where each $\cG_a$ is an invertible $\cO_{\cC}$-module, with the grading and algebra structure corresponding to the structure of $C$ as a $\mu_N$-torsor over $\cC$.

The image $\overline{R}$ of $R$ in $C$ is a $[\Gamma_1(N)]$-structure on $C$, so we get a $\mu_N$-torsor 
\begin{displaymath}
T := \underline{\mathrm{Spec}}_C \big( \bigoplus_{b \in \mathbb{Z}/(N)} \cL((b\cdot \overline{R}) - (0_C)) \big) 
\end{displaymath}
over $C$; the $\mu_N$-action on $T$ corresponds to the $\mathbb{Z}/(N)$-grading and the algebra structure on 
\begin{displaymath}
\bigoplus_{b \in \mathbb{Z}/(N)} \cL((b\cdot \overline{R})-(0_C)) 
\end{displaymath}
comes from the group law on $C^{\mathrm{sm}}$ and the canonical isomorphism $C^{\mathrm{sm}} \cong \mathrm{Pic}^0_{C/S}$.

Since $C$ is a $\mu_N$-torsor over $\cC$, if $\cL \in \mathrm{Pic}(C)$ we have a canonical decomposition 
\begin{displaymath}
\pi_* \cL = \bigoplus_{a \in \mathbb{Z}/(N) } \cL_a, 
\end{displaymath}
where each $\cL_a$ is an invertible sheaf on $\cC$ and $\zeta \in \mu_N$ acts on $\cL_a$ via multiplication by $\zeta^a$. In particular this applies to the invertible sheaf $\cL = \cL((b\cdot \overline{R}) - (0_C))$, giving us a canonical decomposition 
\begin{displaymath}
\pi_* \cL((b\cdot \overline{R}) - (0_C)) = \bigoplus_{a \in \mathbb{Z}/(N)} \cL_{(a,b)}.
\end{displaymath}

We have $\cL_{0,0} = \cG_0 = \cO_{\cC}$, and the isomorphisms 
\begin{displaymath}
\cL((b_0 \cdot \overline{R}) - (0_C)) \otimes_{\cO_C} \cL((b_1 \cdot \overline{R}) - (0_C)) \cong \cL(((b_0 + b_1) \cdot \overline{R}) - (0_C))
\end{displaymath}
(coming from the algebra structure of $\oplus_b \cL((b\cdot \overline{R}) - (0_C))$) induce isomorphisms 
\begin{displaymath}
\cL_{(a_0,b_0)} \otimes_{\cO_{\cC}} \cL_{(a_1,b_1)} \cong \cL_{(a_0+a_1,b_0+b_1)}
\end{displaymath}
for all $(a_0,b_0),(a_1,b_1) \in (\mathbb{Z}/(N))^2$, giving us a canonical algebra structure on the direct sum 
\begin{displaymath}
\bigoplus_{(a,b) \in (\mathbb{Z}/(N))^2} \cL_{(a,b)}.
\end{displaymath}

Identifying our original $[\Gamma_1(N)]$-structure $P$ with its image in $E'^{\mathrm{sm}}(S)[N]$, there exists some $(a_0,b_0) \in (\mathbb{Z}/(N))^2$ with $P = a_0\cdot Q + b_0\cdot R \in E'^{\mathrm{sm}}[N]$. This determines a $\mu_N$-torsor 
\begin{displaymath}
\cT := \underline{\mathrm{Spec}}_{\cC} \big( \bigoplus_{c \in \mathbb{Z}/(N)} \cL_{(ca_0,cb_0)} \big) 
\end{displaymath}
over $\cC$, corresponding to a morphism $\cC \rightarrow \cB \mu_N$. Here $\oplus \cL_{(ca_0,cb_0)}$ is viewed as a sub-$\cO_{\cC}$-algebra of the algebra $\oplus \cL_{(a,b)}$.
\begin{definition}
We define $\cC_P \rightarrow \cB\mu_N$ to be the relative coarse moduli space of the above morphism $\cC \rightarrow \cB \mu_N$, and we write $\phi_P: \mathbb{Z}/(N) \rightarrow \mathrm{Pic}^0_{\cC_P/S}$ for the corresponding group scheme homomorphism.
\end{definition}

It is immediate that $\phi_P$ is a $[\Gamma_1(N)]$-structure on the twisted curve $\cC_P$. 

\begin{lemma}
$(\cC_P,\phi_P)$ is independent of the choice of $(a_0,b_0)$ with $P = a_0\cdot Q + b_0 \cdot R$, and of the choice of generalized elliptic curve $E'$ and $[\Gamma(N)]$-structure $(Q,R)$ on $E'$ such that $D = \sum [a\cdot Q]$ is \'etale over $S$ and $R$ meets the fiberwise identity components of $E'/S$.
\end{lemma}
\begin{proof} 
First of all, if $E/S$ is an ordinary elliptic curve, then $E' = E$. Our construction defines a map $E[N] \rightarrow \mathrm{Pic}^0_{E/S}[N] \cong E[N]$, which in fact is simply the identity map (which in particular is independent of the choice of $(a_0,b_0)$ and the $[\Gamma(N)]$-structure $(Q,R)$). To see this, recall that by Lemma \ref{quotient} we may identify the generalized elliptic curve $C = E/\langle Q \rangle$ (viewed as a $\mu_N$-torsor over $E \cong [E/E[N]]$ as discussed above) with the $\mu_N$-torsor 
\begin{displaymath}
\underline{\mathrm{Spec}}_E \big( \bigoplus_{a \in \mathbb{Z}/(N)} \cL((a\cdot Q) - (0_E)) \big) 
\end{displaymath}
over $E$. So in the notation of the above construction, $\cC = E$ and $\cG_a = \cL_{(a,0)} = \cL((a\cdot Q) - (0_E))$. We have 
\begin{align*}
\pi_*\cL((\overline{R}) - (0_C)) & \cong \bigoplus_{a \in \mathbb{Z}/(N)} \big( \cG_a \otimes \cL((R) - (0_E)) \big) \\
&\cong \bigoplus_{a \in \mathbb{Z}/(N)} \cL((a\cdot Q + R) - (0_E)).
\end{align*}
So the map defined in the above construction sends $Q$ to $\cL_{(1,0)} = \cL((Q) - (0_E))$ and $R$ to $\cL_{(0,1)} = \cL((R) - (0_E))$. Composing with the usual isomorphism $\mathrm{Pic}^0_{E/S} \cong E$ yields the identity map on $E[N]$.

To complete the proof in the case where $E/S$ is not necessarily smooth, it suffices to consider the case where $S$ is the spectrum of an algebraically closed field $k$ and $E/k$ is a N\'eron $d$-gon (for some $d|N$). $E'/k$ is then a N\'eron $N$-gon, and our Drinfeld basis $(Q,R)$ was chosen so that $\langle Q \rangle$ meets every irreducible component of $E'$ and $\langle R \rangle$ lies on the identity component. We may therefore choose an isomorphism $E'^{\mathrm{sm}} \cong \mathbb{Z}/(N) \times \mu_N$ such that $Q = (1,1)$ and $R = (0,\zeta)$ for some $\zeta \in \mu_N^\times (k)$. 
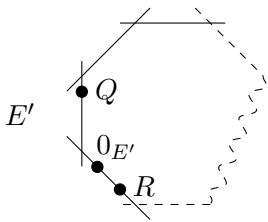
\begin{center}
\begin{tikzpicture}
\clip (-1.8,-0.2) rectangle (2.8,2.7);
\draw (0.2,-0.2) -- (-0.907,0.907);
\draw (-0.707,0.507) -- (-0.707,1.907);
\draw (-0.907,1.507) -- (0.2,2.614);
\draw (-0.2,2.414) -- (1.2,2.414);
\begin{scope}[dashed]
\draw (0.8,2.614) -- (1.707,1.707);
\draw (1,0) -- (-0.2,0);
\draw [decorate,decoration=snake] (1.707,1.707) -- (1,0);
\end{scope}
\node at (-0.5,0.5) [circle,draw,fill,inner sep=1.5pt] {};
\node at (-0.67,0.33) [label=80:$0_{E'}$] {};
\node at (-0.2,0.2) [circle,draw,fill,inner sep=1.5pt] {};
\node at (-0.3,0.25) [label=0:$R$] {};
\node at (-0.707,1.5) [circle,draw,fill,inner sep=1.5pt] {};
\node at (-0.8,1.5) [label=0:$Q$] {};
\node at (-2,1.2) [label=0:$E'$] {};
\end{tikzpicture}
\end{center}
Then if $P = a_0 \cdot Q + b_0 \cdot R = a_1 \cdot Q + b_1 \cdot R$, it follows that $a_0 = a_1$ and $\zeta^{b_0} = \zeta^{b_1}$, the latter of which implies that $b_0 \cdot \overline{R} = b_1 \cdot \overline{R}$. Thus $\cL((b_0 \cdot \overline{R}) - (0_C)) = \cL((b_1 \cdot \overline{R}) - (0_C))$, so $\cL_{(a_0,b_0)} = \cL_{(a_1,b_1)}$, hence $(\cC_P,\phi_P)$ is independent of the choice of $(a_0,b_0)$.

Now we must check that $(\cC_P,\phi_P)$ is independent of the choice of $(E',(Q,R))$. $E'/k$ is a N\'eron $N$-gon, so the choice of $E'$ is unique up to composition with an automorphism of $E'$ fixing $E^{\mathrm{sm}} \subset E'^{\mathrm{sm}}$. $E'$ is the special fiber of an $N$-gon Tate curve $\cE' / k[\![q^{1/N}]\!]$. Let $\cC = [\cE'/\cE'^{\mathrm{sm}}[N]]$, so $\cC_k = [E'/E'^{\mathrm{sm}}[N]] = [E/E^{\mathrm{sm}}[N]]$. 
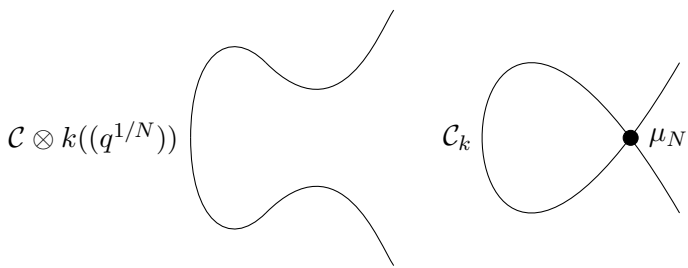
\begin{center}
\begin{tikzpicture}
\clip (-4,-0.8) rectangle (6,2.8);
\draw (1.2,2.7) .. controls (1,2.4) and (0.5,1) .. (-0.5,2);
\draw (-0.5,2) .. controls (-1,2.5) and (-1.5,2.1) .. (-1.5,1);
\draw (1.2,-0.7) .. controls (1,-0.4) and (0.5,1) .. (-0.5,0);
\draw (-0.5,0) .. controls (-1,-0.5) and (-1.5,-0.1) .. (-1.5,1);
\node at (-1.5,1) [inner sep=0pt,label=180: $\cC \otimes k((q^{1/N}))$] {};
\draw [name path=curve] (5,2) .. controls (1.5,-4) and (1.5,6) .. (5,0);
\node at (2.4,1) [inner sep=0pt,label=180:$\cC_k$] {};
\node at (4.35,1) [circle,draw,fill,inner sep=2pt,label=0:$\mu_N$] {};
\end{tikzpicture}
\end{center}
We may choose an isomorphism $\cE'^{\mathrm{sm}}[N] \cong \mathbb{Z}/(N) \times \mu_N$ of finite flat group schemes over $k[\![q^{1/N}]\!]$. $(Q,R)$ extends to a $[\Gamma(N)]$-structure $(\cQ,\cR)$ on $\cE'$ with $\cD = \sum (a\cdot \cQ)$ \'etale over $k[\![q^{1/N}]\!]$ (and of course $\cR$ meets the identity component of every geometric fiber of $\cE'/ k[\![q^{1/N}]\!]$). Given such a $[\Gamma(N)]$-structure on $\cE'$, our construction defines a group scheme homomorphism 
\begin{displaymath}
\cE'^{\mathrm{sm}}[N] \rightarrow \mathrm{Pic}^0_{\cC/k[\![q^{1/N}]\!]}[N]. 
\end{displaymath}
Both of these are finite flat group schemes over $k[\![q^{1/N}]\!]$ which are isomorphic to $\mathbb{Z}/(N) \times \mu_N$, and $\underline{\mathrm{End}}(\mathbb{Z}/(N) \times \mu_N)$ is finite (hence proper) over $k[\![q^{1/N}]\!]$. Since 
\begin{displaymath}
\cE'^{\mathrm{sm}}[N] \otimes k((q^{1/N})) \rightarrow \mathrm{Pic}^0_{\cC/k[\![q^{1/N}]\!]}[N] \otimes k((q^{1/N})) 
\end{displaymath}
is independent of the choice of $[\Gamma(N)]$-structure over $k((q^{1/N}))$ (as $\cE' \otimes k((q^{1/N}))$ is an elliptic curve), we conclude that $\cE'^{\mathrm{sm}}[N] \rightarrow \mathrm{Pic}^0_{\cC/k[\![q^{1/N}]\!]}[N]$ is independent of the choice of $(\cQ,\cR)$. Thus in particular $E'^{\mathrm{sm}}[N] \rightarrow \mathrm{Pic}^0_{\cC_k/k}[N]$ is independent of the choice of $(Q,R)$ and the resulting homomorphism $E^{\mathrm{sm}}[N] \rightarrow \mathrm{Pic}^0_{\cC_k/k}[N]$ is independent of the choice of $E'$.
\end{proof}

Thus by descent, $(\cC_P, \phi_P) \in \scr X^{\mathrm{tw}}_1(N)(S)$ is well-defined globally over our initial base scheme $S$ (even allowing supersingular fibers) and depends only on the pair $(E,P) \in \scr X_1(N)(S)$. We define our map $\scr X_1(N) \rightarrow \scr X^{\mathrm{tw}}_1(N)$ by sending $(E,P)$ to $(\cC_P, \phi_P)$. 
\begin{corollary}\label{modular1}
Over any base scheme $S$, the morphism $\scr X_1(N) \rightarrow \scr X^{\mathrm{tw}}_1(N)$ sending $(E,P)$ to $(\cC_P, \phi_P)$ is an isomorphism of algebraic stacks.
\end{corollary}

Similarly, given a generalized elliptic curve $E/S$ equipped with a $[\Gamma(N)]$-structure $(P_1,P_2)$, the above procedure produces a $[\Gamma(N)]$-structure $\phi_{(P_1,P_2)}$ on the twisted curve $\cC_E := [E/E^{\mathrm{sm}}[N]]$.

\begin{corollary}\label{modular2}
Over any base scheme $S$, the morphism $\scr X(N) \rightarrow \scr X^{\mathrm{tw}}(N)$ sending $(E,(P_1,P_2))$ to $(\cC_E, \phi_{(P_1,P_2)})$ is an isomorphism of algebraic stacks.
\end{corollary}

\section{Other compactified moduli stacks of elliptic curves}
It is worth noting that the techniques in the proof of \ref{gamma1} are easily adapted to prove properness of the natural analogues in our current setting of well-known modular compactifications of other various moduli stacks of elliptic curves with extra structure, even when these moduli stacks do not naturally lie in a moduli stack of twisted stable maps:

\begin{definition}
Let $\cC/S$ be a $1$-marked genus $1$ twisted stable curve with non-stacky marking.

(i) A \textit{$[\Gamma_0(N)]$-structure} on $\cC$ is a finite locally free $S$-subgroup scheme $G$ of $\mathrm{Pic}^0_{\cC/S}$ of rank $N$ over $S$ which is cyclic (fppf-locally admits a $\mathbb{Z}/(N)$-generator), such that for every geometric point $\overline{p} \rightarrow S$, $G_{\overline{p}}$ meets every irreducible component of $\mathrm{Pic}^0_{\cC_{\overline{p}}/k(\overline{p})}$. We write $\scr X^{\mathrm{tw}}_0(N)$ for the stack over $S$ associating to $T/S$ the groupoid of pairs $(\cC,G)$, where $\cC/T$ is a $1$-marked genus $1$ twisted stable curve with non-stacky marking, and $G$ is a $[\Gamma_0(N)]$-structure on $\cC$.

(ii) A \textit{balanced $[\Gamma_1(N)]$-structure} (cf. \cite[\S3.3]{KM1}) on $\cC$ is an fppf short exact sequence of commutative group schemes over $S$ 
\begin{displaymath}
\:\:\:\:\:\:\:0 \rightarrow K \rightarrow \mathrm{Pic}^0_{\cC/S}[N] \rightarrow K' \rightarrow 0,\:\:\:\:\:\:\: (\dagger)
\end{displaymath}
where $K$ and $K'$ are locally free of rank $N$ over $S$, together with sections $P \in K(S)$ and $P' \in K'(S)$ which are $\mathbb{Z}/(N)$-generators of $K$ and $K'$ in the sense of \cite[\S1.4]{KM1}. We write $\scr X^{\mathrm{bal,tw}}_1(N)$ for the stack over $S$ associating to $T/S$ the groupoid of pairs $(\cC,\dagger)$, where $\cC/T$ is a $1$-marked genus $1$ twisted stable curve with non-stacky marking, and $\dagger$ is a balanced $[\Gamma_1(N)]$-structure on $\cC$.

(iii) An $[N$-Isog$]$\textit{-structure} (cf. \cite[\S6.5]{KM1}) on $\cC$ is a finite locally free commutative $S$-subgroup scheme $G \subset \mathrm{Pic}^0_{\cC/S}[N]$ of rank $N$ over $S$, such that for every geometric point $\overline{p} \rightarrow S$, $G_{\overline{s}}$ meets every irreducible component of $\mathrm{Pic}^0_{\cC_{\overline{p}}/k(\overline{p})}$. We write $\scr X^{\mathrm{tw}}(N\textrm{-Isog})$ for the stack over $S$ associating to $T/S$ the groupoid of pairs $(\cC,G)$, where $\cC/T$ is a $1$-marked genus $1$ twisted stable curve with non-stacky marking, and $G$ is an $[N$-Isog$]$-structure on $\cC$.

(iv) If $N$ and $n$ are positive integers such that $\mathrm{ord}_p(n) \leq \mathrm{ord}_p(N)$ for all primes $p$ dividing both $N$ and $n$, a \textit{$[\Gamma_1(N;n)]$-structure} (cf. \cite[2.4.3]{C}) on $\cC$ is a pair $(\phi,G)$, where: 
\begin{itemize}
  \item $\phi: \mathbb{Z}/(N) \rightarrow \mathrm{Pic}^0_{\cC/S}$ is a $\mathbb{Z}/(N)$-structure in the sense of \cite[\S1.5]{KM1}; 
  \item $G \subset \mathrm{Pic}^0_{\cC/S}$ is a finite locally free $S$-subgroup scheme which is cyclic of order $n$; 
  \item the degree $Nn$ relative effective Cartier divisor 
  \begin{displaymath}
  \sum_{a = 0}^{N-1} (\phi(a) + G)
  \end{displaymath}
  in $\mathrm{Pic}^0_{\cC/S}$ meets every irreducible component of each geometric fiber of $\mathrm{Pic}^0_{\cC/S}$ over $S$; 
  \item for all primes $p$ dividing both $N$ and $n$, for $e = \mathrm{ord}_p(n)$ we have an equality of closed subschemes of $\mathrm{Pic}^0_{\cC/S}$ 
  \begin{displaymath}
  \sum_{a = 0}^{p^e-1} ((N/p^e) \cdot \phi(a) + G[p^e]) = \mathrm{Pic}^0_{\cC/S}[p^e].
  \end{displaymath}
\end{itemize}
We write $\scr X^{\mathrm{tw}}_1(N;n)$ for the stack over $S$ associating to $T/S$ the groupoid of tuples $(\cC,(\phi,G))$, where $\cC/T$ is a $1$-marked genus $1$ twisted stable curve with non-stacky marking, and $(\phi,G)$ is a $[\Gamma_1(N;n)]$-structure on $\cC$.
\end{definition}

\begin{corollary}
The stacks $\scr X^{\mathrm{tw}}_0(N)$, $\scr X^{\mathrm{bal,tw}}_1(N)$, $\scr X^{\mathrm{tw}}(N\textrm{-}\mathrm{Isog})$, and $\scr X^{\mathrm{tw}}_1(N;n)$ are algebraic stacks which are flat and locally finitely presented over $S$, with local complete intersection fibers. They are proper and quasi-finite over $\overline{\cM}_{1,1}$, and each is isomorphic to the corresponding moduli stack for generalized elliptic curves.
\end{corollary}

As shown explicitly for the stacks $\scr X^{\mathrm{tw}}_1(N)$ and $\scr X^{\mathrm{tw}}(N)$ earlier in this paper, one may study the reductions modulo $p$ of these moduli stacks, and one finds that over a perfect field of characteristic $p$, each stack is a disjoint union with crossings at the supersingular points of various closed substacks, which come naturally equipped with moduli interpretations extending the definitions given in \cite{KM1} (except for $\scr X^{\mathrm{tw}}_1(N;n)$, which is not studied in \cite{KM1}).

\begin{question}
If $\cC/S$ is a $1$-marked genus $1$ twisted stable curve with non-stacky marking, the group scheme $\mathrm{Pic}^0_{\cC/S}$ behaves just like the smooth part of a generalized elliptic curve. Over the Zariski open set of $S$ where $\cC \rightarrow S$ is smooth, it agrees with $C$ (which is in this case a smooth elliptic curve); if $\overline{p} \rightarrow S$ is a geometric point such that $C_{\overline{p}}$ is singular, then $\mathrm{Pic}^0_{\cC_{\overline{p}}/k(\overline{p})} \cong \mathbb{G}_m \times \mathbb{Z}/(N)$ for some $N$, and this is just the smooth part of a N\'eron $N$-gon over $k(\overline{p})$. Is there a natural way to exhibit $\mathrm{Pic}^0_{\cC/S}$ as the smooth part of a generalized elliptic curve, giving an equivalence between the stack of generalized ellipic curves over $S$ and the stack of $1$-marked genus $1$ twisted stable curves over $S$ with non-stacky marking? More precisely, is there a natural way of defining compactified Jacobians of twisted curves, such that the degree zero compactified Jacobian of a standard $\mu_d$-stacky N\'eron $1$-gon is a N\'eron $d$-gon?
\end{question}

\begin{appendix}
\section{A note on moduli of curves of higher genus}

As in the case of elliptic curves, stacks of twisted stable maps allow for natural compactifications of stacks of genus $g$ curves equipped with certain extra structure. For example, 
\begin{displaymath}
\overline{\cK}^\circ_{g,0}(\cB\mu_N) := \overline{\cK}_{g,0} (\cB\mu_N) \times_{\overline{\cM}_g} \cM_g
\end{displaymath}
classifies pairs $(C/S,\phi)$ where $C/S$ is a smooth genus $g$ curve and $\phi: \mathbb{Z}/(N) \rightarrow \mathrm{Pic}^0_{C/S}$, which we view as an $N$-torsion point in $\mathrm{Pic}^0_{C/S} = \mathrm{Jac}(C/S)$. This stack is naturally contained in the proper algebraic stack $\overline{\cK}_{g,0}(\cB\mu_N)$ as an open dense substack. However, the situation becomes considerably more complicated when we attempt to use this to obtain proper moduli stacks of curves with level structure, e.g. replacing ``$N$-torsion points'' with ``points of exact order $N$.''

Over $\mathbb{Z}[1/N]$, we have a stack $\overline{\cM}^{(N)}_g$ of twisted curves with level $N$ structure, studied in \cite[\S6]{ACV}; this is a smooth proper modular compactification of the stack classifying (not necessarily symplectic) Jacobi level $N$ structures on smooth genus $g$ curves, as in \cite[5.4]{DM}. One may be tempted to proceed as follows:
\begin{definition}
Let $\cC/S$ be an unmarked genus $g$ ($g>1$) twisted stable curve over a scheme $S$. A \textit{full level $N$ structure} on $\cC$ is a group scheme homomorphism $\phi: (\mathbb{Z}/(N))^{2g} \rightarrow \mathrm{Pic}^0_{\cC/S}$ such that $\{\phi(a)|a\in (\mathbb{Z}/(N))^2\}$ is a full set of sections for the finite flat group scheme $\mathrm{Pic}^0_{\cC/S}[N]$ over $S$ in the sense of \cite[\S1.8]{KM1}.

We write $\overline{\cM}^{(N),\mathrm{tw}}_g$ for the substack of $\overline{\cK}_{g,0}(\cB\mu_N^{2g})$ associating to $T/S$ the groupoid of pairs $(\cC,\phi)$, where $\cC/T$ is an unmarked genus $g$ twisted stable curve and $\phi$ is a full level $N$ structure on $\cC$.
\end{definition}

Unfortunately, this is not the ``right'' definition. By this we mean that we would like the stack $\overline{\cM}^{(N),\mathrm{tw}}_g$ to be a closed substack of $\overline{\cK}_{g,0}(\cB \mu_N^{2g})$, flat over $S$; but it follows immediately from the example \cite[appendix]{CN} that flatness of $\overline{\cM}^{(N),\mathrm{tw}}_g$ fails in mixed characteristic, even over the ordinary locus of $\cM_g$. Of course, if $N$ is invertible on $S$ then this definition is the correct one. More precisely, the choice of an isomorphism $(\mathbb{Z}/(N))^{2g} \cong \mu_N^{2g}$ identifies $\overline{\cM}^{(N),\mathrm{tw}}_g$ with the stack $\overline{\cM}^{(N)}_g$ of \cite[\S6]{ACV}, which is shown in loc. cit. to be smooth over $\mathbb{Z}[1/N]$ and proper over $\overline{\cM}_g$.  One would hope to be able to change the above definition to get a closed substack $\overline{\cM}^{(N),\mathrm{tw}}_g$ of $\overline{\cK}_{g,0}(\cB \mu_N^{2g})$, flat over $S$, agreeing with $\overline{\cM}^{(N)}_g$ over $S[1/N]$ and with a natural moduli interpretation in terms of the maps from $(\mathbb{Z}/(N))^{2g}$ to the group schemes $\mathrm{Pic}^0_{\cC/S}$. 

More generally we have good properties for the moduli stack $\overline{\cK}_{g,0}(\cB G)$ whenever $G$ is a finite diagonalizable or locally diagonalizable group scheme over $S$, so the Cartier dual $G^*$ is commutative and constant or locally constant. Recall that for a finite group $G$ there is a stack $_G\cM_g$ over $\mathbb{Z}[1/|G|]$ of Teichm\"uller structures of level $G$ on smooth curves, (cf. \cite[5.6]{DM} and \cite{PJ}). Now if $G$ is a diagonalizable group scheme with $|G|$ invertible on the base scheme $S$, then after adjoining appropriate roots of unity we may fix an isomorphism $G \cong G^*$. In \cite[5.2.3]{ACV} it is observed that this determines an isomorphism between $_{G^*}\cM_g$ and a substack $\scr B^{\mathrm{tei}}_g(G)^\circ$ of $\overline{\cK}^\circ_{g,0}(\cB G)$ whose closure $\scr B^{\mathrm{tei}}_g(G)$ in $\overline{\cK}_{g,0}(\cB G)$ is a moduli stack whose geometric objects correspond precisely to $G$-torsors $P \rightarrow \cC$ which are connected (where $\cC$ is a genus $g$ twisted stable curve); these are called \textit{twisted Teichm\"uller $G$-structures}. One would hope that $\scr B^{\mathrm{tei}}_g(G)$ can be defined in arbitrary characteristic, with a natural moduli interpretation, but it is not clear to the author how to proceed with this for genus $g>1$; as discussed above, it does not suffice to simply consider the substack of $\overline{\cK}_{g,0}(\cB G)$ whose geometric objects correspond to $G$-torsors which are connected, since $\mu_{p^n}$ is connected in characteristic $p$, and the definition in terms of ``full sets of sections'' does not give a stack flat over the base scheme in mixed characteristic.

\end{appendix}

\end{document}